\documentclass[reqno, 11pt, a4paper]{amsart}

\usepackage[truedimen,top=3truecm, bottom=3truecm, left=2.5truecm, right=2.5truecm, includefoot]{geometry}
\usepackage{fancyhdr}

\usepackage{amsmath, amssymb, amsthm, mathtools} 
\usepackage{mathrsfs}
\usepackage{enumerate}

\usepackage{graphicx}
\usepackage{float}
\usepackage{xcolor} 
\usepackage[percent]{overpic}
\usepackage{wrapfig} 

\definecolor{cyan(process)}{rgb}{0.0, 0.72, 0.92}
\definecolor{columbiablue}{rgb}{0.61, 0.87, 1.0}
\definecolor{sandstone}{HTML}{786D5F}
\definecolor{beaublue}{rgb}{0.74, 0.83, 0.9}
\definecolor{cherryblossompink}{rgb}{1.0, 0.72, 0.77}
\definecolor{light-gray}{gray}{0.95}

\usepackage{tikz}
\usepackage{pgfplots}
\pgfplotsset{compat=1.16}
\usepackage{pstricks,pst-node}
\usetikzlibrary{shapes.misc, arrows, decorations.pathmorphing, backgrounds, positioning, fit, petri, shapes, cd, arrows.meta, patterns, fadings, calc}

\usepackage{natbib}
\numberwithin{equation}{section}

\usepackage[active]{srcltx} 
\usepackage{comment}
\usepackage{todonotes}
\usepackage[normalem]{ulem} 

\usepackage[colorlinks=true, citecolor=blue, urlcolor=blue, pagebackref=true]{hyperref}
\usepackage{cleveref}
\usepackage{autonum} 

\renewcommand*{\backref}[1]{}
\renewcommand*{\backrefalt}[4]{%
    \ifcase #1 %
    \or
        \textcolor{red}{\textsuperscript{#2}}%
    \else
        \textcolor{red}{\textsuperscript{#2}}%
    \fi
}

\theoremstyle{plain}
\newtheorem{theorem}{Theorem}[section] 
\newtheorem{lemma}[theorem]{Lemma}

\newtheorem{proposition}[theorem]{Proposition}
\theoremstyle{definition}
\newtheorem{definition}[theorem]{Definition}
\newtheorem{remark}[theorem]{Remark}

\makeatletter

\@addtoreset{equation}{section}
\makeatother

\newcommand\EE{{\mathbb E}}
\newcommand\PP{{\mathbb P}}

\newcommand\RR{{\mathbb R}}

\newcommand{\vertiii}[1]{{\left\vert\kern-0.25ex\left\vert\kern-0.25ex\left\vert #1 \right\vert\kern-0.25ex\right\vert\kern-0.25ex\right\vert}}

\newcommand{\ced}[1]{\todo[author=\textcolor{white}{\bf{Cedric}},inline]{\textcolor{white}{{\textbf{#1}}}}}

\title[fMFT for superdiffusive GL dynamics]{Fractional Macroscopic Fluctuation Theory for a superdiffusive Ginzburg-Landau dynamics}

\begin{document}

\makeatletter

\newcommand{\acks}{\textbf{Acknowledgements.}}
\newenvironment{acknowledgements}{%
  \renewcommand{\abstractname}{Acknowledgements}
  \begin{abstract}
}{%
  \end{abstract}
}

\allowdisplaybreaks


\author[C. Bernardin]{C\'edric Bernardin}
\address{Faculty of Mathematics, National Research University Higher School of Economics, 6 Usacheva, Moscow, 119048 , Russia.}
\email{sedric.bernardin@gmail.com}

\author[P. Gon\c{c}alves]{Patr\'{i}cia Gon\c{c}alves}
\address{Center for Mathematical Analysis, Geometry and Dynamical Systems, Instituto Superior T\'ecnico, Universidade de Lisboa, Av. Rovisco Pais, 1049-001 Lisboa, Portugal.}
\email{pgoncalves@tecnico.ulisboa.pt}
\author[J. Mangi]{João Pedro Mangi}
\address{Instituto de matem\'atica pura e aplicada, Estrada Dona Castorina 110, J. Botanico, 22460 Rio de Janeiro, Brazil and Center for Mathematical Analysis, Geometry and Dynamical Systems, Instituto Superior T\'ecnico, Universidade de Lisboa, Av. Rovisco Pais, 1049-001 Lisboa, Portugal.} 
\email{mangi.joao@impa.br}

\begin{abstract} 
We investigate a boundary-driven Ginzburg-Landau dynamics with long-range interactions. In the hydrodynamic limit, the macroscopic evolution is governed by a fractional heat equation with Dirichlet boundary conditions, while the corresponding stationary profile is characterized by a fractional Laplace equation. We establish a dynamical large deviations principle for the empirical measure and derive the associated stationary large deviations principle for the non-equilibrium steady state, which can be computed semi-explicitly. We further show that the stationary rate function coincides with the quasi-potential associated with the dynamical large deviations functional.
\end{abstract}

\maketitle

\section{Introduction}

The 1990s witnessed the spectacular development of deep mathematical techniques for proving hydrodynamic limits of interacting particle systems (IPS). In the seminal paper \cite{GPV88}, Guo, Papanicolaou, and Varadhan derived the hydrodynamic limit of the empirical density for a nearest-neighbor interacting system known as {the} Ginzburg-Landau (GL) dynamics, showing that the hydrodynamic equation is a nonlinear diffusion equation. GL dynamics is a stochastic lattice field model, which can be seen as a fluctuating interface preserving the algebraic volume of the latter. In doing so, they introduced the \textit{entropy method} —a robust approach for deriving hydrodynamic limits for diffusive systems, i.e., systems whose macroscopic behavior is described by a diffusion equation. An alternative approach, the \textit{relative entropy method}, also presented for GL dynamics, was proposed in 1991 by Yau \cite{Y91}.

\subsection*{Macroscopic Fluctuation Theory and Large Deviations} In 1995, Quastel \cite{quastel1995large} proved the large deviations principle (LDP) of the empirical density for the same system in the non-gradient case{\footnote{An IPS is called gradient if the microscopic currents of the conserved quantities are discrete gradients, which simplify considerably the study of diffusive IPS with nearest-neighbor interactions (see \cite{KL,BS22}).}}, {proving} that the non-gradient method introduced by Varadhan in \cite{V93} could be extended to the large deviations level. All these works were restricted to IPS evolving under periodic boundary conditions and isolated from the environment. In \cite{ELS90,ELS91}, Eyink, Lebowitz, and Spohn extended the entropy method of \cite{GPV88} to establish hydrodynamic limits for diffusive lattice gases with nearest neighbor interactions \footnote{They considered only almost-gradient models. Extension to non-gradient models remains a challenging open problem.} driven by boundary reservoirs. Introducing boundary reservoirs considerably complicates the analysis because the stationary state—known in the physics terminology as a non-equilibrium stationary state (NESS)—is generally unknown. The corresponding dynamical large deviations have since been derived, mainly in the context of lattice gas \cite{ITA,BGL03, BLM09, FLM11, FGLN22, FGN23}, but also (without full rigor) for the Kipnis-Marchioro-Presutti model \cite{BGL05} or for chains of oscillators with a stochastic noise \cite{B08}. {Nevertheless, the same problem was never established} for the GL dynamics {\footnote{The technical estimates obtained in this paper permits also to fill this gap for nearest-neighbor interactions.}}. In fact, for models with non-compact configuration space, the only boundary driven models for which a static (but not dynamical) large deviations principle has been rigorously obtained are the harmonic model considered in \cite{carinci2025large} and the zero range process in \cite{bernardin2022non}. Our method of proof cannot be carried easily to derive dynamical large deviations for these boundary driven models (see Section 3.3 of \cite{BGL05} for explanations about the mathematical difficulties).

In the setting of boundary-driven diffusive systems, one of the most recent developments is the \textit{Macroscopic Fluctuation Theory} (MFT) introduced by Bertini et al. \cite{ITA,bertini2015macroscopic}. MFT has become the cornerstone of modern non-equilibrium statistical physics and can be viewed as an infinite-dimensional version of the Freidlin–Wentzell theory \cite{FW12}. One of its major contributions is that, for IPS driven by external forces, it provides a definition of a non-equilibrium free energy for the NESS as the solution of a dynamical variational problem. When the system is in equilibrium, this variational problem becomes trivial and recovers the usual equilibrium free energy given by the Gibbs formalism. In non-equilibrium settings, however, it is generally intractable, except in very specific cases.

In \cite{BC25}, the extension of MFT to Ginzburg–Landau dynamics with long‑range interactions in contact with thermal baths is studied in great generality and gives rise to a fractional MFT (fMFT). Long‑range interactions imply that, at the macroscopic level, the hydrodynamic behaviour is no longer described by a diffusion equation but by a fractional diffusion equation with Dirichlet boundary conditions. The aim of the present work is to rigorously justify the program outlined in \cite{BC25} in the simplest case, where the underlying potential defining the GL dynamics is quadratic. Although this assumption leads to significant simplifications compared with the generic (non‑quadratic) case, the mathematical justification of \textit{fractional MFT}—even in this setting—remains nontrivial.

\subsection*{Our contribution and difficulties} The main contributions of this paper are the following. First, we establish a dynamical LDP for the empirical density on a sub‑diffusive space‑time scale (Theorem \ref{thm:dldp}). Second, 
we derive a static LDP for the empirical density of the system in its NESS (Theorem \ref{thm:static LDP}). Finally, we prove that the non‑equilibrium free energy, i.e. the rate functional of the previous static LDP, coincides with the solution to a variational problem involving the rate functional of the dynamical large deviations principle (Theorem \ref{thm:quasipot-freeenergy}).  
Let us now comment on the mathematical difficulties encountered in establishing the main results of this paper. First, as in related works, we are concerned with systems described by (linear) fractional diffusion equations subject to Dirichlet boundary conditions. To give a rigorous meaning to these equations, we adopt the concept of weak solutions introduced in \cite{BCGS2023spa}. These notions differ from the ones used in standard diffusion problems, where boundary conditions are built directly into the weak formulation. In the fractional setting, they instead appear as an additional constraint. At the microscopic level, this requires establishing delicate “replacement lemmas” (see Section \ref{app:superexp RL sec}) that ensure that the empirical density stays close to the prescribed boundary values imposed by the thermal baths.

Second, unlike lattice gases, the configuration space is non-compact. This fact introduces several technical obstacles throughout the proofs, most notably the need to control possible large values in the evolving configurations, required to obtain some exponential tightness. To address this, we establish an original non-reversible maximal inequality to bound the growth of the total mass in Appendix~\ref{app_controlling the mass}. 
This is an extension of the classical Kipnis-Varadhan inequality, \cite{kipnis1986central}, to the non-reversible case. The inequality is based on a Sector Condition on the generator of the process, which is proved for the superdiffusive GL dynamics in Appendix \ref{app. sector cond}.

{To our knowledge, this work provides the first rigorous derivation of both dynamic and static LDP for a superdiffusive system. Furthermore, it establishes the hydrodynamic limit and the corresponding  dynamic LDP for a non-reversible system with a non-compact state space — a problem that had remained open even for systems involving only nearest-neighbor interactions.}

\subsection*{Related works} In recent years, several works have been devoted to the study of hydrodynamic limits and equilibrium fluctuations of lattice gas, particularly exclusion processes with long jumps. Without aiming to be exhaustive, we cite \cite{BGJ21,GS22,BCGS2023spa}, where hydrodynamic limits are derived for the symmetric simple exclusion process with long jumps under various boundary conditions (see also \cite{SS18} for the derivation of hydrodynamic limits in long-range asymmetric interacting particle systems). Further results include the analysis of the hydrodynamic behavior of the symmetric exclusion process with long jumps and a slow barrier \cite{CGJ24}; studies of the non-equilibrium stationary state for the symmetric simple exclusion process with long jumps and the zero-range process with long jumps \cite{BJ17,bernardin2022non}; density fluctuations for the symmetric exclusion process with long jumps \cite{GJ18}; and the derivation of the stochastic Burgers equation from this model \cite{GJ17}.

\subsection*{Outline of the article:}  In Section \ref{sec:Mainresults}, we define the boundary-driven Ginzburg-Landau dynamics and state the main results, namely the hydrodynamic and hydrostatic limits, the dynamical and static large deviation principles for the empirical volume, and the equality between the non-equilibrium free energy and the quasi-potential. Section \ref{sec:hydrodynamics} and Section \ref{sec:hydrostatic} are devoted to the proofs of the hydrodynamic and hydrostatic limits, respectively. The proof of the dynamical large deviation principle is presented in Section \ref{sec:dynLDP}. In Section \ref{sec:MFT}, we establish fractional macroscopic fluctuation theory by computing and showing the equality between the non-equilibrium free energy and the quasi-potential. The paper is supplemented by several appendices.

\section*{Notations}
For the readers' convenience, here we summarize  the main notations adopted in the article.

\subsection*{Function Spaces}

Throughout the paper we will fix a time horizon $T>0$ and denote by $\Omega_T$ the cylinder $[0,T]\times[0,1]$. For $p \in [1,\infty]$, we denote by $L^p([0,1])$ the usual Lebesgue space on $[0,1]$ equipped with the Lebesgue measure. The inner product in $L^2([0,1])$ between two elements $f,g$ is denoted by $\langle f,g \rangle =\int_0^1 f(u)g(u)\,du$ and the corresponding norm by $\Vert \cdot \Vert_2$. We denote by $C^k([0,1])$ (resp. $C_c^k([0,1])$) the space of real-valued functions (resp. with compact support strictly contained in $[0,1]$) whose first $k$ derivatives are continuous. Moreover, we say that a function $H \in C^{m,n}(\Omega_T)$ if $H(\cdot,u)\in C^m([0,T])$ and $H(t,\cdot) \in C^n([0,1])$ for any $u \in [0,1]$ and $t \in [0,T]$. Analogously, $H \in C_c^{m,n}(\Omega_T)$ if $H \in C^{m,n}(\Omega_T)$ and $H(t,\cdot)$ is compactly supported and all the supports of $H(t,\cdot)$ are included in a fixed compact subset of $(0,1)$ for any $t \in [0,T]$. If one (or more) superscript is equal to $\infty$, it means that the function considered is smooth in the respective variable. We will use the notations $H_t(u)$ and $H(t,u)$ interchangeably, so the subscripts should not be confused with partial derivatives. We denote the derivatives of a function $H \in C^{m,n}([0,T]\times [0,1])$ by $\partial_t H$ for the time derivative and $\partial_u H$ for the space derivative.

Given two real valued functions $f,g$ defined on the same space $X$, we will write hereinafter $f(u) \lesssim g(u)$ if there exists a constant $C$ independent of $u$ such that $f(u) \le C g(u)$ for every $u\in X$; moreover, we will write $f(u) = {O} (g(u) )$ if the condition $|f (u) | \lesssim |g(u) |$ is satisfied for all $u\in X$.



\subsection*{Measures}

If $\mu$ is a (signed) measure on a measurable space $X$ and $f:X \to \mathbb R$ an integrable function with respect to $\mu$, we denote equivalently the integral of $f$ with respect to $\mu$ by $\int_X f d\mu$, $\mu (f)$, $\langle \mu, f \rangle$. If $\mu$ is a probability measure, we also use the notation $E_\mu (f)$.

We denote by $\mathcal M$ the set of finite signed measures on $[0,1]$. Since this space is not metrizable for the weak convergence, we follow the approach of \cite{DV89, BBP20}. Let us first recall that the total variation $\Vert \mu \Vert_{\rm{TV}}$ of $\mu \in \mathcal M$ is defined by
\begin{equation}
\label{eq:defTV}
  \Vert \mu \Vert_{\rm{TV}} =\sup_{  \Vert G\Vert_{\infty} \le 1} \,   \langle \mu, G \rangle.
\end{equation}
For any $a\ge 0$,  let $\mathcal M_a$ be the subset of $\mathcal M$ composed of signed measures $\mu$ with total variation $\Vert \mu \Vert_{\rm{TV}}$  bounded by $a$ and equipped with the topology of the weak convergence. Then $\mathcal M_a$ is a compact Polish space (for the weak topology). Observe that $\mathcal M=\cup_{a\in \mathbb N} \mathcal M_a$. We equip $\mathcal M$ with the direct limit topology, i.e., a subset $\mathcal O$ of $\mathcal M$ is open if, and only if, for any $a \in \mathbb N$, $\mathcal O \cap \mathcal M_a$ is open. If $X$ is a topological space $X$, $C([0,T],X)$ denotes the space of continuous paths from $[0,T]$ to $X$. We have that $C([0,T], \mathcal M) = \cup_{a \in \mathbb N} C([0,T], \mathcal M_a)$ and this space is also equipped with the direct limit topology, i.e., a subset $\mathcal O$ of $C([0,T], \mathcal M)$ is open if, and only if, for any $a \in \mathbb N$, $\mathcal O \cap C([0,T], \mathcal M_a)$ is open.

\section{Ginzburg-Landau dynamics and main results}
\label{sec:Mainresults}
\subsection{The microscopic model}
For $n \ge 2$, let $\Lambda_n=\{1, \ldots,n-1\}$  denote the discrete one dimensional lattice. From now we fix a parameter $\gamma\in(1,2)$ and let $p:\mathbb R \to[0,1]$ be the function defined by
\begin{equation}
    p(z):=c_\gamma\frac{\boldsymbol{1}_{z\neq0}}{|z|^{\gamma+1}},
\end{equation}
where $c_\gamma=\sum_{z \in \mathbb Z \backslash \{0\}}1/|z|^{1+\gamma}$ is a normalizing constant such that the restriction of $p$ to $\mathbb Z$ is a probability distribution. The microscopic model under consideration in this paper is the boundary driven long-range Ginzburg-Landau dynamics. This is a stochastic lattice field model $\varphi=(\varphi_t)_{t\geq0}=\{\varphi_t(x)\; ; \; t\geq0\;,\; x\in\Lambda_n\}$ living in the state-space $\Omega_n=\RR^{\Lambda_n}$. Here, $\varphi_t(x)$ denotes the height of the interface at site $x$ for the configuration $\varphi_t \in \Omega_n$. It is a Markov process with state-space $\Omega_n$ whose infinitesimal generator is described as follows. Fix two constants $\Phi_\ell$ and $\Phi_r$ that will denote the fixed densities at the left and right boundary, respectively.The generator is given by 
\begin{equation}\label{eq:generator}
    \mathcal L_n=n^\gamma(\mathcal L^\ell +\mathcal L^r +\mathcal{L}^{\text{bulk}}),
\end{equation}
where
\begin{equation}
\begin{split}
   &\mathcal{L}^{\text{bulk}} =\cfrac 12 \sum_{x,y \in \Lambda_n} p( y -x)\Big(  ({\varphi (x)} -{\varphi (y)} )(\partial_{\varphi (y)} -\partial_{\varphi (x)} )  +  (\partial_{\varphi (y)} -\partial_{\varphi (x)} )^2  \Big),\\
   &\mathcal L^\ell   = \partial^2_{\varphi(1)}  + (\Phi_\ell -\varphi (1))\partial_{\varphi(1)}\quad\text{and}\quad\mathcal L^r   = \partial^2_{\varphi(n-1)}  + (\Phi_r -\varphi (n-1))\partial_{\varphi(n-1)}.
\end{split}
\end{equation}
For the purposes of this paper, we will need to consider also the dynamics perturbed by an external field $H$. Fix $H$ in $ C_c^{1,2}(\Omega_T)$. The generator $\mathcal L_n^H$ of the perturbed Ginzburg-Landau dynamics is given by
\begin{equation}\label{eq:generator_new}
\mathcal L_n^H =n^\gamma\Big(\mathcal L^\ell +\mathcal L^r +\mathcal{L}^{\text{bulk}}+\mathcal{L}^{\text{tilt}} \Big)=\mathcal L_n+n^\gamma\mathcal{L}^{\text{tilt}}
\end{equation}
where the tilt part is given by 
\begin{equation}
\begin{split}
  {\mathcal L}^{\text{tilt}} = \frac{1}{2}  \sum_{x,y \in \Lambda_n} p (y-x)\Big( H_t\left( \tfrac yn \right) -H_t \left( \tfrac xn \right)  \Big) \Big( \partial_{\varphi(y)} - \partial_{\varphi(x)} \Big).
\end{split}
\end{equation}

The same notation $(\varphi_t)_{t \ge 0}$ is used to denote the Markov process, whether generated by $\mathcal L_n$ or by $\mathcal L_n^H$; the intended generator will be unambiguous from the context. The unperturbed Ginzburg-Landau dynamics, when $H=0$, is irreducible and therefore admits a unique stationary state denoted by $\mu_{ss}^n$, which we describe now  more precisely. Let $\mathcal V$ and $\mathcal E$ denote the volume and the energy, respectively, of a configuration $\varphi$ given by 
\begin{equation}
\mathcal V (\varphi )=\sum_{x \in \Lambda_n} \varphi (x) \quad\text{and}\quad\mathcal E (\varphi) =\cfrac12 \sum_{x \in \Lambda_n} \varphi^2 (x).
\end{equation}

If $\Phi_\ell=\Phi_r=\Phi$, we say that the model is in equilibrium, and the Gaussian (product) probability measure $\mu_\Phi$ given by 
\begin{equation}
\label{invariant_measure}
\mu_\Phi (d\varphi) = \frac{1}{\mathcal Z^n_\Phi} \exp\Big( -\mathcal E (\varphi) +\Phi \mathcal V (\varphi) \Big) d\varphi
\end{equation}
is  the unique invariant probability measure of the dynamics.  Here $\mathcal Z^n_\Phi = [Z (\Phi)]^n$ where $Z(\Phi) =\sqrt{2\pi} \ e^{\Phi^2/2}$. When $\Phi_\ell\neq \Phi_r$, we say that the system is out of equilibrium. Contrarily to many other boundary driven systems, it is still  possible to compute semi-explicitly the non-equilibrium stationary state (NESS) $\mu_{ss}^n$. In Section \ref{subsubsec:theo-stationarity0} the following result will be proved.
\begin{lemma}
\label{lem:NESS}
The NESS $\mu_{ss}^n$ of the boundary driven Ginzburg-Landau dynamics with long range interactions has a density $f_{ss}^n$ with respect to the Lebesgue measure $d\varphi =\prod_{x\in \Lambda_n} d\varphi (x)$  in  product form
\begin{equation}
\label{NESS}
\mu_{ss}^n (d\varphi) = \prod_{x\in \Lambda_n} \frac{1}{\sqrt{2\pi}}  \exp\left(-\frac{(\varphi(x)-\Phi^n_{ss}(x))^2}{2}\right) d\varphi (x)
\end{equation}
where $\Phi^n_{ss}:\Lambda_n\to[\Phi_\ell,\Phi_r]$ is the unique solution to the equation 
\begin{equation}
\label{eq:discreteprofile}
    \sum_{y\in\Lambda_n} p(y-x)\big(\Phi^n_{ss}(y)-\Phi^n_{ss}(x)\big) +\textbf{1}_{x=1}(\Phi_\ell - \Phi_{ss}^n (1) )+\textbf{1}_{x=n-1}(\Phi_r - \Phi_{ss}^n (n-1) )= 0
\end{equation}
for all $ x\in\Lambda_n$.
In particular, $E_{\mu_{ss}^n} [\varphi (x)] = \Phi_{ss}^n (x)$ for any $x \in \Lambda_n$. Moreover,
\begin{equation}
    \label{eq:boundprofileness}
    \min (\Phi_\ell, \Phi_r) \le \inf_{x \in \Lambda _n} \Phi_{ss}^n (x) \le \sup_{x \in \Lambda_n} \Phi_{ss}^n (x) \le \max (\Phi_\ell, \Phi_r). 
\end{equation}
We denote by $\pi_{ss}^n \in \mathcal M$ the finite signed measure on $[0,1]$ given by  \begin{equation}\label{eq:empirical_deterministic_stationary}
    \pi_{ss}^n (du) =\frac{1}{|\Lambda_n|} \sum_{x \in \Lambda_n} \Phi_{ss}^n (x) \delta_{x/n}.
\end{equation}
\end{lemma}

\subsection{Fractional Operators}
Fix $\gamma \in (1,2)$. Since the variance of the transition probability is infinite, the macroscopic behavior of the system will be given in terms of fractional operators.

\subsubsection{Regional fractional Laplacian on $[0,1]$} The regional fractional Laplacian{\footnote{On domains with boundaries, one can find several notions of (different) ``fractional Laplacian operators'' in the literature. The three main famous are the regional, the restricted and the spectral. Unfortunately, the nomenclature is not fixed and some authors call the regional fractional Laplacian defined here the restricted fractional Laplacian and vice-versa.}} $\mathbb L^\gamma$ on $[0,1]$ is defined for functions $F\in L^1 ([0,1])$ as the principal value
\begin{equation}
   \forall u \in [0,1], \quad \mathbb L^\gamma F(u):=c_\gamma\lim_{\varepsilon\to0}\int_{[0,1]}\boldsymbol{1}_{|u-v|\geq\varepsilon}\frac{F(v)-F(u)}{|v-u|^{1+\gamma}}dv, 
\end{equation}
as soon as the previous limit exists. We recall that $c_\gamma$ is the normalizing constant of the kernel $p$ on $\mathbb Z$. 

In fact, if $F\in C_c^2([0,1])$ we have
that $\mathbb L^\gamma F$ is a bounded function on $[0,1]$ which can equivalently be defined, for any $u\in [0,1]$, by 
\begin{equation}
\label{eq:fraclapC2}
\begin{split}
(\mathbb L^\gamma F) (u) &= c_\gamma \int_{[0,1]}\frac{F(v)-F(u) -F'(u) (v-u)}{|v-u|^{1+\gamma}}dv  - F'(u) \frac{c_\gamma}{\gamma-1} \big[u^{1-\gamma} + (1-u)^{1-\gamma}\big].
\end{split}
\end{equation}

 We introduce the semi-inner product $\langle\cdot,\cdot\rangle_{\gamma/2}$ and its associated semi-norm $\|\cdot\|_{\gamma/2}^2=\langle\cdot,\cdot\rangle_{\gamma/2}$ by 
\begin{equation}
\langle F,G\rangle_{\gamma/2}=\frac{c_\gamma}{2}\iint_{[0,1]^2}\frac{\big(F(v)-F(u)\big)\big(G(v)-G(u)\big)}{|u-v|^{1+\gamma}}dudv
\end{equation}
for functions $F,G:[0,1]\to\mathbb R$ with $\|F\|_{\gamma/2},\|G\|_{\gamma/2}<\infty$. Note that $\|F\|^2_{\gamma/2}=\langle F,-\mathbb L^\gamma F\rangle$.

The fractional Sobolev space $\mathcal H^{\gamma/2}$ associated to $\mathbb L^\gamma$ is defined by
\begin{equation}
\mathcal H^{\gamma/2}= \Big\{ G\in L^2 ([0,1])\; ; \; \|G\|^2_{\gamma/2}=\frac{1}{2}\iint_{\mathbb [0,1]^2} p (v-u) {\vert G(v)-G(u)\vert^2} \ dudv < \infty \Big\}.
\end{equation}
It is a Hilbert space endowed with the norm $\Vert \cdot \Vert_{\mathcal H^{\gamma/2}}$ defined by $\Vert G \Vert^2_{\mathcal H^{\gamma/2}}:=\Vert G \Vert_2^2+ \Vert G \Vert^2_{\gamma/2}.$ Its elements coincide a.e. with continuous functions (in fact, Theorem 8.2 of \cite{DNPV} guarantees that elements of $\mathcal H^{\gamma/2}$ are $\frac{\gamma-1}{2}$-H\"older continuous function). We will denote the completion of $C^\infty_c ([0,1])$ with this norm by $\mathcal H_0^{\gamma/2}$. This is a Hilbert space whose functions coincide a.e. with continuous functions vanishing at $0$ and $1$. Thanks to the fractional Poincar\'e inequality \cite{DNPV}, on $\mathcal H_0^{\gamma/2}$, $\Vert \cdot \Vert^2_{\mathcal H^{\gamma/2}}$ and $\Vert \cdot \Vert^2_{\gamma/2}$ are equivalent. 

\subsubsection{Discrete fracional Laplacian on $\tfrac{\Lambda_n}{n}$}
For any $G: \tfrac{\Lambda_n}{n} \to \mathbb R$, the discrete fractional Laplacian $\mathbb L_n^\gamma G $ of the function $G$ is defined by
\begin{equation}
\label{eq:operator_L_gamma}
 \forall x \in \Lambda_n, \quad (\mathbb L_n^\gamma G)\big( \tfrac{x}{n} \big):=n^\gamma\sum_{y\in\Lambda_n}p(y-x)\big(G \big(\tfrac{y}{n}\big)-G\big(\tfrac{x}{n}\big)\big)
\end{equation}
and the inner-product $\langle F, G \rangle_{n,\gamma/2}$ between $F,G:\tfrac{\Lambda_n}{n}\to\mathbb R$  is given by
\begin{equation}
\label{eq:seminorm_L_gamma}
    \langle F, G \rangle_{n,\gamma/2} := \frac{n^\gamma}{2n}\sum_{x,y\in\Lambda_n}p(y-x)\big(F\big(\tfrac{y}{n}\big)-F\big(\tfrac{x}{n}\big)\big)\big(G \big(\tfrac{y}{n}\big)-G \big(\tfrac{x}{n}\big)\big).
\end{equation}
The corresponding semi-norm is denoted by $\Vert \cdot \Vert_{n,\gamma/2}$, i.e., for $F: \tfrac{\Lambda_n}{n} \to \mathbb R$,
\begin{equation}
\label{eq:DiscreteSobolevNorm}
\begin{split}
\Vert F \Vert_{n,\gamma/2}^2 = \frac{1}{2n} \sum_{x \in \Lambda_n} p(y-x)\big(F \big(\tfrac{y}{n}\big)-F \big(\tfrac{x}{n}\big)\big)^2.
\end{split}
\end{equation}
From \cite[Lemma 5.1]{BCGS2023spa}, if $F \in C^2 ([0,1])$, then 
\begin{equation}
\label{eq:Sob1}
\begin{split}
\lim_{n \to \infty} \cfrac{1}{|\Lambda_n|} \sum_{x \in \Lambda_n} \left\vert (\mathbb L_n^\gamma F ) \big( \tfrac{x}{n} \big) - (\mathbb L^\gamma F ) \big( \tfrac{x}{n}\big)  \right\vert  = 0, \quad  \lim_{n \to \infty}  \Vert F \Vert_{n,\gamma/2} = \Vert F \Vert_{\gamma/2}.
\end{split}
\end{equation}
Moreover, by \cite[Lemma 3.3]{BJ17}, if $F\in C_c^2 ([0,1])$ with compact support included in $[a, 1-a]$, $0<a<1$, then 
\begin{equation}
\label{eq:Sob2}
\begin{split}
\sup_{x \in \Lambda_n} \left\vert (\mathbb L_n^\gamma F ) \big( \tfrac{x}{n} \big) - (\mathbb L^\gamma F ) \big( \tfrac{x}{n}\big)  \right\vert  \le C(a, \Vert F'' \Vert_{\infty})/n.
\end{split}
\end{equation}

\subsection{Hydrodynamic and Hydrostatic Equations}
Let $L^{2}([0,T],\mathcal{H}^{\gamma/2})$ be the space of measurable functions $\Phi:[0,T]\to \mathcal H^{\gamma/2}$ such that
$$\|\Phi\|_{L^{2}([0,T],\mathcal{H}^{\gamma/2})}^2:=\int_0^T\|\Phi_t\|^2_{\mathcal {H}^{\gamma/2}}dt<+\infty.$$
We define similarly the space  $L^{2}([0,T],\mathcal{H}^{\gamma/2}_0)$.

\begin{definition}
\label{def:space_C_ab}
Denote by $\mathcal C^{ac}$ the subset of $C([0,T],\mathcal M)$ composed of paths $\pi$ satisfying:
\begin{enumerate}[i)]
\item $\pi_t(du)=\Phi_t(u)du$ for some $\Phi\in L^2([0,T],\mathcal H^{\gamma/2})$;
\item $\Phi_t(0)=\Phi_\ell$ and $\Phi_t(1)=\Phi_r$ for a.e. $t\in[0,T]$. 
\end{enumerate}
\end{definition}

Hereinafter, if $\pi \in \mathcal C^{ac}$, we denote its density by $\Phi_t$.
\begin{definition}
	\label{Def. Dirichlet Condition2}
	Let $H\in  L^2 ([0,T], \mathcal H_0^{\gamma/2})$. A function $\Phi:\Omega_T\to \mathbb  R$ is a weak solution of the drifted fractional diffusion equation with non-homogeneous Dirichlet boundary conditions and initial condition $g\in L^2 ([0,1])$:
	\begin{equation}
	\label{eq:Dirichlet Equation2}
	\begin{dcases}
	&\partial_{t} \Phi_{t} (u)= \mathbb L^\gamma \Phi_t (u)  - \mathbb L^\gamma H_t (u) ,  \quad (t,u) \in  \Omega_T ,\\
& \Phi_{t} (0)=\Phi_\ell, \quad  \Phi_{t} (1)=\Phi_r,\quad t \in (0,T], \\
	& \Phi_{0} (u)= g(u),\quad u \in (0,1),
	\end{dcases}
	\end{equation}
	if the following conditions hold:
    \begin{enumerate}[i)]
    \item $\Phi\in L^2([0,T],\mathcal H^{\gamma/2})$;
    
    \item for all functions $G \in C_c^{1,\infty} (\Omega_T)$ and any $t\in[0,T]$ we have
		\begin{equation}
		\label{eq:Dirichlet integral2}
		\begin{split}
		&F_{H}(t,\Phi,G,g) :=\left\langle \Phi_{t},  G_{t} \right\rangle -\left\langle g,   G_{0}\right\rangle - \int_0^t\left\langle \Phi_{s},\Big(\partial_s + \mathbb L^\gamma \Big) G_{s}  \right\rangle ds - \int_0^t \left\langle H_s, G_s \right\rangle_{\gamma/2} \, ds=0;
		\end{split}   
		\end{equation}
        \item $\Phi_t(0)=\Phi_\ell$ and $\Phi_t(1)=\Phi_r$ for a.e. $t\in[0,T]$.
        \end{enumerate}
        
\end{definition}

\begin{proposition}
\label{prop:uniqueness}
Let $H \in L^2([0,T], \mathcal H^{\gamma/2}_0)$. There exists at most one weak solution to \eqref{eq:Dirichlet Equation2}. 
\end{proposition}

The argument to prove this result is analogous to that of Proposition 2.16 in \cite{BCGS2023spa}, which establishes uniqueness for equation (2.18) in that reference. We therefore omit the details and refer the reader to Section 4.3 of the aforementioned article. The proof given there does not require boundedness of solutions and thus extends directly to our setting; the presence of the perturbation term introduces no additional complications.

At this stage, we have not yet established the existence of a weak solution to \eqref{eq:Dirichlet Equation2}. Combining Theorem \ref{theo:hydro_limit} (stated below) with Proposition \ref{prop:uniqueness} yields the following result.

\begin{theorem}
\label{theo:existenceuniquenesshydrodynamicequations}
Let $g \in L^2([0,1])$ and $H \in C_c^{1,2}(\Omega_T)$. There exists a unique weak solution to \eqref{eq:Dirichlet Equation2} that we denote by $\Phi^{H,g}$.
\end{theorem}

\begin{remark}
    This theorem is in fact true as soon as $H \in L^2 ([0,T], \mathcal H_0^{\gamma/2})$, see Remark \ref{rem:extension-existence}.
\end{remark}

Define, for any $t \in [0,T]$, the finite signed measure $\pi_{t}^{H,g} \in \mathcal M$ by $\pi_t^{H,g}(\mathrm du) = \Phi_{t}^{H,g}(u)\, \mathrm du$.

Let us now consider the problem of existence and uniqueness of the stationary profile $\Phi_{ss}$ solution of the equation 
\begin{equation}
    \label{eq:stationary}
    \begin{cases}
      & \mathbb L^\gamma \Phi_{ss} \, (u) = 0, \quad u\in [0,1],\\
      &  \Phi_{ss} (0)=\Phi_\ell, \quad {\Phi}_{ss} (1) = \Phi_r.
    \end{cases}
\end{equation}
As  for the hydrodynamic equation, we have to interpret this equation in a weak sense.

\begin{definition}
	\label{def:stationary}
	We say that the measurable function $\Phi_{ss}:[0,1] \to \mathbb R$ is a weak solution of the stationary problem with non-homogeneous Dirichlet boundary conditions \eqref{eq:stationary} if:
    \begin{enumerate}[i)]
\item $\Phi_{ss}\in\mathcal H^{\gamma/2}$;
\item $\Phi_{ss} (0)=\Phi_\ell$ and $\Phi_{ss}(1)=\Phi_r$;
\item For any function $G \in C_c^{\infty} ([0,1])$, we have $\langle \Phi_{ss}, \mathbb L^{\gamma} G \rangle =0$.
    \end{enumerate}
\end{definition}

We have the following theorem, which is proved in Section \ref{subsubsec:theo-stationarity}.
\begin{theorem}
    \label{theo:stationary}
    There exists a unique weak solution $\Phi_{ss}$ of the stationary problem with non-homogeneous Dirichlet boundary conditions given in \eqref{eq:stationary} and it satisfies the following maximum principle
    \begin{equation}
    \label{eq:boundphissn}
    \min (\Phi_\ell, \Phi_r) \le \inf_{u \in [0,1]} \Phi_{ss} (u) \le \sup_{u \in [0,1]} \Phi_{ss} (u) \le \max (\Phi_\ell, \Phi_r). 
\end{equation}
Moreover, for any $G\in C([0,1])$, 
\begin{equation}
\label{eq:convPhissn-Phiss}
\lim _{n\to\infty }  \left\vert \frac{1}{n}\sum_{x\in\Lambda_n}G(\tfrac xn)\Phi_{ss}^n(x)  - \int_{0}^1G(u) \Phi_{ss}(u)du \right\vert   = 0.
\end{equation} 
We then denote by $\pi_{ss} \in \mathcal M$ the finite signed measure defined by $\pi_{ss} (du)=\Phi_{ss} (u) du$.

\end{theorem}

\subsection{Hydrodynamic and Hydrostatic Limits}

The empirical measure $\pi^n (\varphi) \in \mathcal M$ corresponding to the configuration $\varphi \in \Omega_n$ is the finite signed measure on $[0,1]$  defined by  
\begin{equation}
\pi^n (\varphi,du) = \cfrac{1}{\vert \Lambda_n\vert} \sum_{x \in \Lambda_n} \varphi (x) \ \delta_{\tfrac xn} (du).
\end{equation}

To simplify notation, for any time $t\ge 0$, we denote $\pi^n_t:=\pi^n (\varphi_t,du)$. For $G\in C([0,1])$ and a measure $\pi$ in $\mathcal{M}$, recall that we denote by $\langle \pi, G\rangle${\footnote{There is a small abuse of notation since we also denote by $\langle \cdot,\cdot \rangle$ the usual scalar production in $L^2 ([0,1])$.}} the integral of $G$ with respect to $\pi$:
\begin{equation}
    \langle \pi, G\rangle =\int_0^1 G(u)\pi(du).
\end{equation}

We denote by $\PP_n^H$ the  law on $C([0,T], \Omega_n)$ of the  Markov process $(\varphi_t)_{t \ge 0}$ with generator $\mathcal L_n^H$ and initial distribution $\mu_n$ and  $\EE_{n}^{H}$ denotes the corresponding expectation. For any $n\ge 2$,  $\mathbb{Q}_{n}^{H} = \mathbb P_n^H \circ \pi^n$ denotes the push-forward of $\mathbb P_n^H$ by $\pi^n$, i.e., the law on $C([0,T],\mathcal{M})$ of the processes $(\pi^{n} (\varphi_t))_{t \ge 0}$ when $(\varphi_t)_{t \ge 0}$ has generator $\mathcal L_n^H$ and starts from $\mu_n$. If $H=0$, we suppress from notation the superscript $H$. Sometimes we will start the dynamics from an initial distribution $\lambda_n$ different from $\mu_n$. In this case, we will change the notation $\mathbb P_n$ to $\mathbb P_{\lambda_n}$ etc.

\begin{definition}
\label{eq:def-initialprofile}
Let $g \in L^2 ([0,1])$. We say that a sequence of probability measures $(\mu_{n})_{n\geq 2 }$ on $\Omega_{n}$  is associated to the profile $g$ if for any $G\in C([0,1])$  and every $\delta > 0$ 
\begin{equation}\label{eq:seq_associated_profile}
\lim _{n\to\infty } \mu _{n}\Big( \Big\vert \langle \pi^n, G\rangle - \int_{0}^1G(u) g (u)du \Big\vert    > \delta \Big)= 0 .
\end{equation}  
\end{definition}

Let $\nu_n$ denotes the standard Gaussian product measure on $\Omega_n$, namely
\begin{equation}
    \nu_n(d\varphi)=\prod_{x\in\Lambda_n}\frac{1}{\sqrt{2\pi}}\exp\big(-\varphi(x)^2/2\big)d\varphi(x).
\end{equation}

We shall assume the following conditions on the initial measures:
\begin{equation}
    \int_{\Omega_n} \Big(\frac{d\mu_n} {d\nu_n}\Big)^r  d {\nu_n} \le e^{K_r n},\quad \text{for some}\quad r>1\quad \text{and some}\quad K_r> 0. 
    \label{eq:Kr}
\end{equation}

By Lemma \ref{lem:Patricia1},  assumption \eqref{eq:Kr} implies the entropy bound
\begin{equation}
\label{eq:assump entropy initial measure gaussian}
\forall n \ge 2, \quad H(\mu_n|\nu_n)\leq K_0 n,\quad \text{with}\quad K_0>0.    
\end{equation}


It is easy to check that a sequence of local equilibrium states $(\mu_n)_{n\ge 2}$  given by 
\begin{equation}
\mu_n (d\varphi) = \cfrac{1}{(2\pi)^{|\Lambda_n|/2}}\prod_{x \in \Lambda_n} e^{- \tfrac{1}{2} (\varphi(x) - \Phi_0(x/n))^2 }d\varphi (x) 
\end{equation}
where $\Phi_0:[0,1] \to \mathbb R$ is a bounded measurable function satisfies \eqref{eq:Kr}. The next result is the law of large numbers for the empirical measure, whose  proof will be detailed in Section \ref{sec:hydrodynamics}.

\begin{theorem}[Hydrodynamic limit]
\label{theo:hydro_limit}
Let $g \in L^2 ([0,1])$, $H\in C_c^{1,2}(\Omega_T)$ and let $(\mu _{n})_{n\geq 2}$ be a sequence of probability measures on $\Omega_{n}$ associated to the profile $g$ and satisfying \eqref{eq:Kr}. Then, for any $0\leq t \leq T$, any $\delta>0$ and any $G\in C([0,1])$,
	\begin{equation}
	\label{limHidreform}
	\lim_{n\to\infty}\PP^H_{n} \Big(\Big\vert \langle \pi^n_t,G\rangle- \int_{0}^1G(u)\Phi_{t} (u)du \Big\vert    > \delta \Big)=0,
	\end{equation}
 where $\Phi$ is the weak solution of \eqref{eq:Dirichlet Equation2} in the sense of Definition \ref{Def. Dirichlet Condition2}.	
\end{theorem}

The next theorem is a consequence of the previous one and we present its proof in Section \ref{sec:hydrostatic}.

\begin{theorem}[Hydrostatic Limit]
\label{thm:hydro static limit}

We recall that $\mu^n_{ss}$ denotes the invariant measure of the process $(\varphi_t)_{t\ge 0}$ with generator $\mathcal L_n$ and that $\Phi_{ss}$ denotes the unique weak solution of \eqref{eq:stationary} in the sense of Definition \ref{def:stationary}. Then,  for every $G\in C([0,1])$ and every $\delta>0$,
\begin{equation}
    \lim_{n\to\infty}\mu_{ss}^n\Big(\Big|\langle \pi^n,G\rangle-\int_0^1 G(u)\Phi_{ss}(u)du\Big|>\delta\Big)=0.
\end{equation}
    
\end{theorem}

\subsection{Dynamical Large Deviations}
For each $H\in C_c^{1,2}(\Omega_T)$ and for any density profile $g\in L^2([0,1])$ define the linear functional $J_{H}( \cdot \vert g):\mathcal C^{ac}\to\mathbb R$ by 
\begin{equation}
\label{eq:JHPG}
J_{H}(\pi|g)=
    \langle\pi_T,H_T\rangle -\langle g,H_0\rangle-\int_0^T\langle\pi_s,(\partial_s+\mathbb L^\gamma)  H_s\rangle\;ds-\int_0^T\|H_s\|_{\gamma/2}^2\;ds.
\end{equation}
Recall Definition \ref{def:space_C_ab}. The large deviations rate function $I_{[0,T]}(\cdot|g):C([0,T],\mathcal M)\to\mathbb R$ is defined by 
\begin{equation}
\label{eq:rate function}
I_{[0,T]}(\pi|g)=
\begin{cases}\sup_{H\in C_c^{1,2}(\Omega_T)}J_{H}(\pi|g),\quad\text{if}\quad \pi\in\mathcal C^{ac}\\
+\infty,\quad\text{otherwise.}
    \end{cases}
\end{equation}

Since the functional $J_H(\cdot \vert g)$ is linear, the rate function $I_{[0,T]}(\cdot|g):C([0,T],\mathcal M)\to[0,+\infty]$ is convex and lower semi-continuous. Now we can state our main result.
\begin{theorem}[Dynamical Large Deviations Principle]
\label{thm:dldp}
    Let $g \in L^2 ([0,1])$ and let $(\mu_n)_{n \ge 2}$ be a sequence of initial probability measures associated to the profile $g$ (see Definition \ref{eq:def-initialprofile}) satisfying \eqref{eq:assump entropy initial measure gaussian}. Then the sequence $(\mathbb{Q}_n)_{n\ge 2}$ satisfies a large deviations principle with speed $n$ and rate function $I_{[0,T]}$. Namely, for every open set $\mathcal O$ and closed set $\mathcal F$ of $C([0,T],\mathcal M)$, we have 
    \begin{equation}
        \begin{split}
            &\limsup_{n\to\infty}\frac{1}{n}\log\mathbb P_{n}(\pi^n\in\mathcal F)\leq -\inf_{\pi\in\mathcal F}I_{[0,T]}(\pi|g),\\
            &\liminf_{n\to\infty}\frac{1}{n}\log\mathbb P_{n}(\pi^n\in\mathcal O)\geq -\inf_{\pi\in\mathcal O}I_{[0,T]}(\pi|g).
        \end{split}
    \end{equation}
\end{theorem}

\subsection{Static Large Deviations}

We turn to the description of stationary large deviations. For $G\in C ([0,1])$, let 
\begin{equation}
\label{eq:macro MGF}
    \mathfrak F(G) :=\int_{0}^{1} \Big(\Phi_{ss}(u) G(u)  + \dfrac12 {G^2(u)} \Big) \, du
\end{equation}
where $\Phi_{ss}$ denotes the unique weak solution to \eqref{eq:stationary} in the sense of Definition \ref{def:stationary}. The Legendre transform $W:{\mathcal M} \to {\mathbb R}$ of $\mathfrak F$ is given by 
\begin{equation}
\label{eq: static rate func variational}
  \forall\; \varrho \in \mathcal M, \quad  W (\varrho) = \sup_{G\in C ([0,1])} \Big\{ \langle \varrho, G\rangle - \mathfrak F (G) \Big\}.
\end{equation}
It is not difficult to show that $W$ is given by
\begin{equation}
\label{eq:NEFE}
W(\varrho) = 
    \begin{dcases}
      & \frac{1}{2}\int_0^1\big(\rho(u)-\Phi_{ss}(u)\big)^2du, \quad \text{if} \quad \varrho(du)= \rho(u)du \quad \text{with} \quad \rho \in L^2 ([0,1]),\\
   \hspace{0.1cm}
      &+\infty, \quad \text{otherwise}.
    \end{dcases}
\end{equation}

\begin{theorem}[Static Large Deviations Principle]
\label{thm:static LDP}
Under $\mu_{ss}^n$, the sequence of random variables $(\pi^{n})_{n\ge 2}$ satisfies a large deviation principle at speed $n$ with convex lower semi-continuous rate function $W$.
\end{theorem}

\subsection{The Quasi-Potential} 
We define the quasi-potential $V:\mathcal M\to[0,+\infty]$ by 
\begin{equation}
\label{eq:quasi potential}
  \forall\; \varrho \in \mathcal M, \quad   V(\varrho)=\inf_{T>0}\inf_{\pi} I_{[0,T]}(\pi|\Phi_{ss})
\end{equation}
where the second infimum runs over all paths $\pi \in C([0,T], \mathcal M)$ such that $\pi_0 (du)=\Phi_{ss}(u) du$ and $\pi_T (du) =\varrho(du)$. Hence, $V$ represents the minimal cost to produce the profile $\varrho$ starting from the stationary state $\Phi_{ss}$. In Section \ref{subsec:QPHJ} we shall prove the following result.
\begin{theorem}
\label{thm:quasipot-freeenergy}
    For each $\varrho \in\mathcal M$, $V(\varrho)=W(\varrho)$.
\end{theorem} 
\section{Hydrodynamic Limit}
\label{sec:hydrodynamics}
In this section, we consider the process $(\varphi_t)_{t \ge 0}$ generated by the perturbed generator $\mathcal L_n^H$.

\subsection{Dynkin's martingales}

Fix $G\in C_{c}^{1,2}(\Omega_T)$. We start by computing the martingale decomposition for the process $(\langle\pi^n_t,G_t\rangle)_{t \ge 0}$. Recall from Lemma 5.1 in Appendix 1 of \cite{KL} that the process $M^{H,n}$ defined by 
\begin{equation}
\label{eq:dynkin}
    \forall t \ge 0, \quad M_t^{H,n}(G)=\langle\pi_t^n,G_t\rangle-\langle \pi_0^n,G_s\rangle-\int_0^t(\partial_s+\mathcal{L}_n^H)\langle\pi_s^n,G_s\rangle ds
\end{equation}
is a mean zero martingale with quadratic variation at time $t\ge 0$ given by 
\begin{equation}\label{eq:QV_gen}
    \langle M^{H,n}(G) \rangle_t=\int_0^t \mathcal{L}_n^H\langle \pi_s^n,G_s\rangle^2-2\langle\pi_s^n,G_s\rangle\mathcal{L}^H_n\langle\pi_s^n,G_s\rangle\; ds.
\end{equation}
Straightforward computations show that 
\begin{equation}
\label{eq:dynkin martingale}
    \begin{split}
(\partial_s+\mathcal{L}_n^H)\langle\pi_s^n,G_s\rangle&=\langle\pi_s^n,(\partial_s+\mathbb{L}_n^\gamma)G_s\rangle+ \langle H_s, G_s\rangle_{n,\gamma/2}
    \\
    &+n^{\gamma-1}(\Phi_\ell-\varphi(1))G_t\big(\tfrac{1}{n}\big)+n^{\gamma-1}(\Phi_r-\varphi(n-1))G_t\big(\tfrac{n-1}{n}\big),
    \end{split}
\end{equation}
where $\mathbb L_n^\gamma$ and $\langle \cdot,\cdot \rangle_{n,\gamma/2}$ have been defined in \eqref{eq:operator_L_gamma} and \eqref{eq:seminorm_L_gamma}. Moreover, the same kind of computations show that the quadratic variation of the martingale writes as
\begin{equation}
\label{eq:QV}
\begin{split}
    \langle M^{H,n}(G) \rangle_t&=\int_0^t \frac{n^\gamma}{n^2}\sum_{x,y\in\Lambda_n}p(y-x)\big(G_s\big(\tfrac{y}{n}\big)-G_s\big(\tfrac{x}{n}\big)\big)^2ds
    +\int_0^t \frac{2n^\gamma}{n^2}\Big(G_s\big(\tfrac{1}{n}\big)^2+G_s\big(\tfrac{n-1}{n}\big)^2\Big)\;ds.
    \end{split}
\end{equation}

We proceed now as usual: using the above martingale decomposition, we prove tightness of  the sequence $(\mathbb Q^H_n)_{n\ge 2}$ and concentration for all limit points on absolutely continuous measures. Subsequently, we also characterize uniquely the limit point thanks to the uniqueness of the weak solution to the  hydrodynamic equation.

\subsection{Tightness}

In this subsection, we prove that $(\mathbb {Q}_{n}^H)_{n \geq 2} $ is tight with respect to the topology of $C([0,T],\mathcal M)$. In \cite{DV89, BBP20} one can find the following statement.

\begin{theorem}[Arzela-Ascoli]\label{thm:tightness criteria}
    The sequence of probability measures $(\mathbb Q_n^H)_{n \ge 2}$ on $C([0,T], \mathcal M)$ is tight if the two following conditions are satisfied:
    \begin{enumerate}[i)]
    \item $\limsup_{a \to \infty} \limsup_{n \to \infty} \mathbb Q_n^H \Big( \pi \notin C([0,T], \mathcal M_a)\Big) = 0$;
    \item for any $G\in C([0,1])$ 
    \begin{equation}
         \limsup_{\delta \to 0} \limsup_{n \to \infty} \mathbb Q_n^H \Bigg( \sup_{\substack{0 \le s \le t \le T\\ |t-s| \le \delta}}\,  \left\vert \langle \pi_t, G \rangle  -  \langle \pi_s, G \rangle \right\vert \ge \delta  \Bigg) =0.
    \end{equation}
    \end{enumerate}
\end{theorem}

We start proving the first item. This is a direct consequence of Proposition \ref{prop: tightness estimate} whose proof follows from  a non-reversible maximal inequality, proved in Appendix \ref{app_controlling the mass}. Observe that
\begin{equation}
    \begin{split}
    \mathbb Q_n^H \left( \pi \notin C([0,T], \mathcal M_a)\right) & = \mathbb Q_n^H \left(  \sup_{0\le t \le T} \Vert \pi_t \Vert_{TV}>a  \right)
    \end{split}
\end{equation}
where the total variation $\Vert \cdot\Vert_{TV}$ is defined by \eqref{eq:defTV}.  By Proposition \ref{prop: tightness estimate} proved below,  the statement  \textit{i)}  holds for $H=0$. We extend then \textit{i)}   to the case $H\ne 0$ by using Lemma \ref{lem:super_exchange_measures}.

\begin{proposition}
\label{prop: tightness estimate}
There exist $p,q>1$ such that $p^{-1}+q^{-1}=1$, $\kappa_q>0$, $\theta>0$,  such that, for any $B>0$, for any $G\in C([0,1])$ such that $\Vert G \Vert_{\infty} \le 1$,   
 \begin{equation}
 \label{eq:tighness-hard}
 \begin{split}
 &\mathbb P_{n}\left(\sup_{0\leq t\leq T}\frac{1}{n}\sum_{x\in\Lambda_n}| G(\tfrac xn) \varphi_t(x)|\geq a \right)\\
 & \lesssim (1+ B^2)^{\tfrac{1}{p}}  e^{-\tfrac{n}{p} \big(\tfrac{Ba}{2} -\tfrac{\log 2}{2} \big)} n^{\tfrac{\theta}{p}} \, e^{\tfrac{B^2}{p}  \sum_{x\in \Lambda_n} G^2(\tfrac xn)} e^{\frac{\kappa_q n}{q}} + {\bf 1}_{ \max\{ |\Phi_\ell |, |\Phi_r |\}  \frac{1}{n}\sum_{x\in\Lambda_n}| G(\tfrac xn)|\geq a/2} .
 \end{split}
 \end{equation}
 In particular, 
    \begin{equation}
        \lim_{a\to\infty}\limsup_{n\to\infty}\frac 1n 
        \log \mathbb P_{n}\left(\sup_{0\leq t\leq T}\frac{1}{n}\sum_{x\in\Lambda_n}\big| \varphi_t(x)\big|\geq a\right)=-\infty.
    \end{equation}
\end{proposition}

\begin{proof}
In the proof it is convenient to use the notation $\mathbb P_{\mu_n}$ instead of $\mathbb P_{n}$ because we will consider different initial states. We start by estimating the probability under the stationary state {$\mu_{ss}^n$}. We denote $\bar\varphi (x) = \varphi(x) - \Phi_{ss}^n(x)$. First, note that for any $B>0$,
\begin{equation}
\label{eq:maximal_GL}
        \mathbb P_{\mu_{ss}^n}\left(\sup_{0\leq t\leq T}\frac{1}{n}\sum_{x\in\Lambda_n}| G(x/n) \bar \varphi_t(x)|\geq a\right)=\mathbb P_{\mu_{ss}^n}\left(\sup_{0\leq t\leq T}e^{B \sum_{x\in\Lambda_n}|G(x/n) \bar\varphi_t(x)|}\geq e^{n Ba}\right).
    \end{equation}
    Since, by Theorem \ref{thm:sector cond}, the generator $\mathcal L_n$ satisfies a sector condition with a constant $K_n$ which is polynomial in $n$, then we can apply the maximal inequality of Theorem \ref{thm:maximal_ineq} to the function $g_B(\varphi)=\exp\big( B\sum_{ x\in\Lambda_n}|G(x/n)\varphi (x)|\big)$.  Therefore 
    \begin{equation}
    \label{eq:maximal_GL_exp}
        \mathbb P_{\mu_{ss}^n}\left(\sup_{0\leq t\leq T}e^{B \sum_{x\in\Lambda_n}| G(x/n) \bar\varphi_t(x)|}\geq e^{n Ba }\right)\lesssim \frac{1}{e^{n Ba}}\sqrt{\langle g_B,g_B\rangle_{\mu_{ss}^n}+(K_n^2+1)T\;\mathfrak{D}_n(g_B)}
    \end{equation}
where the Dirichlet form ${\mathfrak D}^n (g_B)$ of $g_B$ is given in \eqref{eq:Dirichletform}. Now we compute each term on the right-hand side separately. First, since by Lemma \ref{lem:NESS}, $\mu_{ss}^n$ is a product of Gaussians, we have  
    \begin{equation}
        \langle g_B,g_B\rangle_{\mu_{ss}^n}=E_{\mu_{ss}^n}\Big[e^{2B\sum_{x\in\Lambda}|G(x/n)\bar\varphi (x)|}\Big] = \prod_{x\in\Lambda_n} E\big[e^{2B| G(x/n)\mathcal N_x|}\big],
    \end{equation}
    where $\mathcal N_x=\mathcal N(0,1)$. Observe that if $\lambda \ge 0$, then
    \begin{equation}
\begin{split}
\frac{1}{\sqrt{2\pi}} \int_{\mathbb R} e^{2 \lambda |y| - \tfrac{y^2}{2}}\, dy &= 2 e^{2\lambda^2} \Psi(2\lambda) \le 2 e^{2 \lambda^2}
\end{split}
\end{equation}
where $\Psi: a\in \mathbb R \to (2\pi)^{-1/2} \int_{-\infty}^a e^{-v^2/2} dv \in [0,1]$ is the cumulative distribution function of a standard Gaussian random variable. Applying this bound with $\lambda = B|G(x/n)|$, it follows that
  \begin{equation}
    \label{eq:l2g}
    \begin{split}
      \langle g_B,g_B\rangle_{\mu_{ss}^n} \le 2^n e^{2B^2 \sum_{x \in \Lambda_n} |G(x/n)|^2}.
    \end{split}
    \end{equation}

 We turn now to estimating the Dirichlet form. By \eqref{eq:Dirichletform},  $\mathfrak{D}_n(g_B)$ is given by 
    \begin{equation}
    \begin{split}
       n^\gamma  E_{\mu^n_{ss}}\Big[\big(\partial_{\varphi(1)}g_B(\varphi)\big)^2+\big(\partial_{\varphi(n-1)}g_B(\varphi)\big)^2+\sum_{x,y\in\Lambda_n}p(y-x)\Big((\partial_{\varphi(y)}-\partial_{\varphi(x)})g_B(\varphi)\Big)^2\Big].
    \end{split}
    \end{equation}
    Computing the derivatives, we get that for any $z\in\Lambda_n$, $\partial_{\varphi(z)}g_B(\varphi)= B \big\vert G \big(\tfrac{z}{n} \big) \big\vert g_B(\varphi)\text{sign}(\varphi(z))$. Therefore, since $\Vert G \Vert_\infty \le 1$, 
    \begin{equation}
        {\mathfrak D}^n(g_B)\leq 4 n^\gamma B^2 E_{\mu^n_{ss}}\Big[g_B(\varphi)^2\Big(1+\sum_{x,y\in\Lambda_n}p(y-x)\Big)\Big].
    \end{equation}
    Since $\sum_{x,y\in\Lambda_n}p(y-x)\lesssim n^{1-\gamma}$, we conclude that ${\mathfrak D}^n(g_B)\lesssim n^\gamma B^2 \langle g_B,g_B\rangle_{\mu^n_{ss}}.$ Now, recalling that $K_n$ is polynomial in $n$, equations \eqref{eq:maximal_GL}, \eqref{eq:maximal_GL_exp} and \eqref{eq:l2g} give that there exists a constant $\theta>0$ such that for any $n \ge 2$, any continuous function $G$ such that $\Vert G \Vert_{\infty} \le 1$, any $a,B>0$,
 \begin{equation}\label{eq:bound_TV_under_NESS}
       \mathbb P_{\mu_{ss}^n}\left(\sup_{0\leq t\leq T}\frac{1}{n}\sum_{x\in\Lambda_n}| G(\tfrac xn) \bar \varphi_t(x)|\geq a \right)\lesssim (1+ B^2) e^{-n \big(Ba -\tfrac{\log 2}{2} \big)} n^{\theta} \, e^{B^2 \sum_{x\in \Lambda_n} G^2(\tfrac xn)}.
    \end{equation}
Applying \eqref{eq:est_tight} there exists $p>1$ and a constant $\kappa>0$ such that    
\begin{equation}
       \mathbb P_{\mu_n}\left(\sup_{0\leq t\leq T}\frac{1}{n}\sum_{x\in\Lambda_n}| G(\tfrac xn) \bar \varphi_t(x)|\geq a \right)\lesssim (1+ B^2)^{1/p}  e^{-\tfrac{n}{p} \big(Ba -\tfrac{\log 2}{2} \big)} n^{\theta/p} \, e^{\tfrac{B^2}{p}  \sum_{x\in \Lambda_n} G^2(\tfrac{x}{n})} e^{\frac{\kappa_q n}{q}}.
    \end{equation}

We recall that $\sup_{n \ge 2}  \sup_{x \in \Lambda_n} \vert \Phi_{ss}^n (x) \vert \le  \max \{ |\Phi_\ell |, |\Phi_r |\} $. Hence, 
\begin{equation}\label{eq:tight_2}
 \begin{split}
 &\mathbb P_{\mu_n}\left(\sup_{0\leq t\leq T}\frac{1}{n}\sum_{x\in\Lambda_n}| G(\tfrac xn) \varphi_t(x)|\geq a \right) \\
& \le \mathbb P_{\mu_n}\left(\sup_{0\leq t\leq T}\frac{1}{n}\sum_{x\in\Lambda_n}| G(\tfrac xn) \bar\varphi_t(x)|\geq a/2 \right) + \mathbb P_{\mu_n}\left( \max\{ |\Phi_\ell |, |\Phi_r |\}  \frac{1}{n}\sum_{x\in\Lambda_n}| G(\tfrac xn)|\geq a/2 \right)\\
& \lesssim (1+ B^2)^{1/p}  e^{-\tfrac{n}{p} \big(Ba/2 -\tfrac{\log 2}{2} \big)} n^{\theta/p} \, e^{\tfrac{B^2}{p}  \sum_{x\in \Lambda_n} G^2(\tfrac xn)} e^{\frac{\kappa_q n}{q}} + {\bf 1}_{ \max\{ |\Phi_\ell |, |\Phi_r |\}  \frac{1}{n}\sum_{x\in\Lambda_n}| G(\tfrac xn)|\geq a/2} .
       \end{split}
    \end{equation}
This proves \eqref{eq:tighness-hard}. 

For the second claim of the lemma, we use the following classical inequality: for any positive sequences $(\alpha_n)_n, (\beta_n)_n, (\gamma_n)_n$ such that $\lim_{n\to \infty} \gamma_n=\infty$, we have
	\begin{equation}
	\label{sum_log_super}
	\varlimsup_{n\to+\infty} \frac{1}{\gamma_n}\log(\alpha_n+\beta_n)\;=\;\max\Big\{\varlimsup_{n\to+\infty} \frac{1}{\gamma_n}\log \alpha_n , \varlimsup_{n\to+\infty} \frac{1}{\gamma_n}\log \beta_n\Big\},
	\end{equation}  
    \end{proof}


We turn now to the proof of the second item in Theorem \ref{thm:tightness criteria}.

\begin{proposition}
\label{prop:tightness estimate II} 
For any $G\in C([0,1])$  and for all $\varepsilon>0$,  
    \begin{equation}
    \label{eq:Joaorio}
         \limsup_{\delta \to 0} \limsup_{n \to \infty} \mathbb Q_n^H \Bigg( \sup_{\substack{0 \le s \le t \le T\\ |t-s| \le \delta}}\,  \left\vert \langle \pi_t, G \rangle  -  \langle \pi_s, G \rangle \right\vert \ge \varepsilon \Bigg) =0.
\end{equation}
\end{proposition}

\begin{proof}
We shall first prove the estimates for test functions $G \in C^1([0,1])$ that vanish at the boundary. Later, we argue that by a density argument, it is possible to extend this to test functions in $G\in C([0,1])$. Moreover, from Lemma \ref{lem:super_exchange_measures}, it is sufficient to prove the result only for $\mathbb Q_n$, since the result for $\mathbb Q_n^H$ follows from this lemma. Hence, now, we assume $H=0$.

From \eqref{eq:dynkin martingale}, we have the following decomposition
\begin{equation}
\langle\pi^n_t,G_t\rangle-\langle \pi^n_0,G_0\rangle=V^n_t(G)+ R^{n,\ell}_t(G)+R^{n,r}_t(G)+M^{n}_t(G)
\end{equation}
where the bulk term is given by $V^n_t(G)=\int_0^t \langle\pi_s^n,\mathbb{L}_n^\gamma G\rangle\;ds$, the boundary terms are given by
\begin{equation}
\begin{split}
    R^{n,\ell}_t(G)=\int_0^t &n^{\gamma-1}(\Phi_\ell-\varphi_s(1))G\big(\tfrac{1}{n}\big)ds,\quad R^{n,r}_t(G)=\int_0^t n^{\gamma-1}(\Phi_r-\varphi_s(n-1))G\big(\tfrac{n-1}{n}\big)ds,
\end{split}
\end{equation}
and $M^{n}(G)$ is the martingale term \eqref{eq:dynkin}. For the bulk term, namely $V^n(G)$, the proof follows from the truncation arguments in \cite{GPV88} and \cite{quastel1995large}. Thus we only present the proof for the other terms. 
Let us now look at the  martingale term. From Chebishev's  and  Doob's inequalities:
     \begin{equation}
     \begin{split}
\mathbb P_{n}\Big(\sup_{\substack{0 \le s \le t \le T\\ |t-s| \le \delta}}|M_t^{n}(G)-M_s^{n}(G)|\geq \varepsilon \Big)&\leq  \frac{1}{\varepsilon^{2}}\mathbb E_{n}\Big[\sup_{\substack{0 \le s \le t \le T\\ |t-s| \le \delta}}|M_t^{n}(G)-M_s^{H,n}(G)|^2\Big]\\ &\leq 
\frac{1}{\varepsilon^{2}\delta}\mathbb E_{n}\Big[|M_\delta^{n}(G)|^2\Big]=\frac{1}{\varepsilon^{2}\delta}\mathbb E_{n}\Big[\langle M^{n}(G)\rangle_\delta\Big].
\end{split}
\end{equation}Now,  
     recall the expression for the quadratic variation \eqref{eq:QV}. Then, by  Taylor expansions, it follows that $ \langle M^{n}\rangle_\delta\lesssim \delta n^{-2}\Big(\| G'\|_{\infty}^2+n^{\gamma-2}\|G'\|_\infty\Big)$. To conclude we need to estimate also the boundary contribution, i.e., the terms $R_t^{n,\ell}(G)$ and $R_t^{n,r} (G)$. The argument is presented only for the left boundary, since the other is analogous.  We first observe that since $G$ vanishes at the boundary, by a simple argument, we have that 
\begin{equation}\begin{split}
\mathbb P_{n}\Bigg(\sup_{\substack{0 \le s \le t \le T\\ |t-s| \le \delta}}|R_t^{n,\ell}(G)-R_s^{n,\ell}(G)|\geq \varepsilon \Bigg)&\leq 
\mathbb P_{n}\Big(\delta n^{\gamma-2} \|G'\|_\infty \Phi_\ell\geq \frac{\varepsilon}{2}\Big)\\&+\mathbb P_{n}\Bigg(\sup_{\substack{0 \le s \le t \le T\\ |t-s| \le \delta}}\Big|\int_s^tn^{\gamma-2} \, n (G(\tfrac{1}{n}) - G(0))  \varphi_r(1)dr\Big|\geq \frac{\varepsilon}{2}\Bigg).
\end{split}
\end{equation} 
Since $\gamma<2$, the first term vanishes as $n\to+\infty$. For the second term we note that
\begin{equation}\begin{split}
\sup_{\substack{0 \le s \le t \le T\\ |t-s| \le \delta}}\Big|\int_s^tn^{\gamma-2} \, n (G(\tfrac{1}{n}) - G(0)) \varphi_r(1)dr\Big|&\leq n^{\gamma-2}\|G'\|_\infty \sup_{\substack{0 \le s \le t \le T\\ |t-s| \le \delta}}\int_s^t | \varphi_r(1)| dr\\&
\leq n^{\gamma-2}\|G'\|_\infty \int_0^t | \varphi_r(1)| dr\\&\leq n^{\gamma-2}\|G'\|_\infty T \sup_{r\in[0,T]}\sum_{x\in\Lambda_n} |\varphi_r(x)|. 
\end{split}
\end{equation} 
Therefore, 
\begin{equation}
\begin{split}
\mathbb P_{n}\Big(\sup_{\substack{0 \le s \le t \le T\\ |t-s| \le \delta}}|R_t^{n,\ell}(G)-R_s^{n,\ell}(G)|\geq \varepsilon \Big)&\leq  o(1)+\mathbb P_{n}\Big(\sup_{r\in[0,T]}\frac 1n\sum_{x\in\Lambda_n} |\varphi_r(x)| \geq \frac{\varepsilon}{2  n^{\gamma-1}\|G'\|_\infty T }\Big)\\&\lesssim o(1)+ e^{-Cn^{2-\gamma}}, 
\end{split}
\end{equation}  
where above we used Proposition \ref{prop: tightness estimate}. This ends the proof of \eqref{eq:Joaorio} for functions vanishing at the boundary.

Now, we claim that \eqref{eq:Joaorio} also holds if $G\in C([0,1])$. The argument is similar to the one presented above (4.10) in \cite{bernardin2017slow}. However, in our case, since the state-space $\Omega_n$ is not compact, the proof requires an extra estimate contained in Proposition \ref{prop: tightness estimate}. We need to prove, for $G\in C([0,1])$ and for any $\varepsilon>0$,  that
\begin{equation}\label{eq:extension tightness item 2}
 \limsup_{\delta \to 0} \limsup_{n \to \infty} \mathbb P_{n}\Bigg( \sup_{\substack{0 \le s \le t \le T,\\ |t-s| \le \delta}} \Big| \langle \pi^n_t, G\rangle  -\langle \pi^n_s, G\rangle\Big| >\varepsilon \Bigg) = 0.
\end{equation}
Modifying $\varepsilon$ if necessary, it is enough to prove this result for functions $G\in C([0,1])$ such that $\Vert G\Vert_{\infty} \le 1/2$. Let $(G_k)_{k\geq1}$ be a sequence of functions in $C^{1}([0,1])$ that vanish at the boundary, satisfying $\Vert G - G_k\Vert_{\infty} \le 1$ and $\limsup_{k \to \infty} \int_{0}^1 |G-G_k|^2 (u) du=0$. Then, the probability in \eqref{eq:extension tightness item 2} can be bounded from above by
\begin{equation}
     \mathbb P_{n}\Bigg( \sup_{\substack{0 \le s \le t \le T,\\ |t-s| \le \delta}} \Big| \langle \pi^n_t, G_k\rangle  -\langle \pi^n_s, G_k\rangle\Big| >\frac{\varepsilon}{2}\Bigg) + \mathbb P_n \Bigg( \sup_{\substack{0 \le s \le t \le T,\\ |t-s| \le \delta}} \Big| \langle \pi^n_t, (G-G_k)\rangle  -\langle \pi^n_s, (G-G_k)\rangle\Big| >\frac{\varepsilon}{2} \Bigg) .
\end{equation}
Since $G_k$ vanishes at the boundary,  the first probability vanishes as $n\to+\infty$ and then $\delta \to 0$. For the second, letting $H_k :=G-G_k$, we get that it is bounded by
\begin{equation}
\begin{split} 2 \mathbb P_n \left( \sup_{0 \le t \le T} \frac{1}{n}  \sum_{x \in \Lambda_n} \vert H_k (\tfrac xn) \vert \vert \varphi_t (x) \vert    >\frac{\varepsilon}{4} \right).
\end{split}
\end{equation}
By using \eqref{eq:tighness-hard}, since $\Vert H_k\Vert_{\infty} \le 1$, we can bound from above the expectation on the right-hand side of the previous inequality by a constant times
\begin{equation}
 \begin{split}
(1+ B^2)^{1/p}  e^{-\tfrac{n}{p} \big(\frac{B\varepsilon}{8} -\tfrac{\log 2}{2}- B^2  \frac{1}{n}\sum_{x\in \Lambda_n} H_k^2(\tfrac xn)-\frac{\kappa_q p}{q}\big)} n^{\theta/p}  + {\bf 1}_{ \max\{ |\Phi_\ell |, |\Phi_r |\}  \frac{1}{n}\sum_{x\in\Lambda_n}| H_k(\tfrac xn)|\geq \varepsilon /8} .
       \end{split}
    \end{equation}
By \eqref{sum_log_super}  
\begin{equation}
    \begin{split}
 &\limsup_{n \to \infty} \frac{1}{n} \log \mathbb P_n \left( \sup_{0 \le t \le T} \frac{1}{n}  \sum_{x \in \Lambda_n} \vert H_k (\tfrac xn) \vert \vert \varphi_t (x) \vert    >\frac{\varepsilon}{4} \right) \\
 & \le \max\left\{ - \frac{1}{p}\limsup_{n \to \infty} \left\{\frac{B\varepsilon}{8} -\tfrac{\log 2}{2}- B^2  \frac{1}{n}\sum_{x\in \Lambda_n} H_k^2(\tfrac xn)-\frac{\kappa_q p}{q}\right\}  \,  , \, \right.\\
 & \left. \quad \quad \quad \quad \quad \quad  \limsup_{n \to \infty} \frac{1}{n} \log {\bf 1}_{ \max\{ |\Phi_\ell |, |\Phi_r |\}  \frac{1}{n}\sum_{x\in\Lambda_n}| H_k(\tfrac xn)|\geq \varepsilon /8} \right\}\\
 &= \max \left\{ - \frac{B\varepsilon}{8p} + \frac{\log2}{2p} +\frac{\kappa_q}{q} +\frac{B^2}{p} \int_0^1 H_k^2 (u) du     \, , \, - \infty \cdot {\bf 1}_{ \max\{ |\Phi_\ell |, |\Phi_r |\}  \int_0^1 | H_k(u)| du < \varepsilon /8}\right\}.
    \end{split}
\end{equation}
For the last equality we use the fact that, since $H_k$ is continuous, $\lim_{n \to \infty} n^{-1} \sum_{x\in \Lambda_n} H_k^2(\tfrac xn) = \int_{0}^1 H^2_k (u) du $ and $\lim_{n \to \infty} n^{-1} \sum_{x\in \Lambda_n} \vert H_k (\tfrac xn)\vert = \int_{0}^1 \vert H_k (u) \vert du $. Recall now that we have chosen $(H_k)_{k \ge 1}$ such that $\lim_{k \to \infty}\int_{0}^1 H^2_k (u) du =0$ (and thus, $\lim_{k \to \infty}\int_{0}^1 |H_k(u)| du =0$). Therefore 
\begin{equation}
    \begin{split}
 &\limsup_{k \to \infty} \limsup_{n \to \infty} \frac{1}{n} \log \mathbb P_n \Big( \sup_{0 \le t \le T} \frac{1}{n}  \sum_{x \in \Lambda_n} \vert H_k (\tfrac xn) \vert \vert \varphi_t (x) \vert    >\frac{\varepsilon}{4} \Big) \le  - \frac{B\varepsilon}{8p} + \frac{\log2}{2p} +\frac{\kappa_q}{q}.
\end{split}
\end{equation}
Since $B>0$ is arbitrary, we conclude that 
\begin{equation}
    \begin{split}
 &\limsup_{k \to \infty} \limsup_{n \to \infty} \frac{1}{n} \log \mathbb P_n \Big( \sup_{0 \le t \le T} \frac{1}{n}  \sum_{x \in \Lambda_n} \vert H_k (\tfrac xn) \vert \vert \varphi_t (x) \vert    >\frac{\varepsilon}{4} \Big)=-\infty.
\end{split}
\end{equation}

\end{proof}

\subsection{Characterization of limit points} We show in this section the following result. Recall the definition of $\mathcal C^{ac}$ given in Definition  \ref{def:space_C_ab}. Let us fix $H\in C_c^{1, 2} (\Omega_T)$ for the rest of this section.

\begin{proposition}
\label{prop:absolute_continuity}  
All limiting points $\mathbb Q^H$ of the sequence $(\mathbb Q^H_n)_{n\ge 2}$ satisfy 
$
\mathbb Q^H ( \mathcal C^{ac}) = 1$ and
\begin{equation}
\begin{split}
\mathbb Q^H\big(\pi_\cdot \; ; \; F_H(t,\Phi,G,g)=0,\;\forall\;t\in[0,T],\;\forall\; G\in C^{1,\infty}_c(\Omega_T)\big)=1.
\end{split}
\end{equation}
\end{proposition}

The proof of this proposition will be a consequence of several lemmas. In Lemma \ref{lem:acsniek} and Lemma \ref{lem:energyestimate2} we prove  that every limit point of $\mathbb Q^H$ is concentrated on trajectories of measures that are absolutely continuous with respect to the Lebesgue measure whose densities belong to $L^2([0,T],\mathcal H^{\gamma/2})$. In fact, since we will need later, we prove these lemmata at the super-exponential level (Lemma \ref{lem:acL2-superexp} and Lemma \ref{lem:EGPI}). In Lemma \ref{lem:wfsnieck} we show that the densities satisfy the weak formulation of the hydrodynamic equation \eqref{eq:Dirichlet Equation2}  and in Lemma \ref{lem:bcsnieck} we prove that they have the desired boundary conditions, concluding the proof of Proposition \ref{prop:absolute_continuity}.

\begin{lemma}
\label{lem:acL2-superexp}
For any $G \in C^{1,\infty} (\Omega_T)$, consider the continuous functional ${\mathcal F}_G$ on $ C([0,T], \mathcal M)$ defined by 
\begin{equation}
\begin{split}
\forall \pi \in C([0,T], \mathcal M), \quad {\mathcal F}_G (\pi) = \int_0^T \Big\{ \langle \pi_t -\pi_{ss} , G_t \rangle - \tfrac{T}{2} \Vert G_t\Vert^2  \Big\} dt.
\end{split}
\end{equation}
Then, there exists a constant $C>0$ such that, for any family $(G_j)_{j\geq 1}$ of functions in $C^{0,\infty} (\Omega_T)$, for any $\ell>0$ and for any $k\geq 1$, we have that 
\begin{equation}
\label{eq:Jeramai0}
\begin{split}
\limsup_{n \to \infty} \frac{1}{n} \log \mathbb Q^H_n \left( \max_{1 \le j \le k} {\mathcal F}_{G_j} (\pi) \ge \ell \right) \le - C\ell +1/C. 
\end{split}
\end{equation}

\end{lemma}

\begin{proof}
We start by establishing \eqref{eq:Jeramai0} for the process starting from the NESS and for $H=0$.  Using the union bound and the classical inequality \eqref{sum_log_super}, it is sufficient to prove the lemma for $k=1$. To simplify notation, we denote $G_1$ by $G$. Recall \eqref{eq:empirical_deterministic_stationary} and \eqref{eq:convPhissn-Phiss}, i.e. that $(\pi_{ss}^n)_{n \ge 2}$ converges weakly to $\pi_{ss}$. Hence  $ \lim_{n \to \infty} \varepsilon_n (G) =0$ where
\begin{equation}
\begin{split}
\varepsilon_n (G) = \int_0^T \left\vert \langle \pi_{ss}^n, G_t \rangle - \langle \pi_{ss} , G_t \rangle \right\vert \, dt.
\end{split}
\end{equation}

We fix any $\ell>0$. By the union bound, Chebichev's and Jensen's inequalities
\begin{equation}
\begin{split}
&\mathbb P_{\mu_{ss}^n} \left[ \int_0^T \big\{ \langle \pi_t^n -\pi_{ss} , G_t \rangle - \tfrac{T}{2} \Vert G_t\Vert^2_2  \big\} dt \ge 2\ell \right]\\ 
& \le \mathbb P_{\mu_{ss}^n} \left[ \int_0^T \big\{ \langle \pi_t^n -\pi^n_{ss} , G_t \rangle - \tfrac{T}{2} \Vert G_t\Vert^2_2  \big\} dt \ge \ell \right] + \mathbb P_{\mu_{ss}^n} [ \varepsilon_n(G) \ge \ell]\\ 
&\le  e^{-\ell  n} \mathbb E_{\mu_{ss}^n} \left[ \exp\Big(n  \int_0^T \big\{ \langle \pi_t^n -\pi^n_{ss} , G_t \rangle - \tfrac{T}{2} \Vert G_t\Vert^2_2  \big\} dt\Big)  \right] + \mathbb P_{\mu_{ss}^n} [ \varepsilon_n(G) \ge \ell]\\ 
& \le e^{-\ell  n} \frac1T \int_0^T\mathbb E_{\mu_{ss}^n} \Big[ \exp\Big(n  T  \big\{ \langle \pi_t^n -\pi^n_{ss} , G_t \rangle - \tfrac{T}{2} \Vert G_t \Vert^2_2  \big\}\Big)  \Big] dt + \mathbb P_{\mu_{ss}^n} [ \varepsilon_n(G) \ge \ell].
\end{split}
\end{equation}
Since $\mu_{ss}^n$ is stationary and under $\mu_{ss}^n$, $\varphi(x) \sim \mathcal N (\Phi_{ss}^n (x), 1)$, we have that
\begin{equation}
\begin{split}
& \mathbb E_{\mu_{ss}^n}\Big[ \exp\Big(n  T  \Big\{ \langle \pi^n -\pi^n_{ss} , G_t  \rangle - \tfrac{T}{2} \Vert G_t\Vert_{2}^2  \Big\}\Big) \Big]= \exp\left\{ \frac{nT^2}{2} \Big[ \Vert G_t \Vert_{2,n}^2 - \Vert G_t \Vert^2_2 \Big]\right\} . 
\end{split}
\end{equation}
Hence, by \eqref{sum_log_super}, we get
\begin{equation}
\begin{split}
&\frac{1}{n} \log \mathbb P_{\mu_{ss}^n} \left[ \int_0^T  \big\{ \langle \pi_t^n -\pi_{ss} , G \rangle - \tfrac{T}{2} \Vert G\Vert^2_2  \big\} dt \ge 2\ell \right] \\
&\le \max\left\{ - \ell  +\frac{T^2}{2} \sup_{0\le t \le T} \Big\vert \Vert G_t \Vert_{2,n}^2 - \Vert G_t \Vert^2_2 \Big\vert  \ , \  \frac{1}{n} \log \mathbb P_{\mu_{ss}^n} [ \varepsilon_n(G) \ge \ell] \right\}. 
\end{split}
\end{equation}
By sending $n$ to infinity we conclude that
\begin{equation}
\begin{split}
\limsup_{n \to \infty} \frac{1}{n} \log \mathbb P_{\mu_{ss}^n} \left[ \int_0^T \big\{ \langle \pi_t^n -\pi_{ss} , G_t \rangle - \tfrac{T}{2} \Vert G_t \Vert^2_2  \big\} dt \ge 2\ell \right] \le -\ell,
\end{split}
\end{equation}
so that \eqref{eq:Jeramai0} is satisfied for the process starting from $\mu_{ss}^n$. By using Lemma \ref{lem:mussnmun}, we conclude that the same estimate (by modifying the constant $C$) holds for $\mathbb P_n$. Thanks to \eqref{eq:Glucksmann}, we can extend it now for $H \ne 0$ by replacing $\mathbb Q_n$ by $\mathbb Q_n^H$.

\end{proof}

\begin{lemma}
\label{lem:acsniek} 
The probability measure $\mathbb Q^H$ is concentrated on paths $\pi \in C([0,T], \mathcal M)$ satisfying $\pi_t(du)= \Phi_t (u) du$, with $\Phi \in L^2 (\Omega_T)$.
\end{lemma}

\begin{proof}
Let $(G_j)_{j\ge 1}$ be a  sequence of elements in $C^{1,\infty} (\Omega_T)$ which is dense in $L^2 (\Omega_T)$.  From  \eqref{eq:Jeramai0}, for any $k \ge 1$ and for any $\varepsilon>0$, there exists $n_{\varepsilon, k} \ge 1$ such that for any $n \ge n_{\varepsilon,k}$, 
\begin{equation}
\begin{split}
\frac{1}{n} \log \mathbb Q^H_n \left( \max_{1 \le j \le k} {\mathcal F}_{G_j} (\pi) \ge \ell \right) \le - C\ell +1/C + \varepsilon. 
\end{split}
\end{equation}
Choosing  $\ell=4 C^{-2}$ and $\varepsilon = C\ell/2$ so that $- C\ell +1/C + \varepsilon = -1/C$. Then, for any $k \ge 1$ that 
\begin{equation}
\begin{split}
\liminf_{n \to \infty} \mathbb Q^H_n \left( \max_{1 \le j \le k} {\mathcal F}_{G_j} (\pi) > 4/C^2 \right)  \le \limsup_{n \to \infty} \mathbb Q^H_n \left( \max_{1 \le j \le k} {\mathcal F}_{G_j} (\pi) \ge 4/C^2 \right) \le \limsup_{n \to \infty} e^{-n/C} =0. 
\end{split}
\end{equation}
Since $\max_{1 \le j \le k} {\mathcal F}_{G_j}$ is continuous, $\{ \max_{1 \le j \le k} {\mathcal F}_{G_j} (\pi) > 4/C^2\}$ is an open set, and since $(\mathbb Q_n^H)_{n \ge 2}$ converges to $\mathbb Q^H$, Portmanteau's theorem gives that, for any $k \ge 1$,
\begin{equation}
\begin{split}
\mathbb Q^H \left( \max_{1 \le j \le k} {\mathcal F}_{G_j} (\pi) > 4/C^2 \right) =0 \ .
\end{split}
\end{equation}
Since the events $\{\max_{1 \le j \le k} {\mathcal F}_{G_j} (\pi) > 4/C^2 \}$ are increasing in $k$ and converge to $\{\sup_{j \ge 1} \mathcal F_{G_j} (\pi) > 4/C^2 \}$, then, $\mathbb Q^H$ a.s., $\sup_{j \ge 1} {\mathcal F}_{G_j} (\pi) = \sup_{G \in L^2 (\Omega_T)} \mathcal F_G (\pi) \le 4C^{-2}.$ This implies that $\pi-\pi_{ss} \in L^2 (\Omega_T)$ and, since by \eqref{eq:boundphissn} we have $\pi_{ss} \in L^{2} (\Omega_T)$, this implies $\pi \in L^2 (\Omega_T)$.

\end{proof}

We now state an energy estimate that allows us to conclude that  $\Phi\in L^2([0,T],\mathcal H^{\gamma/2})$. 
For any $G \in C_c^{0,\infty}(\Omega_T)$, we define the continuous functional $\mathcal E_G$ on $C([0,T], \mathcal M)$ by 
\begin{equation}
\label{eq:EGPI}
\begin{split}
\forall \pi \in C([0,T], \mathcal M), \quad \mathcal  E_G(\pi) = \int_{0}^T \left\{  2 \langle \mathbb L^\gamma G_t, \pi_t \rangle \  - \  2 \Vert G_t \Vert_{\gamma/2}^2  \right\}  \;dt.
\end{split}
\end{equation}
Recall \eqref{eq:fraclapC2} and the discussion above this identity. We have then that $ \sup_{0 \le t \le T} \vert \langle \mathbb L^\gamma G_t, \pi_t \rangle \vert < \infty$ since $\sup_{0 \le t \le T} \Vert \mathbb L^\gamma G_t \Vert_{\infty} < \infty$ and $\sup_{0 \le t \le T} \Vert \pi_t\Vert_{TV} < \infty$.

The energy functional $\mathcal  E:C([0,T],\mathcal M)\to \mathbb R\cup\{+\infty\}$ is then defined as
\begin{equation}
\label{eq:energy}
\mathcal  E(\pi)=\sup_{G \in C_c^{0,\infty} (\Omega_T)}\ \mathcal E_G(\pi - \pi_{ss}),
\end{equation}
and it  is convex and lower semi-continuous. If $\pi\notin\mathcal C^{ac}$, then the energy is infinite. For $\pi\in\mathcal C^{ac}$, there exists a density $\Phi\in L^2([0,T],\mathcal H^{\gamma/2})$.


\begin{lemma}
[Energy Estimate LDP level]
\label{lem:EGPI}
There exists $C:=C(\Phi_\ell,\Phi_r)>0$ such that for any family $(G_j)_{j\geq 1}$ of functions in $C_c^{0,\infty} (\Omega_T)$, for any $\ell \ge 1$ and for any $k\geq 1$,
    \begin{equation}
       \label{eq:Jeremai1}
       \limsup_{n\to\infty}\frac{1}{n}\log\mathbb Q^H_n\Big(\max_{1 \le j \le k} \mathcal E_{G_j} (\pi - \pi_{ss} )\geq \ell\Big)\leq -C\ell+C^{-1}.
    \end{equation}
\end{lemma}

\begin{proof}
As in the proof of  Lemma \ref{lem:acL2-superexp}, by invoking  Lemma \ref{lem:mussnmun} and \eqref{eq:Glucksmann}, it is sufficient to establish \eqref{eq:Jeremai1} with the process starting from the NESS and with $H=0$.  Using the union bound and the classical inequality \eqref{sum_log_super}, it is moreover sufficient to prove the lemma for $k=1$. To simplify notations, we denote $G_1$ by $G$.

Let $\ell \ge 1$ be fixed. To simplify expressions we denote $\overline\pi = \pi -\pi_{ss}^n$ and $\overline \varphi(x) = \varphi(x) - \Phi_{ss}^n (x)$.
Observe first that since we have \eqref{eq:convPhissn-Phiss}, i.e., that $(\pi_{ss}^n)_{n \ge 2}$ converges weakly to $\pi_{ss}$, we can easily show that
\begin{equation}
\begin{split}
\lim_{n \to \infty} \int_0^ T \langle \mathbb L^\gamma G_t , \pi_{ss} - \pi_{ss}^n \rangle dt =0 .
\end{split}
\end{equation}
Hence, as in the proof of Lemma \ref{lem:acL2-superexp}, it is sufficient to prove that 
 \begin{equation}
       \label{eq:Jeremai2}
       \limsup_{n\to\infty}\frac{1}{n}\log\mathbb P_{\mu_{ss}^n} \Big(\mathcal E_{G} (\overline\pi^n)\geq \ell\Big)\leq -C\ell+C^{-1}.
    \end{equation}
Recall the bound \eqref{eq:tighness-hard}. Since, by \eqref{eq:boundprofileness}, $\sup_{x \in \Lambda_n} \vert \Phi_{ss}^n (x) \vert$ is uniformly bounded in $n$, the bound \eqref{eq:tighness-hard} also holds by replacing $\varphi$ by $\bar \varphi$. Observe that
\begin{equation}
\begin{split}
\mathbb P_{\mu_{ss}^n} \Big(\mathcal E_{G} (\bar\pi^n)\geq \ell\Big) & \le \mathbb P_{\mu_{ss}^n} \left(\mathcal E_{G} (\bar\pi^n) \ge \ell, \int_0^T \Vert {\bar \pi}^n_t \Vert_{TV} dt  \le  \ell/2\right) + \mathbb P_{{\mu_{ss}^n}} \left( \int_0^T \Vert {\bar \pi}^n_t \Vert_{TV} dt \ge \ell/2 \right)\\
&\le \mathbb  P_{\mu_{ss}^n} \left(\mathcal E_{G} (\bar\pi^n) - \int_0^T \Vert {\bar \pi}^n_t \Vert_{TV} dt  \geq \ell/2\right) + \mathbb P_{{\mu_{ss}^n}} \left( \int_0^T \Vert {\bar \pi}^n_t \Vert_{TV} dt \ge \ell/2 \right).
\end{split}
\end{equation}
By \eqref{eq:bound_TV_under_NESS}, applied with the test function $G=1$, there exists a positive constant $\theta$, such that for any $\ell>0$ and $B>0$, the second term in the last display is bounded by 
\begin{equation}
    \mathbb P_{{\mu_{ss}^n}} \left( \int_0^T \Vert {\bar \pi}^n_t \Vert_{TV} dt \ge \ell/2 \right)\leq (1+B^2)n^\theta e^{-n(\frac{B\ell}{2T}-\frac{\log 2}{2})+B^2 n}.
\end{equation}
Hence, by using \eqref{sum_log_super}, we conclude that 
\begin{equation}
    \limsup_{n\to\infty}\frac{1}{n}\log  \mathbb P_{{\mu_{ss}^n}} \left( \int_0^T \Vert {\bar \pi}^n_t \Vert_{TV} dt \ge \ell/2 \right)\leq -\frac{B\ell}{2T}+\frac{\log 2}{2} +B^2\lesssim-C_0(\ell+1)
\end{equation}
for some $C_0:=C_0(B,T)$. To conclude the proof, it is therefore sufficient to prove that there exists a constant $C>0$ independent of $G$  such that
\begin{equation}
\label{eq:JeremaiDead}
\begin{split}
\limsup_{n \to \infty} \frac{1}{n} \log \mathbb  P_{\mu_{ss}^n} \left(\mathcal E_{G} (\bar\pi^n) - \int_0^T \Vert {\bar \pi}^n_t \Vert_{TV} dt  \geq \ell/2\right)  \le -C\ell + C^{-1}. 
\end{split}
\end{equation}

For any $t \ge 0$, let  $V_n (t, \cdot): \Omega_n \to \mathbb R$ be the potential defined, for any $\varphi \in \Omega_n$, by 
\begin{equation}
\begin{split}
V_n (t, \varphi) & =  n \left\{ 2 \langle \mathbb L^\gamma G_t, {\bar \pi}^n \rangle  - 2 \Vert G_t \Vert_{\gamma/2}^2   - \Vert \bar \pi \Vert_{TV}  \right\},
\end{split}
\end{equation}
and let 
\begin{equation}
\label{eq:Gamma_n}
\begin{split}
\Gamma_n (t)  = \sup_{f} \left\{ \int_{\Omega_n}  V_n (t, \varphi) f^2(\varphi) d\mu_{ss}^n (\varphi) \, - \, {\mathfrak D}^n ( f) \right\},
\end{split}
\end{equation}
where the supremum is taken over functions $f$ such that $\int_{\Omega_n} f^2 d \mu_{ss}^n=1$ and the Dirichlet form ${\mathfrak D}^n ( f) = -\langle  f, \mathcal S_n f \rangle_{\mu_{ss}^n}$ of $f$ is given by \eqref{eq:Dirichletform} with $\mathcal S_n$ the symmetric part $\tfrac{1}{2} (\mathcal L_n + \mathcal L_n^*)$ in $L^2 (\mu_{ss}^n)$ given explicitly by \eqref{eq:S_n-expression}.
By Chebichev's inequality and Feynman-Kac's formula it holds
\begin{equation}
\label{eq:JeremaiAlive}
\begin{split}
&\mathbb  P_{\mu_{ss}^n} \left(\mathcal E_{G} (\bar\pi^n) - \int_0^T \Vert {\bar \pi}^n_t \Vert_{TV} dt  \geq \ell/2\right) 
\le e^{-n \ell/2} \exp \left\{ \int_0^T \Gamma_n (t)  dt \right\}.
\end{split}
\end{equation}
Now, note that
\begin{equation}
\begin{split}
\int_{\Omega_n} V_n (t, \varphi) f^2(\varphi) d\mu_{ss}^n (\varphi)  &=  2  \sum_{x \in \Lambda_n} ({\mathbb L}_n^\gamma G_t) \big( \tfrac{x}{n} \big) \,    \int_{\Omega_n} \bar\varphi (x)  f^2 (\varphi) d\mu_{ss}^n (\varphi) \quad \quad \label{eq:1_new}\\
& + 2 \sum_{x \in \Lambda_n} \Big[({\mathbb L}^\gamma G_t) \big( \tfrac{x}{n}\big) - ({\mathbb L}_n^\gamma G_t) \big( \tfrac{x}{n} \big) \Big]  \,    \int_{\Omega_n} \bar\varphi (x)  f^2 (\varphi) d\mu_{ss}^n (\varphi) \\
& -  2n \Vert G_t \Vert_{\gamma/2}^2  \ - \ \sum_{x \in \Lambda_n} \int_{\Omega_n} \vert \bar\varphi (x) \vert  f^2(\varphi) d\mu_{ss}^n (\varphi).
\end{split}
\end{equation}

By the explicit formula \eqref{NESS} for $\mu_{ss}^n$, two changes of variables and integrating by parts, we can rewrite the first term on the right-hand side of the previous display as
\begin{equation}
\begin{split}
\eqref{eq:1_new}
&= n^\gamma \sum_{x,y \in \Lambda_n} p \big( y-x  \big) \Big[G_t \big (\tfrac{y}{n}\big) -G_t \big( \tfrac{x}{n} \big) \Big]  \,    \int_{\Omega_n} [\bar\varphi (x)   - \bar \varphi (y)] f^2 (\varphi) d\mu_{ss}^n (\varphi) \\
&= n^\gamma \sum_{x,y \in \Lambda_n} p \big( y-x  \big) \Big[G_t \big (\tfrac{y}{n}\big) -G_t \big( \tfrac{x}{n} \big) \Big]  \,    \int_{\Omega_n} \Big ( \partial_{\varphi(y)} f^2 - \partial_{\varphi(x)} f^2  \Big)  (\varphi) d\mu_{ss}^n (\varphi)\\
&= 2n^\gamma \sum_{x,y \in \Lambda_n} p \big( y-x  \big) \Big[G_t \big (\tfrac{y}{n}\big) -G_t \big( \tfrac{x}{n} \big) \Big]  \,    \int_{\Omega_n} \Big ( \partial_{\varphi(y)} {f} - \partial_{\varphi(x)} f  \Big)  (\varphi)  \, {f} (\varphi) \, d\mu_{ss}^n (\varphi).
\end{split}
\end{equation}
By Cauchy-Schwarz's inequality this term is bounded by
\begin{equation}
\begin{split}
&2n^\gamma \sum_{x,y \in \Lambda_n} p \big( y-x  \big) \Big[G_t \big (\tfrac{y}{n}\big) -G_t \big( \tfrac{x}{n} \big) \Big]  \,    \int_{\Omega_n} \bar\varphi (x)  f^2 (\varphi) d\mu_{ss}^n (\varphi) \\
&\le 2 \sqrt{ n^\gamma \sum_{x,y \in \Lambda_n} p \big( y-x  \big) \Big[G_t \big (\tfrac{x}{n}\big) -G_t \big( \tfrac{y}{n} \big) \Big]^2 }\, \sqrt{{\mathfrak D}^n ( f)  }.
\end{split}
\end{equation}
Recall the formula \eqref{eq:DiscreteSobolevNorm} for the discrete Sobolev norm. By using that for any $a \ge 0$, $\sup_{x \ge 0} \{ 2ax -x^2\} =a^2$, we deduce that
\begin{equation}
\begin{split}
\int_{\Omega_n} V_n (t, \varphi) f^2(\varphi) d\mu_{ss}^n (\varphi) -  {\mathfrak D}^n ( f) &\le 2 \sqrt{ n^\gamma\sum_{x,y \in \Lambda_n} p \big( y-x  \big) \Big[G_t \big (\tfrac{x}{n}\big) -G_t \big( \tfrac{y}{n} \big) \Big]^2 }\, \sqrt{{\mathfrak D}^n ( f)  } \\ &-  {\mathfrak D}^n ( f) 
 - 2n \Vert G_t \Vert_{\gamma/2}^2   - \sum_{x \in \Lambda_n} \int_{\Omega_n} \vert \bar\varphi (x) \vert  f^2(\varphi) d\mu_{ss}^n (\varphi) \\
& + 2 \sum_{x \in \Lambda_n}  \left[ (\mathbb L_n^\gamma G_t ) \big( \tfrac{x}{n}\big) - (\mathbb L^\gamma G_t ) \big( \tfrac{x}{n} \big) \right] \,    \int_{\Omega_n} \bar\varphi (x)  f^2 (\varphi) d\mu_{ss}^n (\varphi) \\
& \le 2 n \left[\Vert G_t \Vert_{n,\gamma/2}^2  - \Vert G_t \Vert_{\gamma/2}^2 \right] \\
& + \left[2 \sup_{x \in \Lambda_n} \left\vert (\mathbb L_n^\gamma - \mathbb L^\gamma )G_t  (\tfrac{x}{n}) \right\vert -1\right]    \,  \sum_{x \in \Lambda_n}  \int_{\Omega_n} \vert \bar \varphi (x) \vert f^2(\varphi) d \mu_{ss}^n (\varphi).
\end{split}
\end{equation}
Observe now that, uniformly in $t \in [0,T]$,
\begin{equation}
\begin{split}
\limsup_{n \to \infty}  & \Bigg\{ 2 \left[ \Vert G_t \Vert_{n,\gamma/2}^2 - \Vert G_t \Vert_{\gamma/2}^2 \right]  \\
 & + \left[2 \sup_{x \in \Lambda_n} \left\vert (\mathbb L_n^\gamma - \mathbb L^\gamma )G_t  (\tfrac{x}{n}) \right\vert -1\right]  \frac{1}{n} \sum_{x \in \Lambda_n}  \int_{\Omega_n} \vert \bar \varphi (x) \vert f^2(\varphi) d \mu_{ss}^n (\varphi) \Bigg\} \leq 0,
\end{split}
\end{equation}
because, uniformly in $t \in [0,T]$, $\lim_{n\to\infty}\sup_{x \in \Lambda_n} \left\vert (\mathbb L_n^\gamma G_t ) \big( \tfrac{x}{n} \big) - (\mathbb L^\gamma G_t ) \big( \tfrac{x}{n}\big)  \right\vert  =0$ and \newline   $\lim_{n\to \infty} \Vert G_t \Vert_{n,\gamma/2} = \Vert G_t \Vert_{\gamma/2}.$
The first limit is trivial and the second one can be established as in \cite{BJ17}.
Hence, recalling \eqref{eq:Gamma_n}, we conclude that 
\begin{equation}
\begin{split}
\limsup_{n \to \infty} \frac{1}{n} \int_0^T \Gamma_n (t) dt =0.
\end{split}
\end{equation}
Plugging this result in \eqref{eq:JeremaiAlive}, we conclude that \eqref{eq:JeremaiDead} holds and this completes the proof.
\end{proof}

We note that as a consequence of the previous lemma, the classical energy estimate can be obtained.

\begin{lemma}
\label{lem:energyestimate2}
The probability measure $\mathbb Q^H$ is concentrated on trajectories $\pi_t(du)=\Phi_t(u)du$ such that the density satisfies 
    \begin{equation}\label{eq:classic energy estimate}
        \mathbb E_{\mathbb Q^H}\Big[\int_0^T\|\Phi_t\|_{\mathcal H^{\gamma/2}}^2 \ dt\Big]\lesssim T.
    \end{equation}
\end{lemma}
    
\begin{proof}
This lemma can be deduced from Lemma \ref{lem:EGPI} as we did to deduce Lemma \ref{lem:acsniek}  from Lemma \ref{lem:acL2-superexp}. Alternatively, it can be obtained by following the same steps as \cite[Theorem 3.2]{BGJ21} as application of Lemmas \ref{lem:Patricia0} and \ref{lem:Patricia1} and the computations in Section \ref{subsec:dirichlet forms}.

\end{proof}

\begin{lemma}
\label{lem:wfsnieck}
 We have that    
    $    \mathbb Q^H\big(F_H(t,\Phi,G,g)=0,\;\forall\;t\in[0,T],\;\forall\; G\in C^{1,\infty}_c(\Omega_T)\big)=1.
   $
\end{lemma}

\begin{proof}
Since the macroscopic equation is linear, the proof follows directly from Lemma \ref{lem:acL2-superexp}; and the convergence of the discrete fractional operator, proved in \cite[Lemma 5.1]{BCGS2023spa}.
\end{proof}

\begin{lemma}
\label{lem:bcsnieck}
We have that
\begin{equation}
\begin{split}
\mathbb Q^H \left(\pi_t (du) =\Phi_t (u) du \ \ \text{and} \ \  \Phi_t(0)=\Phi_\ell, \ \   \Phi_t(1)=\Phi_r \ \ \text{for a.e.}\ \  t \in[0,T]\right) =1.
\end{split}
\end{equation}
\end{lemma}

\begin{proof}
Under $\mathbb Q^H$, we denote for any $t \in [0,T]$ the density of $\pi_t$ by $\Phi_t$.  We now prove that $\Phi$ has the desired boundary conditions. We present the proof for the left boundary but, for the right, it is analogous. For any $1\le \ell \le n-2$, $x\in \Lambda_n$ such that $x\in\{1,\dots,n-1-\ell\}$ (resp.  $ x\in\{n-1-\ell, \dots, n-1\}$) and any configuration $\varphi \in \Omega_n$,  the empirical averages are defined by
\begin{equation}
\label{eq:emp_average}
	\overrightarrow{\varphi}^{\ell}(x)=\frac{1}{\ell}\sum_{z=x+1}^{x+\ell}\varphi (z)\quad \Big(\textrm{resp.} \quad 	\overleftarrow{\varphi}^{\ell}(x)=\frac{1}{\ell}\sum_{z=x-\ell}^{x-1}\varphi (z)\Big).
\end{equation}

Recall Proposition \ref{prop:charact dirichlet bd cond super exp-0}. Note now that 
\begin{equation}
\label{eq:fix_bound}
\begin{split}
    \mathbb Q^H\Big(\Big|\int_0^t(\Phi_s(0)-\Phi_\ell)\, ds\Big|>\delta\Big)    
    \leq &\mathbb Q^H\Big(\Big|\int_0^t\frac {1}{\varepsilon} \int_0^\varepsilon \Big(\Phi_s(0)-\Phi_s(u)\Big) du\, ds\Big|>\frac {\delta}{2}\Big)\\
   +& \mathbb Q^H\Big(\Big|\int_0^t\frac {1}{\varepsilon} \int_0^\varepsilon \Big(\Phi_s(u) du - \Phi_\ell\Big)\, ds\Big|>\frac {\delta}{2}\Big).
    \end{split}
\end{equation} 
Now we focus on the last line of last display, and since $\pi_t(du)=\Phi_t(u)du$, it can be written as 
\begin{equation}\begin{split}
  \mathbb Q^H\Big(\Big|\int_0^t\Big(\langle \pi_s,\iota_\varepsilon^0\rangle -\Phi_\ell\Big)\, ds\Big|>\frac {\delta}{2}\Big)
    \end{split}
\end{equation} 
where $\iota_\varepsilon^0= \varepsilon^{-1} {\bf 1}_{[0,\varepsilon)}$. We note that the set appearing in the probability in last display is not an open set but by approximating the function $\iota_{\varepsilon}^0$ by a continuous function, we can then use Portmanteau's theorem to bound the last display by
\begin{equation}\begin{split}
  \liminf_{n\to+\infty}\mathbb Q_n^H\Big(\Big|&\int_0^t\Big(\langle \pi_s,\iota_\varepsilon^0\rangle -\Phi_\ell\Big)\, ds\Big|>\frac {\delta}{2}\Big)\leq  \liminf_{n\to+\infty}\frac{2}{\delta}\mathbb E_{n}^H\Big[\Big|\int_0^t \Big(\overrightarrow{\varphi}_s^{\,\varepsilon  n}(1)-\Phi_\ell\Big)ds\Big|\Big]
    \end{split}
\end{equation}
Now we invoke Proposition \ref{prop:charact dirichlet bd cond super exp-0}. 
Finally, we analyze the first term on the right-hand side of the inequality in \eqref{eq:fix_bound}. By Lemma \ref{lem:EGPI}, we know that $\Phi \in L^2 ([0,T], \mathcal H^{\gamma/2})$, so that $\Phi_s(\cdot)$ satisfies item a) of Definition \ref{Def. Dirichlet Condition2}. Using the same arguments above we have that 
\begin{equation}\begin{split}
\Big|\int_0^t\frac {1}{\varepsilon} \int_0^\varepsilon \Big(\Phi_s(0)-\Phi_s(u)\Big) du\, ds\Big|&\leq \int_0^t\frac {1}{\varepsilon} \int_0^\varepsilon \Big|\Phi_s(0)-\Phi_s(u)\Big| du\, ds\leq \int_0^t\frac {1}{\varepsilon} \int_0^\varepsilon C|u|^{\frac{\gamma-1}{2}} du\, ds.
    \end{split}
\end{equation}
In last inequality we used Theorem 8.2 of \cite{DNPV} that guarantees that $\Phi_s(\cdot)$ is a $\frac{\gamma-1}{2}$-H\"older continuous function and the H\"older constant does not depend on the time parameter $s$. A simple computations shows that last display is bounded from above by $CT\varepsilon^{\frac{\gamma-1}{2}}$, and since $\gamma>1$, it  vanishes as $\varepsilon\to 0$. This ends the proof.
\end{proof}


\section{Hydrostatic Limit}  
\label{sec:hydrostatic}

In this section we prove Lemma \ref{lem:NESS}, Theorem \ref{theo:stationary} and Theorem \ref{thm:hydro static limit}.

\subsection{Proof of Lemma \ref{lem:NESS}} 
\label{subsubsec:theo-stationarity0}

Observe  that in the end of Appendix \ref{app_adjoint}, we prove that the product measure $\mu_{ss}^n$ is invariant for the process generated by $\mathcal L_n$. Now, for any $x \in \Lambda_n$ we have that 
\begin{equation}
\begin{split}
    \mathcal L_n \varphi (x) & = \frac{1}{2} \sum_{y,z \in \Lambda_n} p(z-y) (\varphi(y) -\varphi (z)) ( \delta_{z,x} - \delta_{y,x}) + \delta_{x,1} (\Phi_\ell - \varphi (1)) + \delta_{x, n-1} (\Phi_r -\varphi(n-1))\\
    &=\sum_{y \in \Lambda_n} p(x-y) (\varphi(y) -\varphi (x)) + \delta_{x,1} (\Phi_\ell - \varphi(1)) + \delta_{x, n-1} (\Phi_r -\varphi (n-1)).
\end{split}  
\end{equation}
Taking the expectation w.r.t. $\mu_{ss}^n$ on both sides of last display, we deduce that $\Phi^n_{ss}(x):= E_{\mu^n_{ss}}[\varphi(x)]$ is solution of 
\begin{equation}
\label{eq:phissneharmonic}
\sum_{y \in \Lambda_n} p(x-y) (\Phi_{ss}^n (y) -\Phi_{ss}^n (x)) + \delta_{x,1} (\Phi_\ell - \Phi_{ss}^n (1) ) + \delta_{x, n-1} (\Phi_r -\Phi_{ss}^n (n-1) )= 0 .
\end{equation}
Like in \cite{bernardin2022non}, we can prove  that there exists a unique solution to this equation.  Consider the continuous-time random walk $(X_t)_{t \ge 0}$ on $\{0,\ldots,n\}=\Lambda_n \cup \{ 0, n\}$ with generator $\mathfrak G_n$ acting on functions $f:\{0,\ldots,n\} \to \mathbb R$ as
\begin{equation}
    (\mathfrak G_n f) (x) ={\bf 1}_{x \in \Lambda_n} \sum_{y \in \Lambda_n} p(y-x) [f(y)- f(x)] + {\bf 1}_{x=1} [f(0)- f(x)] + {\bf 1}_{x=n-1} [f(n)-f(x)],
\end{equation}
i.e., the continuous time random walk with absorbing states $0$ and $n$, transition rates $p(y-x)$ between $x$ and $y$ if $x,y \in \Lambda_n$ and two additional transition rate of intensity $1$ from $1$ to $0$ and $n-1$ to $n$. Then defining $\tau_n = \inf\{t \ge 0 \; ; \; X_t \in \{0,n\}\}$ as the absorption time of the random walk, we have that
\begin{equation}
\label{eq:prob-repphiss}
    \Phi_{ss}^n (x) = \Phi_\ell \, \mathbb P_x (X_{\tau_n} =0) + \Phi_r \,  \mathbb P_x (X_{\tau_n} =n). 
\end{equation}

To justify this formula, just extend $\Phi_{ss}^n$ to $\{0,\ldots,n\}$ by $\Phi_{ss}^n (0)=\Phi_\ell$ and $\Phi_{ss}^n (n)=\Phi_r$. Then $\Phi_{ss}^n$ is an harmonic function for $\mathfrak G_n$, i.e., $\mathfrak G_n \, \Phi_{ss}^n =0$, so that $(\Phi_{ss}^n (X_t))_{t \ge 0}$ is a martingale and the application of the optional stopping theorem at time $\tau_n$ gives the result. In particular we deduce that \eqref{eq:boundprofileness} holds. This concludes the proof of Lemma \ref{lem:NESS}.

\subsection{Proof of Theorem \ref{theo:stationary}}
\label{subsubsec:theo-stationarity}

The existence has been proved in the previous section. To prove uniqueness, let $\pi^1 (du) = \Phi^1 (u) du$ and $\pi^2 (du) = \Phi^2 (u) du$ be two weak solutions, and define their difference $\bar{\pi}(du) = \pi^2(du) - \pi^1(du) = \bar{\Phi}(u) du$. By items i) and ii) of Definition~\ref{def:stationary}, we have $\bar{\Phi} \in \mathcal H_0^{\gamma/2}$. Now let $(G_k)_{k \ge 1}$ be a sequence of functions in $C_c^{\infty} ([0,1])$ that converges to $\bar{\Phi}$ in $\mathcal H_0^{\gamma/2}$. Using item iii) of Definition~\ref{def:stationary}, for each $k \ge 1$ we obtain
\[
\langle \bar{\Phi}, \mathbb{L}^\gamma G^k \rangle = \langle \bar{\Phi}, G^k \rangle_{\gamma/2} = 0.
\]
Passing to the limit as $k \to \infty$ yields $\bar{\Phi} = 0$, which establishes uniqueness.

Let us now prove \eqref{eq:boundphissn} and \eqref{eq:convPhissn-Phiss}. To prove these claims we explore a connection with the Exclusion Process with long jumps in the discrete lattice $\Lambda_n$. This approach is similar to the one used in \cite{bernardin2022non} for the zero range process with long jumps. We describe here the main aspects of the proof and refer the interested reader to \cite{bernardin2022non} for the precise details.

Let $(\eta_t)_{t\geq 0}$ denote the exclusion process with long jumps on $\Lambda_n$ with boundary densities given by $\alpha\in (0,1)$ and $\beta\in (0,1)$. More precisely $(\eta_t)_{t \ge 0}$ is the Markov process with state space $\{0,1\}^{\Lambda_n}$ and generator $\mathcal G_n$ acting on functions $f:\{0,1\}^{\Lambda_n}\to\mathbb R$ as 
\begin{equation}
\begin{split}
    (\mathcal G_n f ) (\eta) &= \cfrac{1}{2} \sum_{x,y \in \Lambda_n} p(y-x) [ f(\eta^{xy}) -f(\eta)]\\
    &+ (1-\alpha) \eta_1 [ f(\eta^1) -f(\eta)] + \alpha(1-\eta_1)[ f(\eta^1) - f(\eta)]\\
    &+ (1-\beta) \eta_{n-1} [ f(\eta^{n-1}) -f(\eta)] + \alpha(1-\eta_{n-1}1)[ f(\eta^{n-1}) - f(\eta)]\\
\end{split}
\end{equation}
for any $\eta \in \Lambda_n$. Here $\eta^{xy}$ is the configuration obtained from $\eta$ by exchanging the occupation variables $\eta_x$ and $\eta_y$, and $\eta^x$ is the configuration obtained from $\eta$ by replacing the occupation variable $\eta_x$ by $1-\eta_x$. This Markov process is ergodic and we denote by $\nu^n_{ss}$ its stationary state and by $\rho_{ss}^{n}$ the discrete profile $\rho_{ss}^n (x)=E_{\nu^n_{ss}}[\eta(x)] \in [0,1]$. By computing $\mathcal G_n \eta (x)$, we see that $\rho_{ss}^{n}:\Lambda_n\to[0,1]$ is the unique solution of  \eqref{eq:phissneharmonic} by replacing $\Phi_\ell$ (resp. $\Phi_r$) by $\alpha$ (resp. $\beta$). Hence we have
\begin{equation}
\label{eq:linkphissnrhossn}
    \Phi^n_{ss}(x)=\frac{\Phi_r -\Phi_\ell}{\beta -\alpha} \, \rho^n_{ss}(x) \, +\, \Phi_\ell -\Phi_r.
\end{equation}

%

Like in{\footnote{In \cite{bernardin2022non}, the reservoirs of the exclusion process considered are extended reservoirs and not like here one-site reservoirs. However it is straightforward to extend the results there to our case.}} \cite{bernardin2022non}, we can prove there exists a measurable function $\rho_{ss}:[0,1] \to [0,1]$ belonging to ${\mathcal H}^{\gamma/2}$ such that for any continuous function $G:[0,1]\to \mathbb R$ we have
\begin{equation}
\lim _{n\to\infty }  \left\vert \frac{1}{n}\sum_{x\in\Lambda_n}G(\tfrac xn)\rho_{ss}^n(x)  - \int_{0}^1G(u) \rho_{ss}(u)du \right\vert   = 0.
\end{equation} 
Moreover $\rho_{ss}$ is solution of \eqref{eq:stationary} with $\Phi_\ell$ (resp. $\Phi_r$) replaced by $\alpha$ (resp. $\beta$) and $\min(\alpha,\beta) \le \rho_{ss} (u) \le \max(\alpha,\beta)$ for any $u\in [0,1]$. This implies \eqref{eq:convPhissn-Phiss} with 
\begin{equation}
\begin{split}
 \forall u \in [0,1], \quad \Phi_{ss}(u)=\frac{\Phi_r -\Phi_\ell}{\beta -\alpha} \, \rho_{ss}(u) \, +\, \Phi_\ell -\Phi_r.
\end{split}
\end{equation}
Moreover $\Phi_{ss}$ is solution of   \eqref{eq:stationary} and satisfies $\min(\Phi_\ell,\Phi_r) \le \Phi_{ss} (u) \le \max(\Phi_\ell, \Phi_r)$ for any $u\in [0,1]$. This concludes the proof of   \eqref{eq:boundphissn}  and \eqref{eq:convPhissn-Phiss}, and thus of Theorem \ref{theo:stationary}.

\subsection{Proof of Theorem \ref{thm:hydro static limit} }

We prove the hydrostatic limit as a consequence of the hydrodynamic limit stated in Theorem \ref{theo:hydro_limit}. Note that $\mu^n_{ss}$ satisfies \eqref{eq:Kr} as a consequence of \eqref{eq:kappaqbis_new}. To conclude, we need to show
\begin{equation}
\label{eq:seq_associated_profile_NESS}
\lim _{n\to\infty } \mu_{ss}^n\left( \varphi\,:\, \left\vert \langle \pi^n, G\rangle - \int_{0}^1G(u) \Phi_{ss}(u)du \right\vert    > \delta \right)= 0 
\end{equation}  
where $\Phi_{ss}(\cdot)$ is the stationary solution of the hydrodynamic equation. The above limit is a consequence of showing
\begin{equation}
\lim _{n\to\infty } \mu_{ss}^n\left( \left\vert \frac{1}{n}\sum_{x\in\Lambda_n}G(\tfrac xn)(\varphi(x)-\Phi_{ss}^n(x) ) \right\vert    > \frac{\delta}{2} \right)= 0 \
\end{equation}  
and 
\begin{equation}
\lim _{n\to\infty } \mu_{ss}^n\left( \left\vert \frac{1}{n}\sum_{x\in\Lambda_n}G(\tfrac xn)\Phi_{ss}^n(x)  - \int_{0}^1G(u) \Phi_{ss}(u)du \right\vert    >  \frac{\delta}{2} \right)= 0 \
\end{equation} 
To estimate the former probability, we use Chebyshev's inequality and we bound it from above by
\begin{align}
    \frac{4}{\delta^2} E_{\mu_{ss}^n}\Big[\Big(\frac{1}{n}\sum_{x\in\Lambda_n}G(\tfrac xn)(\varphi(x)-\Phi_{ss}^n(x) )  \Big)^2\Big]&=\frac{4}{\delta^2}\frac{1}{n^2}\sum_{x,y\in\Lambda_n}G(\tfrac xn)G(\tfrac yn) E_{\mu_{ss}^n}[\bar\varphi(x)\bar\varphi(y)]\\
    &=\frac{4}{\delta^2}\frac{1}{n^2}\sum_{x\in\Lambda_n}G(\tfrac xn)^2 E_{\mu_{ss}^n}[\bar\varphi^2(x)].
\end{align}
where we used the fact that the NESS is product. The last expression above is of order  $O(n^{-1})$ since  $E_{\mu_{ss}^n}[\varphi^2(x)] =1 + \big[\Phi_{ss}^n (x) \big]^2\le 1 + \max\{\Phi_\ell^2,\Phi_r^2\}$. It remains now to show that 
\begin{equation}
\lim _{n\to\infty }  \left\vert \frac{1}{n}\sum_{x\in\Lambda_n}G(\tfrac xn)\Phi_{ss}^n(x)  - \int_{0}^1G(u) \Phi_{ss}(u)du \right\vert   = 0,
\end{equation} 
but this has been proved just above. For that reason, we shall conclude, as in Corollary 2.16 of \cite{bernardin2022non}, the hydrostatic limit for $\Phi^n_{ss}$ from the one for $\rho^n_{ss}$, and we are done.

\section{Dynamical Large Deviations Principle}
\label{sec:dynLDP}

\subsection{The Rate Function $I_{[0,T]} (\cdot \vert g)$} 
\label{subsec:I0T}
In this section we prove some properties of the dynamical large deviations rate function $I_{[0,T]}$. The next lemma will be used in the identification of the Quasi-Potential with the rate function under the stationary state in Section \ref{sec:MFT}.

Fix {$g \in L^{2} ([0,1])$} and let $\mathcal C^{ac}_g$ denote the subset of $\mathcal C^{ac}$ (see Definition \ref{def:space_C_ab}) composed of paths $(\pi_t)_{t \in [0,T]}$ such that the density $\Phi_0$ of $\pi_0$ satisfies $\Phi_0=g$ a.e.
\begin{lemma}
\label{lem:LDinitialcondition}
 If $I_{[0,T]}(\pi|g)<+\infty$, then $\pi\in \mathcal C^{ac}_g.$
\end{lemma}
\begin{proof}
    Since $I_{[0,T]}(\pi|g)<+\infty$, we know from \eqref{eq:rate function}, that for any $t \in [0,T]$, $\pi_t(du)=\Phi_t(u)du$ with $\Phi \in L^2 ([0,T], \mathcal H^{\gamma/2})$. For $\delta>0$ and $G\in C^2_c([0,1])$, consider the function $G^\delta_t (u)=\big(1-\tfrac{t}{\delta}\big)^+G(u)$, where $F^+$ denotes the positive part of the function $F$. Recall 
    \eqref{eq:JHPG}.
Let us show that 
    \begin{equation}
    \label{eq:Joao23}
        \lim_{\delta\downarrow0} J_{G^\delta}(\pi| g)=\langle \Phi_0, G\rangle-\langle g,G\rangle.
    \end{equation}
    Since $J_{H^\delta}(\pi \vert g) \le I_{[0,T]}(\pi|g)<+\infty$, then for any $G\in C^2_c([0,1])$, we have $\langle \Phi_0, G\rangle-\langle g,G\rangle=0$ (to see it, replace $G$ by $\lambda G$ for an arbitrary $\lambda \in \mathbb R$) and therefore $\Phi_0=g$ a.e. 
    
    To prove \eqref{eq:Joao23} we compute separately the limit as $\delta$ goes to $0$ of each term appearing in $J_{G^\delta} (\pi \vert g)$. First, note  that $\langle \Phi_T,G_T^\delta\rangle=0$, for $T>\delta$. Moreover, computing $\partial_t G^\delta_t(u)=-\delta^{-1}\boldsymbol{1}_{[0,\delta]}(t)G(u)$, we get
    \begin{equation}
        \int_0^T\langle\Phi_t,\partial_t G^\delta_t\rangle \, dt=-\frac{1}{\delta}\int_0^\delta\langle \Phi_t,G\rangle \,  dt.
    \end{equation}
    Since $\pi \in C([0,T],\mathcal M)$, the function $t \in [0,T] \to\langle\Phi_t,G\rangle\in\mathbb R$ is continuous. We conclude thus that as $\delta\downarrow0$, the term in the previous identity converges to $-\langle\Phi_0,G\rangle$. For the fractional norm, note that 
    \begin{equation}
        \begin{split}
            \int_0^T\| G_t^\delta\|^2_{\gamma/2}\, dt 
            &=\|G\|^2_{\gamma/2}\int_0^\delta \big[1-\tfrac t\delta\big]^2 \, dt \lesssim {\delta},
        \end{split}
    \end{equation}
which goes to zero as $\delta\downarrow0$. Finally, it remains to show that
    \begin{equation}
    \label{eq:Joao24}
\lim_{\delta\downarrow0}\int_0^T\langle\Phi_t,\mathbb L^\gamma G^\delta_t\rangle\, dt =0.
    \end{equation}
    By the Cauchy-Schwarz's inequality, 
    \begin{equation}
        \begin{split}
           \left\vert \int_0^T\langle\Phi_t,\mathbb L^\gamma G^\delta_t\rangle\, dt \right\vert \le \int_0^\delta\!\! \big(1-\tfrac t\delta\big)\left\vert \langle\Phi_t,\mathbb L^\gamma G\rangle \right\vert \, dt
           \leq\|G\|_{\gamma/2} \sqrt{\int_0^\delta||\Phi_t\Vert_{\gamma/2}^2 \, dt }\ \sqrt{\int_0^\delta\big(1-\tfrac t\delta\big)^2 \, dt}\,\lesssim  \sqrt \delta.
        \end{split}
    \end{equation}
Above, we used the fact that $\Phi\in L^2([0,T],\mathcal H^{\gamma/2})$. Therefore,
   \eqref{eq:Joao24} holds, and this concludes the proof of \eqref{eq:Joao23} and hence the proof of the lemma. 
\end{proof}

\begin{lemma}
\label{lem:Poissoneq}
Let $g \in L^1 ([0,1])$. Let $ \pi \in C([0,T], \mathcal M)$ such that $I_{[0, T]} (\pi|g)<\infty$. Then there exists a unique $H \in L^2 ([0,T], \mathcal H_0^{\gamma/2})$ such that $\pi$ is a weak solution of \eqref{eq:Dirichlet Equation2}. Moreover, 
\begin{equation}\label{eq:rate function representation}
I_{[0,T]} (\pi \vert g) = \cfrac{1}{4}\int_0^T \Vert H_t \Vert^2_{\gamma/2} \, dt.
\end{equation}

\end{lemma}

\begin{proof}
For any $G\in C^{1,2}_c(\Omega_T)$, let $\ell_G: C([0,T], \mathcal M) \to \mathbb R$ be the linear functional defined by 
\begin{equation}
    \ell _G (\pi) = \langle \pi_T, G_T\rangle -\langle \pi_0, G_0 \rangle - \int_{0}^T \langle \pi_t, (\partial _t +\mathbb L^\gamma) G_t \rangle \, dt.
\end{equation}
Since $I_{[0,T]} (\pi \vert g) < \infty$ there exists $\Phi \in L^2 ([0,T], \mathcal H^{\gamma/2})$ such that for any $t \in [0,T]$, $\pi_t (du) = \Phi_t (u) du$ and by Lemma \ref{lem:LDinitialcondition}, $\Phi_0=g$. Moreover, for any $G\in C^{1,2}_c(\Omega_T)$,
\begin{equation}
J_G (\pi \vert g) = \ell_G (\pi)- \int_0^T \Vert G_t \Vert_{\gamma/2}^2 \, dt \le I_{[0,T]} (\pi|g) .
\end{equation}
By replacing  $G$ by $aG$, in the previous inequality, with $a$ being an arbitrary real number, and minimizing over $a$, we obtain  
\begin{equation}
    \vert \ell_G (\pi)\vert \le 2\sqrt{ I_{[0,T]} (\pi)} \sqrt{\int_0^2 \Vert G_t \Vert_{\gamma/2}^2 \, dt} .
\end{equation}
Since this true for any $G\in C_c^{1,2} (\Omega_T)$, this proves that the linear functional $G \to \ell_{G} (\pi)$ can be extended to $L^2 ([0,T], {\mathcal H}_0^{\gamma/2})$ and it is a bounded linear functional in that space. By Riesz's representation theorem, there exists  $H \in L^2 ([0,T], {\mathcal H}_0^{\gamma/2})$ such that 
\begin{equation}
\forall G \in L^2 ([0,T], {\mathcal H}_0^{\gamma/2}), \quad  \ell_G (\pi) = \int_0^T \langle H_t, G_t \rangle_{\gamma/2} \ dt.
\end{equation}
Since we already know that $\pi \in \mathcal C^{ac}_g$, the last equation implies that $\pi$ is a weak solution of \eqref{eq:Dirichlet Equation2}. Moreover, with this representation, we can rewrite $I_{[0,T]} (\pi \vert g)$ as
\begin{equation}
\begin{split}
    I_{[0,T]} (\pi\vert g) &= \sup_{G \in C_c^{1,2} (\Omega_T) } \left\{ \int_0^T \langle H_t, G_t \rangle_{\gamma/2} \ dt -\int_0^T \Vert G_t \Vert_{\gamma/2}^2 \, dt \right\}\\
    &= \sup_{G \in L^2([0,T], \mathcal H_0^{\gamma/2}) } \left\{ \int_0^T \langle H_t, G_t \rangle_{\gamma/2} \ dt - \int_0^T \Vert G_t \Vert_{\gamma/2}^2 \, dt \right\} . 
\end{split}    
\end{equation}
The last supremum is easy to compute being  the optimizer  $H/2$, so that \eqref{eq:rate function representation} holds.
\end{proof}
\subsection{Superexponential boundary replacement lemmas}
\label{app:superexp RL sec}

In this section we prove superexponential replacement lemmas at the boundaries. We start with some estimates of the quadratic forms associated with a reference measure that will be defined ahead. Consider a profile $\Phi:[0,1]\to\mathbb R$  such  that
$\Phi$ is locally constant equal to $\Phi_\ell$ (resp. $\Phi_r$) in a neighborhood of $0$ (resp.  $1$) and Lipschitz continuous everywhere. We define the probability measure $\mu_{\Phi(\cdot)}^n$ on  $\Omega_n$ as the product measure  given by 
\begin{equation}
\label{eq:initial_measure}
\mu^n_{\Phi(\cdot)} (d\varphi) = \frac{1}{Z^n_{\Phi(\cdot)}} \exp\Big( -\mathcal E (\varphi) + \sum_{x\in\Lambda_n}  \varphi(x)\Phi(\tfrac xn) \Big) d\varphi
\end{equation}
where ${\mathcal Z}^n_{\Phi (\cdot)}$ is a normalisation constant given by $\mathcal Z^n_{\Phi(\cdot)}=(2\pi)^{n/2}\exp\Big(\frac{1}{2}\sum_{x\in\Lambda_n}\Phi(\tfrac xn)^2\Big)$.

\begin{remark}
    It seems probably more natural to consider as reference measure directly the stationary state $\mu_{ss}^n$ instead of $\mu_{\Phi(\cdot)}^n$ since it is semi-explicit. However, in the proofs below, we use several times that the profile ${\Phi(\cdot)}$ is Lipschitz. This property is not true for the stationary profile $\Phi_{ss}(\cdot)$. In fact, there are numerical evidences that the stationary profile is actually singular at the boundaries, even though it is not rigorously proved (see Figure \ref{fig:myfigure}).
\end{remark}

\begin{figure}[h]
\centering
\includegraphics[width=0.8\textwidth]{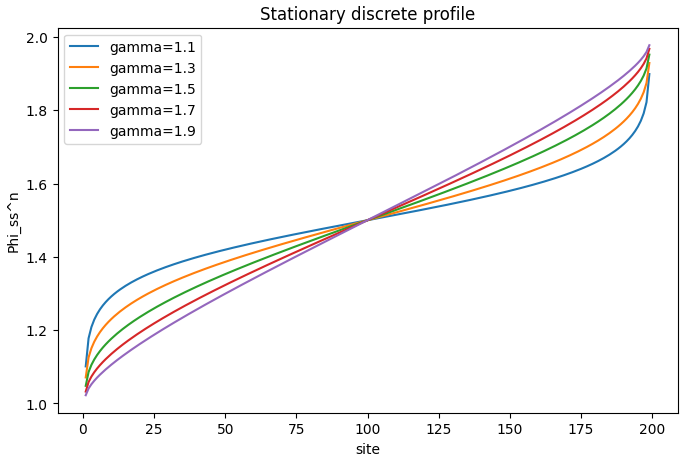}
\caption{Graph of the stationary profile $\Phi^n_{ss}$ for $\Phi_\ell=1$, $\Phi_r=2$ and $n=200$.}
\label{fig:myfigure}
\end{figure}
\subsubsection{Quadratic forms estimates}\label{subsec:dirichlet forms} The quadratic form ${\mathfrak D}^{n}({f},\mu)$ of the continuously differentiable function $f:\mathbb R^{\Lambda_n}\to (0,\infty)$ with respect to a probability measure $\mu$ on $\Omega_n$ is defined by 
\begin{equation}
{\mathfrak D}^{n}({f},\mu):=n^\gamma\Big\{{\mathfrak D}_{b}^n({f},\mu)+{\mathfrak D}_{\ell}({f},\mu)+{\mathfrak D}_{r}({f},\mu) \Big\}.
\end{equation}
Above, the bulk part is given by
\begin{equation}
{\mathfrak D}_{b}^n({f},\mu) =  \sum_{x,y\in\Lambda_n}p (y-x) I_{x,y}(f,\mu)
\end{equation}
with 
\begin{equation}\label{eq:I_in_DF}
 I_{x,y}( f,\mu):=\int_{\Omega_n} \left(\partial_{\varphi(y)} {f(\varphi)}-\partial_{\varphi(x)} {f(\varphi)}\right)^{2}  d{\mu}(\varphi),
\end{equation}
and the boundary parts are given by
\begin{equation}
\begin{split}
&{\mathfrak D}_{\ell}({f},\mu )= \int_{\Omega_n} \Big[\partial_{\varphi(1)} {f(\varphi)}\Big]^2d{\mu}(\varphi), \quad \textrm{and}\quad{\mathfrak D}_{r}({f},\mu )= \int_{\Omega_n} \Big[\partial_{\varphi(n-1)}{f(\varphi)}\Big]^2d{\mu}(\varphi).
\end{split}
\end{equation}
The definitions above require that $f$ satisfies the integrability condition
\begin{equation}
\label{eq:int-f-df}
\begin{split}
\sum_{x \in \Lambda_n} \int_{\Omega_n} \Big( \partial_{\varphi(x)} f \Big)^2 d\mu (\varphi) < \infty .
\end{split}
\end{equation}
In the sequel we will implicitly assume that this condition is satisfied.
\begin{lemma}
\label{lem:dir_carre}
For any continuously differentiable function $f:\mathbb R^{\Lambda_n} \to [0,\infty)$ such that $\int f^2 d\mu_{\Phi(\cdot)}^n(\varphi)=1$, we have that
\begin{equation}\label{dir_est:ini_prof}
\begin{split}
&\langle \mathcal  L_H^{n}{f},{f} \rangle_{\mu_{\Phi(\cdot)}^n}  \;\lesssim\; -\frac{1}{8}{\mathfrak D}^{n}({f},\mu_{\Phi(\cdot)}^n) +\frac {n^\gamma}{ 4}\sum_{x,y\in\Lambda_n} p(y-x)\Big(\Phi(\tfrac{x}{n})-\Phi(\tfrac{y}{n})\Big)^2\\
&+\frac {n^\gamma}{ 2}\sum_{x,y\in\Lambda_n} p(y-x)\Big(H(\tfrac{x}{n})-H(\tfrac{y}{n})\Big)^2+\frac {n^\gamma}{ 4}\Big(\Phi_\ell-\Phi(\tfrac 1n)\Big)^2+\frac {n^\gamma}{4}\Big(\Phi_r-\Phi(\tfrac {n-1}{n})\Big)^2.
\end{split}
\end{equation}
\end{lemma}
\begin{proof}
We split the proof by analyzing the distinct generators that make up $\mathcal L^H_n$. We start with the bulk generator.
Observe that 
\begin{equation}
\label{eq:1_ident}
\begin{split}
\langle \mathcal  L^{bulk}{f},{f} \rangle_{\mu_{\Phi(\cdot)}^n}  \;&=\frac{1}{2}\sum_{x,y\in\Lambda_n}\int_{\Omega_n} p(y-x)\Big[\Big(\partial_{\varphi(y)} -\partial_{\varphi(x)}\Big)^2 {f(\varphi)}\Big]{f(\varphi)}d{\mu_{\Phi(\cdot)}^n}(\varphi)\\
&+\frac{1}{2}\sum_{x,y\in\Lambda_n}\int_{\Omega_n} p(y-x)\Big(\varphi(x)-\varphi(y)\Big)\Big[\Big(\partial_{\varphi(y)} -\partial_{\varphi(x)}\Big) {f(\varphi)}\Big]{f(\varphi)}d{\mu_{\Phi(\cdot)}^n}(\varphi).
\end{split}
\end{equation}
Recall \eqref{eq:initial_measure}. Now we use the following identity, which holds for any $z\in\Lambda_n$
	\begin{equation}\label{eq:useful_id}
	\partial_{\varphi(z)}\exp\left( -\mathcal E (\varphi) +\sum_{x\in\Lambda_n} \varphi(x)\Phi(\tfrac xn) \right) =\big[-\varphi(z)+\Phi(\tfrac zn)\big]\exp\left( -\mathcal E (\varphi) +\sum_{x\in\Lambda_n} \varphi(x)\Phi(\tfrac xn) \right).
	\end{equation}
A simple computation based on the last identity and integration by parts shows that the first term on the right-hand side of \eqref{eq:1_ident} is equal to
	\begin{equation}
\label{eq:1_int_parts}
\begin{split}
&\frac{1}{2}\sum_{x,y\in\Lambda_n}\int_{\Omega_n} p(y-x)\Big[\Big(\partial_{\varphi(y)} -\partial_{\varphi(x)}\Big)^2{f(\varphi)}\Big] {f(\varphi)}d{\mu_{\Phi(\cdot)}^n}(\varphi)\\=&-\frac{1}{2}\sum_{x,y\in\Lambda_n}\int_{\Omega_n} p(y-x)\Big[(\partial_{\varphi(y)} { f(\varphi)})^2+(\partial_{\varphi(y)} {f(\varphi)}){f(\varphi)}\big(-\varphi(y)+\Phi(\tfrac yn)\big)\Big]d{\mu_{\Phi(\cdot)}^n}(\varphi)\\
+&\frac{1}{2}\sum_{x,y\in\Lambda_n}\int_{\Omega_n} p(y-x)\Big[\partial_{\varphi(x)}\partial_{\varphi(y)} { f(\varphi)}+(\partial_{\varphi(x)} {f(\varphi)}){f(\varphi)}\big(-\varphi(y)+\Phi(\tfrac yn)\big)\Big]d{\mu_{\Phi(\cdot)}^n}(\varphi)\\
+&\frac{1}{2}\sum_{x,y\in\Lambda_n}\int_{\Omega_n} p(y-x)\Big[\partial_{\varphi(y)}\partial_{\varphi(x)} { f(\varphi)}+(\partial_{\varphi(y)} {f(\varphi)}){f(\varphi)}\big(-\varphi(x)+\Phi(\tfrac xn)\big)\Big]d{\mu_{\Phi(\cdot)}^n}(\varphi)
\\-&\frac{1}{2}\sum_{x,y\in\Lambda_n}\int_{\Omega_n} p(y-x)\Big[(\partial_{\varphi(x)} { f(\varphi)})^2+(\partial_{\varphi(x)} {f(\varphi)})\sqrt{f(\varphi)}\big(-\varphi(x)+\Phi(\tfrac xn)\big)\Big]d{\mu_{\Phi(\cdot)}^n}(\varphi).
\end{split}
\end{equation}
And grouping terms appropriately, the last expression becomes equal to 
\begin{equation}
\begin{split}
-\frac{1}{2}{\mathfrak D}_b^{n}({f}&,\mu_{\Phi(\cdot)}^n)-\frac{1}{2}\sum_{x,y\in\Lambda_n}\int_{\Omega_n} p(y-x)\Big(\varphi(x)-\varphi(y)\Big)\Big[\Big(\partial_{\varphi(y)} -\partial_{\varphi(x)}\Big){f(\varphi)}\Big]{f(\varphi)}d{\mu_{\Phi(\cdot)}^n}(\varphi)\\&+
\frac{1}{2}\sum_{x,y\in\Lambda_n}\int_{\Omega_n} p(y-x)\Big(\Phi(\tfrac{x}{n})-\Phi(\tfrac{y}{n})\Big) \Big[\Big(\partial_{\varphi(y)} -\partial_{\varphi(x)}\Big) {f(\varphi)}\Big]{f(\varphi)}d{\mu_{\Phi(\cdot)}^n}(\varphi).
\end{split}
\end{equation}
Now,  the second term in the last display cancels with the second term on the right-hand side of \eqref{eq:1_ident}, while, by Young's inequality, for any $A>0$, the last term can be bounded from above by
\begin{equation}
\begin{split}
&\frac{A}{4} \sum_{x,y\in\Lambda_n}\int_{\Omega_n} p(y-x)\left(\partial_{\varphi(y)} {f(\varphi)}-\partial_{\varphi(x)}{f(\varphi)}\right)^2d{\mu_{\Phi(\cdot)}^n}(\varphi)\\+&
\frac{1}{4A}\sum_{x,y\in\Lambda_n}\int_{\Omega_n} p(y-x){f^2(\varphi)}\Big(\Phi(\tfrac{x}{n})-\Phi(\tfrac{y}{n})\Big)^2d{\mu_{\Phi(\cdot)}^n}(\varphi).
\end{split}
\end{equation}
Making the choice $A=1$ and using the fact that $\int f^2 d\mu_{\Phi(\cdot)}^n=1$, we get
\begin{equation}\label{est:bulk}
\begin{split}
\langle \mathcal  L^{bulk}{f},{f} \rangle_{\mu_{\Phi(\cdot)}^n}  \;&\leq\;-\frac{1}{4}{\mathfrak D}_b^{n}({f},\mu_{\Phi(\cdot)}^n) +\frac{1}{4}\sum_{x,y} p(y-x)\Big(\Phi(\tfrac{x}{n})-\Phi(\tfrac{y}{n})\Big)^2.
\end{split}
\end{equation}
Now we present the proof for the left boundary but we note that for the right boundary it is completely analogous.
Note that
\begin{equation}\label{eq:2_ident}
\begin{split}
\langle \mathcal  L_\ell{f},{f} \rangle_{\mu_{\Phi(\cdot)}^n}  \;&=\int_{\Omega_n} \Big[\partial^2_{\varphi(1)} {f(\varphi)}\Big] {f(\varphi)}d{\mu_{\Phi(\cdot)}^n}(\varphi)+\int_{\Omega_n}  \Big(\Phi_\ell-\varphi(1)\Big)\Big[\partial_{\varphi(1)} {f(\varphi)}\Big] {f(\varphi)}d{\mu_{\Phi(\cdot)}^n}(\varphi).
\end{split}
\end{equation}
As above, using the identity 
 \begin{equation}\label{eq:usef_id_2}
\partial_{\varphi(1)}\exp\left( -\mathcal E (\varphi) +\Phi \mathcal V (\varphi) \right) =\big[-\varphi(1)+\Phi_\ell\big]\exp\left( -\mathcal E (\varphi) +\Phi \mathcal V (\varphi) \right)
\end{equation}
and  integration by parts,  the right-hand side of \eqref{eq:2_ident} is equal to 
\begin{equation}
\begin{split}
-{\mathfrak D}_{\ell}({f},\mu_{\Phi(\cdot)}^{n} )&+\int_{\Omega_n}  \Big(\varphi(1)-\Phi(\tfrac 1n)\Big)\Big[\partial_{\varphi(1)} {f(\varphi)}\Big] {f(\varphi)}d{\mu_{\Phi(\cdot)}^n}(\varphi)\\&+\int_{\Omega_n}  \Big(\Phi_\ell-\varphi(1)\Big)\Big[\partial_{\varphi(1)} {f(\varphi)}\Big] {f(\varphi)}d{\mu_{\Phi(\cdot)}^n}(\varphi).
\end{split}
\end{equation}
From Young's inequality and using the fact that $\int f^2 d\mu_{\Phi(\cdot)}^n=1$, we get, for any $B>0$, that the  last display, can  be bounded from above by
\begin{equation}
\begin{split}
 -{\mathfrak D}_{\ell}({f},\mu_{\Phi(\cdot)}^{n} )+\frac {1}{2B}{\mathfrak D}_{\ell}({f},\mu_{\Phi(\cdot)}^{n} )+\frac B2 \Big(\Phi_\ell-\Phi(\tfrac 1n)\Big)^2.
\end{split}
\end{equation}
For the choice $B=1$  we get
\begin{equation}
\begin{split}
\langle \mathcal  L_\ell{f},{f} \rangle_{\mu_{\Phi(\cdot)}^n}  \;&\leq -\frac 12 {\mathfrak D}_{\ell}({f},\mu_{\Phi(\cdot)}^{n} )+\frac 12\Big(\Phi_\ell-\Phi(\tfrac 1n)\Big)^2.
\end{split}
\end{equation} 
Finally, we note that the contribution from the tilted generator is exactly the same as in \eqref{est:bulk} with $\Phi$ replaced by $H$. Therefore, choosing appropriately the constant $A$ we get
\begin{equation}
    \langle \mathcal{L}^{\text{tilt}}{f},{f}\rangle_{\mu^n_{\Phi(\cdot)}}\lesssim \frac{1}{8}{\mathfrak D}^n({f},\mu^n_{\Phi(\cdot)})+\frac{1}{2}\sum_{x,y\in\Lambda_n}p(y-x)\big(H_t(\tfrac yn))-H_t(\tfrac xn)\big)^2.
\end{equation} Hence, with all estimates in hand, we conclude the proof. 

\end{proof}

\begin{remark}
\label{rmk:bound_dirichlet form}
    Since we assume the profile $\Phi$ to be Lipchitz, it follows that
    \begin{equation}
        \langle \mathcal  L_H^{n}{f},{f} \rangle_{\mu_{\Phi(\cdot)}^n}  \;\lesssim\; -\frac{1}{4}{\mathfrak D}^{n}({f},\mu_{\Phi(\cdot)}^n)+n+n^{\gamma-2}.
    \end{equation}
    Indeed, by the Lipchitz assumption, we have that 
    \begin{equation}
n^\gamma\sum_{x,y\in\Lambda_n}p(y-x)\big(\Phi(\tfrac yn)-\Phi(\tfrac xn)\big)^2\lesssim n^\gamma\sum_{x,y\in\Lambda_n}\frac{1}{|y-x|^{1+\gamma}}\Big(\frac{y-x}{n}\Big)^2=n^{\gamma-2}\sum_{x,y\in\Lambda_n}|y-x|^{1-\gamma}.
    \end{equation}
    Now, we can approximate the sum with a Riemann integral to get 
    \begin{equation}\label{eq:useful_bound}
        \sum_{x,y\in\Lambda_n}|y-x|^{1-\gamma}\lesssim n\sum_{z\in\Lambda_n}z^{1-\gamma}\lesssim n\int_1^nz^{1-\gamma}dz\lesssim n^{2-\gamma}.
    \end{equation}
    Therefore, $n^\gamma\sum_{x,y\in\Lambda_n}p(y-x)\big(\Phi(\tfrac yn)-\Phi(\tfrac xn)\big)^2\lesssim n$. And the same bound holds for the term with the perturbation $H$. With the same type of arguments we get $n^\gamma\big(\Phi_\ell-\Phi(\tfrac 1n)\big)^2\lesssim n^{\gamma-2}$.
\end{remark}
\subsubsection{Boundary Replacement}
The goal of this subsection is to prove Proposition \ref{prop:charact dirichlet bd cond super exp-0}. It guarantees that we can characterize the Dirichlet boundary conditions at the level of large deviations, and as a corollary we also characterize it at the level of the hydrodynamic limit.

Recall the definition of averages given by \eqref{eq:emp_average} and assume for simplicity that  $\varepsilon n$ denotes $\lfloor \varepsilon n\rfloor$, the integer part of $\varepsilon n$.

\begin{proposition}
\label{prop:charact dirichlet bd cond super exp-0}
Let $H \in C^{1,2}_c (\Omega_T)$. There exists a positive constant $\kappa_0$ depending only on $T$, $\Phi$, $\gamma$ and $H$, and a function $\varepsilon_0: \delta \in (0,\infty) \to \varepsilon_0 (\delta) \in (0,1)$ going to zero as $\delta$ goes to zero such that: for any $\delta>0$ and any $\varepsilon \in (0,\varepsilon_0(\delta)]$, for any sequence $(a^{j})_{j\geq 1}$ such that for each $j$, ${a^{j}}:[0,T] \to \mathbb R$ with $\|a^{j}\|_\infty\leq 1$, it holds
\begin{equation}
\label{eq:replacementimproved}
\limsup_{n \rightarrow \infty}\frac{1}{n}\log\mathbb{P}^H_{n}\Bigg[\sup_{1 \le j \le k} \Big|  \int_0^T a^j_t  \, \Big(\Phi_\ell-\overrightarrow{\varphi}_t^{\varepsilon n}(0)\Big)\,dt\Big|>\delta\Bigg] \le -  \frac{\kappa_0}{T} \delta^2 \varepsilon^{1-\gamma} +\frac{T}{\kappa_0}.
\end{equation}
The same result holds  replacing $\overrightarrow\varphi^{\varepsilon n}_t(0)$ by $\overleftarrow\varphi^{\varepsilon n}_t(n)$ and  $\Phi_{\ell}$ by $\Phi_{r}$.
\end{proposition}

\begin{proof}
We assume here Proposition \ref{lem:charact dirichlet bd cond super exp}.
Consider first the case $H=0$. Recall the reference measure $\mu^n_{\Phi(\cdot)}$ defined  in \eqref{eq:initial_measure}. Observe that for any $j \in \{1, \ldots,k\}$, since $\Vert a^j \Vert_{\infty} \le 1$, we have $\int_0^T |a_s^j|^2 ds \le T$ and $\int_0^T |a_s^j| ds \le T$. By the union bound, \eqref{sum_log_super} and Lemma \ref{lem:charact dirichlet bd cond super exp}, for any $\varepsilon \le 2\delta/(T\Vert \Phi'\Vert_{\infty}):=\varepsilon_0 (\delta)$, the following estimate holds:
\begin{equation}
\begin{split}
\limsup_{n \rightarrow \infty}\frac{1}{n}\log\mathbb{P}_{\mu^n_{\Phi(\cdot)}}\Bigg[\sup_{1 \le j \le k} \Big| & \int_0^T a^j_t  \, \Big(\Phi_\ell-\overrightarrow{\varphi}_t^{\varepsilon n}(0)\Big)\,dt\Big|>\delta\Bigg] \\
&\le \max_{j=1, \ldots, k} \left\{ \limsup_{n \rightarrow \infty}\frac{1}{n}\log\mathbb{P}_{\mu^n_{\Phi(\cdot)}}\Bigg[ \Big|  \int_0^T a^j_t  \, \Big(\Phi_\ell-\overrightarrow{\varphi}_t^{\varepsilon n}(0)\Big)\,dt\Big|>\delta\Bigg]\right\}\\
& \le - \cfrac{\kappa_0}{T} \, \delta^2 \varepsilon^{1-\gamma} + \frac{T}{\kappa_0},
\end{split}
\end{equation}
where $\kappa_0$ is a positive constant depending only on $\Phi$ and $\gamma$. By Lemma \ref{lem:mussnmun}, the same results holds with $\mathbb P_{\mu^n_{\Phi(\cdot)}}$ replaced by $\mathbb P_{\mu_n}$ and by Lemma \ref{lem:super_exchange_measures}, it then holds with $\mathbb P_{\mu_n}$ replaced by $\mathbb P_{\mu_n}^H$.

\end{proof}

%

In the proof of the previous proposition we need the result of the next lemma, but the starting measure is the reference measure $\mu^n_{\Phi(\cdot)}$ and with the choice $H=0$. For this reason all the results below are already stated in this simpler setting.

\begin{proposition}
\label{lem:charact dirichlet bd cond super exp}
There exists a constant $\kappa_0$  depending only on $\Phi$ and $\gamma$ such that for any square integrable function $a:[0,T] \to \mathbb R$, any $\delta>0$ and any $\varepsilon \le {\delta} \Big(2 \Vert \Phi'\Vert_{\infty} \int_0^T |a_s| ds\Big)^{-1}$, it holds
\begin{equation}
\limsup_{n \rightarrow \infty}\frac{1}{n}\log\mathbb{P}_{\mu^n_{\Phi(\cdot)}}\Bigg[\Big|\int_0^T a_s \, \Big(\Phi_\ell-\overrightarrow{\varphi}_s^{\varepsilon n}(0)\Big)\,ds\Big|>\delta\Bigg] \le - \kappa_0 \cfrac{\delta^2}{\int_0^T |a_s|^2 ds} \varepsilon^{1-\gamma} + \frac{T}{\kappa_0}.
\end{equation}
The same result holds  replacing $\overrightarrow\varphi^{\varepsilon n}_s(0)$ by $\overleftarrow\varphi^{\varepsilon n}_s(n)$ and  $\Phi_{\ell}$ by $\Phi_{r}$. In particular 
\begin{equation}
\limsup_{\varepsilon \to 0} \limsup_{n \rightarrow \infty}\frac{1}{n}\log\mathbb{P}_{\mu^n_{\Phi(\cdot)}}\Bigg[\Big|\int_0^T a_s \, \Big(\Phi_\ell-\overrightarrow{\varphi}_s^{\varepsilon n}(0)\Big)\,ds\Big|>\delta\Bigg] = - \infty.
\end{equation}
\end{proposition}

\begin{proof}
The previous proposition is fundamental for the proof of  Proposition \ref{prop:charact dirichlet bd cond super exp-0} and relies on a series of results that we present below. The strategy  of the proof is in spirit close to the one presented in  \cite{BCGS2023spa} but it improves it qualitatively, and we take it to the large deviations level. First we prove in Lemma \ref{lem:one_block} that we can approximate the occupation variable $\varphi_s(1)$ by a local average of size $\ell_0=\varepsilon n^{\gamma-1}$, this is the \textit{one-block lemma}. Then, in Lemma \ref{lem:two-blocks}, we apply a multi-scale argument to increase the size of the box $\ell_0$ to a box of  size $\varepsilon n$. Finally, in Lemma \ref{RL:boundary}, we show that $\varphi_s(1)$ is a good approximation of $\Phi_\ell$.
\end{proof}

\begin{lemma}
\label{lem:one_block}
There exists a constant $\kappa_0$ depending only on $\Phi$ and $\gamma$ such that for any square integrable function $a:[0,T] \to \mathbb R$, for any  $\delta>0$ and for any $\varepsilon>0$, it holds{\footnote{Below, if $\int_0^T |a_s|ds =0$, the inequality is trivially true with the convention $1/0= +\infty$.}}
\begin{equation}
\begin{split}
\limsup_{n \to \infty} \frac 1n \log \mathbb{P}_{\mu^n_{\Phi(\cdot)}}\Bigg[ \int_0^T a_s \, (\varphi_s(1)-\overrightarrow\varphi_s^{\ell_0}(0))ds>\delta\Bigg] 
\le -\kappa_0 \frac{\delta^{2}}{\varepsilon \displaystyle\int_0^T |a_s|^2\,ds} + T/\kappa_0, 
\end{split}
\end{equation}
where $\ell_0=\varepsilon n^{\gamma-1}$ and $\overrightarrow\varphi_s^{\ell_0}(0)$ is defined in \eqref{eq:emp_average}. In particular we have
\begin{equation}
\label{rep1}
\limsup_{\varepsilon \rightarrow 0}\limsup_{n \rightarrow \infty}\frac{1}{n}\log\mathbb{P}_{\mu^n_{\Phi(\cdot)}}\Bigg[\Big|\int_0^T a_s\,  (\varphi_s(1)-\overrightarrow\varphi_s^{\ell_0}(0))ds\Big|>\delta\Bigg]=-\infty,
\end{equation}
{The same result above is true when $\varphi_s(1)$ is replaced by $\varphi_s (n-1)$  and $\overrightarrow\varphi_s^{\ell_0}(0)$  is replaced by $\overleftarrow\varphi_s^{\ell_0}(n)$.}
\end{lemma}

\begin{proof}
By \eqref{sum_log_super} it is enough to estimate  
\begin{equation}
\begin{split}
& 
\max_{\theta \in \{-1, 1\}}  \left\{ \limsup_{n \to \infty} \frac 1n \log \mathbb{P}_{\mu^n_{\Phi(\cdot)}}\Bigg[ \int_0^T (\theta a_s) \, (\varphi_s(1)-\overrightarrow\varphi_s^{\ell_0}(0))ds>\delta\Bigg] \right\}.
\end{split}
\end{equation} 
It is sufficient to estimate
\begin{equation}
\label{eq:final}
\frac 1n \log \mathbb{P}_{\mu^n_{\Phi(\cdot)}}\Bigg[ \int_0^T a_s\,  (\varphi_s(1)-\overrightarrow\varphi_s^{\ell_0}(0))ds>\delta\Bigg]
\end{equation}
since by changing the function $a$ into the function $-a$ gives the same final bound. From the exponential Chebychev's inequality, the probability in \eqref{eq:final} is bounded from above by
	\begin{equation}
	\exp\{-B\delta n\} \,\mathbb E_{\mu^n_{\Phi(\cdot)}}\left[\exp\Big\{Bn\int_0^T a_s \, (\varphi_s(1)-\overrightarrow\varphi_s^{\ell_0}(0))ds\Big\}\right] \,,
	\end{equation}
	for   any  $B>0$.  From Feynman-Kac's formula, last expectation is  bounded from above by
\begin{equation}
	\exp\Bigg\{\int_0^T  \sup_{f}\Bigg\{B n\int_{\Omega_n}  a_s \, (\varphi(1)-\overrightarrow\varphi^{\ell_0}(0)) f(\varphi) d\mu^n_{\Phi(\cdot)}(\varphi) +\langle  \mathcal L^n f, f\rangle_{\mu^n_{\Phi(\cdot)}} \Bigg\} \, ds \Bigg\},
	\end{equation}
where the supremum is carried on continuously differentiable  functions $f:\mathbb R^{\Lambda_n} \to [0,\infty)$ such that $\int f^2 d\mu_{\Phi(\cdot)}^n=1$. From this, it follows that \eqref{eq:final} is bounded from above by 
\begin{equation}
	-B\delta+ \int_0^ T \sup_{f}\Bigg\{B a_s \int_{\Omega_n} (\varphi(1)-\overrightarrow\varphi^{\ell_0}(0)) f^2(\varphi) d\mu^n_{\Phi(\cdot)} (\varphi)+\frac 1n\langle  \mathcal L^n{f},{f}\rangle_{\mu^n_{\Phi(\cdot)}} \Bigg\} \, ds  .
	\end{equation}	
From Remark \ref{rmk:bound_dirichlet form}, the rightmost term in last display is bounded from above by 
		\begin{equation}
 \int_0^ T  \sup_{f}\Bigg\{B a_s \int_{\Omega_n}(\varphi(1)-\overrightarrow\varphi^{\ell_0}(0)) f^2(\varphi) d\mu^n_{\Phi(\cdot)}(\varphi) -\frac{C_0}{2n}{\mathfrak D}^{n}({f},\mu_{\Phi(\cdot)}^n) + \frac{1}{C_0}+\frac{n^{\gamma-3}}{C_0}\Bigg\} \, ds 
	\label{FK}
	\end{equation}
where $C_0$ is a positive constant independent of $f$ and $n$ (but depending on $\Phi$ and $\gamma$). 
Note that 
	\begin{equation}
	\begin{split}
		Ba_s \int_{\Omega_n}(\varphi(1)-\overrightarrow\varphi^{\ell_0}(0)) f^2(\varphi) d\mu^n_{\Phi(\cdot)}(\varphi)=&
\frac{Ba_s }{\ell_0}\int_{\Omega_n} \sum_{y =1}^{\ell_0}(\varphi(1)-\varphi (y))f^2(\eta)d\mu^n_{\Phi(\cdot)}(\varphi)\\=&\frac{Ba_s }{\ell_0}\int_{\Omega_n} \sum_{y=1}^{\ell_0}\sum_{z=1}^{y-1}(\varphi(z)-\varphi(z+1))f^2(\varphi)d\mu^n_{\Phi(\cdot)}(\varphi).
\end{split}	
\end{equation}
	
Now we use the  identity \eqref{eq:useful_id} and an integration by parts to rewrite the last display as
		\begin{equation}\begin{split}
	&\frac{Ba_s}{\ell_0}\frac{1}{{\mathcal Z}^n_{\Phi (\cdot)}}\int_{\Omega_n} \sum_{y=1}^{\ell_0}\sum_{z=1}^{y-1}(\varphi(z)-\varphi(z+1))f^2(\varphi)\exp\left( -\mathcal E (\varphi) +\sum_x \varphi(x)\Phi(\tfrac xn)\right)d\varphi\\
	&=	\frac{Ba_s}{\ell_0}\int_{\Omega_n} \sum_{y=1}^{\ell_0}\sum_{z=1}^{y-1}(\Phi(\tfrac zn)-\Phi(\tfrac{z+1}{n}))f^2(\varphi)d\mu_{\Phi(\cdot)}^n(\varphi)\\
	&+\frac{Ba_s}{\ell_0}\frac{1}{{\mathcal Z}^n_{\Phi (\cdot)}}\int_{\Omega_n} \sum_{y=1}^{\ell_0}\sum_{z=1}^{y-1}\Big(\partial_{\varphi(z)}-\partial_{\varphi(z+1))}\Big)f^2(\varphi)\exp\left( -\mathcal E (\varphi) +\sum_x \varphi(x)\Phi(\tfrac xn)\right)d\varphi.
	\end{split}
	\end{equation}
	
	Note that since  $\int f^2 d\mu_{\Phi(\cdot)}^n=1$  and $\Phi$ is a Lipschitz function, the term on the second line of last display is bounded from above by  $B |a_s| \Vert \Phi' \Vert_{\infty}   \ell_0/n$. Now, using the identity $\partial_{\varphi(z)}f^2(\varphi )=2{f(\varphi)}\partial_{\varphi(z)}{f(\varphi)}$, we rewrite the last term in the last display as 
	\begin{equation}\begin{split}
\frac{2Ba_s }{\ell_0}\int_{\Omega_n} \sum_{y=1}^{\ell_0}\sum_{z=1}^{y-1}\Big(\partial_{\varphi(z)}{f(\varphi)}-\partial_{\varphi(z+1)}{f(\varphi)}\Big){f(\varphi)}d\mu^n_{\Phi(\cdot)}(\varphi).
	\end{split}
	\end{equation}	
	From Young's inequality, we bound last display from above, for any $A>0$, by a constant times (this constant depends on $p(1)$)
	\begin{equation}
	\frac{B|a_s|}{A{\ell_0}}\int_{\Omega_n} \sum_{y =1}^{\ell_0}\sum_{z=1}^{y-1}\Big(\partial_{\varphi(z)}{f(\varphi)}-\partial_{\varphi(z+1)}{f(\varphi)}\Big)^2d\mu^n_{\Phi(\cdot)}(\varphi)
	+ \frac{AB|a_s|}{{\ell_0}}\int_{\Omega_n} \sum_{y =1}^{\ell_0}\sum_{z=1}^{y-1}f^2(\varphi)d\mu^n_{\Phi(\cdot)}(\varphi).
	\end{equation}
	Since  $\int |f|^2 d\mu_{\Phi(\cdot)}^n=1$, the last display can be bounded from above by
\begin{equation}
\label{eq:useful}
\frac{B|a_s|}{A}{\mathfrak D}_b^n ({f}, \mu^n_{\Phi(\cdot)})+ AB|a_s|\ell_0.
\end{equation}
    Therefore, we can bound \eqref{FK} by 
    \begin{equation}
        \frac{B|a_s|}{An^{\gamma}}{\mathfrak D}^n({f},\mu^n_{\Phi(\cdot)})+B|a_s| \Big( A + \frac{\Vert \Phi' \Vert_{\infty}}{n} \Big)\ell_0-\frac{C_0}{2n}{\mathfrak D}^n({f},\mu^n_{\Phi(\cdot)})+\frac{1}{C_0}+\frac{n^{\gamma-3}}{C_0}.
    \end{equation}
	Taking $A=2BC^{-1}_0|a_s|/n^{\gamma-1}$, we can bound  \eqref{FK} from above by 
\begin{equation}
\begin{split}
B |a_s|\ell_0 \left(\frac{2B}{C_0}|a_s| n^{1- \gamma} + \frac{\Vert \Phi' \Vert_{\infty}}{n} \right) +\frac{1}{C_0}+\frac{n^{\gamma-3}}{C_0}.
\end{split}
\end{equation}
From this, since $\ell_0= \varepsilon n^{\gamma-1}$,  it follows that
\begin{equation}
\begin{split}
\limsup_{n \to \infty} \frac 1n \log \mathbb{P}_{\mu^n_{\Phi(\cdot)}}\Bigg[ \int_0^T a_s (\varphi_s(1)-\overrightarrow\varphi_s^{\ell_0}(0))ds>\delta\Bigg] 
\le - B\delta +  2 \cfrac{B^2}{C_0} \,  \varepsilon \int_0^T |a_s|^2 ds +\frac{t}{C_0}. 
\end{split}
\end{equation}
Minimizing over $B>0$ we finally obtain
\begin{equation}
\begin{split}
\limsup_{n \to \infty} \frac 1n \log \mathbb{P}_{\mu^n_{\Phi(\cdot)}}\Bigg[ \int_0^T a_s (\varphi_s(1)-\overrightarrow\varphi_s^{\ell_0}(0))ds>\delta\Bigg] 
\le -\frac{\delta^{2}}{8 C_0\varepsilon \displaystyle\int_0^T |a_s|^2\,ds} + \frac{t}{C_0}
\end{split}
\end{equation}
and the result follows by taking $\kappa_0=\min\Big(\tfrac{1}{8C_0}, C_0\Big)$.
\end{proof}

\begin{remark}
We observe that in the previous proof we have only used nearest-neighbour jumps in order to obtain a replacement of occupation sites by averages in bigger boxes. This was not a simplification of the argument, in the sense that these jumps are precisely those  that give the biggest possible size of the box as $\ell_0=\varepsilon n^{\gamma-1}$, in order to initiate the multi-scale argument that we employ below. This argument is developed with a view towards  systematically increasing the size of the box  $\ell_0=\varepsilon n^{\gamma-1}$ up to $\varepsilon n$. 
\end{remark}

Now we prove that, in fact, last replacement can be done in a microscopic box of size $\varepsilon n$. 

\begin{lemma}
\label{lem:two-blocks}
There exists a positive constant $\kappa_0$ depending only on $\Phi$ and $\gamma$ such that for any square integrable function $a:[0,T] \to \mathbb R$, for any $t\in[0,T]$ and for any  $\varepsilon \le {\delta} \Big(2 \Vert \Phi'\Vert_{\infty} \int_0^T |a_s| ds\Big)$, it holds{\footnote{Below, if $\int_0^T |a_s|^2ds =0$, the inequality is trivially true with the convention $1/0= +\infty$.}}
\begin{equation}
\begin{split}
\limsup_{n\rightarrow \infty}\frac 1n\log\mathbb{P}_{\mu^n_{\Phi(\cdot)}}\Bigg[\Big|\int_0^T a_s \, (\overrightarrow \varphi^{\ell_0}_s(0)-\overrightarrow \varphi_s^{\varepsilon n}(0))ds\Big|>\delta\Bigg] \le - \kappa_0 \cfrac{\delta^2}{\int_0^T |a_s|^2 ds} \varepsilon^{1-\gamma} + \frac{t}{\kappa_0}.
\end{split}
\end{equation}
where $\ell_0=\varepsilon n^{\gamma-1}$. In particular, for any $t\in[0,T]$ and any  $\delta>0$,
\begin{equation}
\limsup_{\varepsilon\rightarrow 0}\limsup_{n\rightarrow \infty}\frac 1n\log\mathbb{P}_{\mu^n_{\Phi(\cdot)}}\Bigg[\Big|\int_0^T a_s \, (\overrightarrow \varphi^{\ell_0}_s(0)-\overrightarrow \varphi_s^{\varepsilon n}(0))ds\Big|>\delta\Bigg]=-\infty.
\end{equation}

The results above are also true when both $\overrightarrow\varphi_s^{\ell_0}(0)$ and   $\overrightarrow \varphi_s^{\varepsilon n}(0)$ are  replaced by $\overleftarrow\varphi_s^{\ell_0}(n)$ and $\overleftarrow \varphi_s^{\varepsilon n}(n)$, respectively.
\end{lemma}

\begin{proof}
We follow exactly the same steps as in the previous proof. Therefore we are reduced to estimate, for any $B>0$, 
\begin{equation}\exp\{-B\delta n\} \,\mathbb E_{\mu^n_{\Phi(\cdot)}}\left[\exp\Big\{Bn\int_0^T  a_s \, (\overrightarrow \varphi^{\ell_0}_s(0)-\overrightarrow \varphi_s^{\varepsilon n}(0))ds\Big\}\right].
\end{equation}
From the exponential Chebychev's  inequality and Feynman-Kac's formula, we obtain that 
\begin{equation}
\label{FK2}
\begin{split}
&\frac 1n\log\mathbb{P}_{\mu^n_{\Phi(\cdot)}}\Bigg[\int_0^T a_s \, (\overrightarrow \varphi^{\ell_0}_s(0)-\overrightarrow \varphi_s^{\varepsilon n}(0))ds>\delta\Bigg]\\
&\leq -B\delta+ \int_0^T  
\sup_f\Bigg\{B a_s \int_{\Omega_n} (\overrightarrow \varphi^{\ell_0}_s(0)-\overrightarrow \varphi_s^{\varepsilon n}(0))f^2(\varphi)d\mu^n_{\Phi(\cdot)}(\varphi)+\frac{1}{n} \langle ~\mathcal L_H^n{f},{f}\rangle_{\mu^n_{\Phi(\cdot)}} \Bigg\} \, ds \ ,
\end{split}
\end{equation}
where the supremum is carried on continuously differentiable  functions $f:\mathbb R^{\Lambda_n} \to [0,\infty)$ such that $\int |f|^2 d\mu_{\Phi(\cdot)}^n=1$. From Remark \ref{rmk:bound_dirichlet form}, the rightmost term on the right-hand side of last inequality is bounded from above by 
\begin{equation}
\begin{split}
\int_0^T   \sup_f\Bigg\{B a_s \int_{\Omega_n} (\overrightarrow \varphi^{\ell_0}_s(0)-\overrightarrow \varphi_s^{\varepsilon n}(0))f^2(\varphi)d\mu^n_{\Phi(\cdot)}(\varphi)-\frac{C_0}{2n}{\mathfrak D}^{n}({f},\mu_{\Phi(\cdot)}^n) + \frac{1}{C_0}+\frac{n^{\gamma-3}}{C_0} \Bigg\} \, ds,
\end{split}
\end{equation}
where $C_0$ is a constant independent of $f$ and $n$ (but depending on $\Phi$ and $\gamma$).  We define  $\ell_i:=2^i\ell_0$, and choose $M:=M_n$ such that $2^M\ell_0=\varepsilon n$. Note that  $M$ is order $\log (n)$ since $\ell_0 =\varepsilon n^{\gamma-1}$. Then
	\begin{equation}
	\label{path1}
	\begin{split}
\overrightarrow \varphi^{\ell_0}_s(0)-\overrightarrow \varphi_s^{\varepsilon n}(0)=\sum_{i=1}^M\left[\frac{1}{\ell_{i-1}}\sum_{y=1}^{\ell_{i-1}}\varphi(y)-\frac{1}{\ell_{i}}\sum_{y=1}^{\ell_{i}}\varphi(y)\right]=\sum_{i=1}^M\frac{1}{\ell_i}\sum_{y=1}^{\ell_{i-1}}\Big[\varphi(y)-\varphi(y+\ell_{i-1})\Big] \ .
	\end{split}	
	\end{equation}
Now, the first term in the supremum of \eqref{FK2} can be rewritten as 
\begin{equation}
\begin{split}
&\int_{\Omega_n} B a_s \sum_{i=1}^M\frac{1}{\ell_i}\sum_{y=1}^{\ell_{i-1}}(\varphi(y)-\varphi(y+\ell_{i-1}))f^2(\varphi)d\mu^n_{\Phi(\cdot)}(\varphi)\\
&=Ba_s \sum_{i=1}^M\frac{1}{\ell_i}\frac{1}{\mathcal Z^n_{\Phi(\cdot)}}\int_{\Omega_n} \sum_{y=1}^{\ell_{i-1}}(\varphi(y)-\varphi(y+\ell_{i-1}))f^2(\varphi)\exp\left( -\mathcal E (\varphi) +\sum_x \varphi(x)\Phi(\tfrac xn)\right)d\varphi\\
&=Ba_s\int_{\Omega_n} \sum_{i=1}^M\frac{1}{\ell_i}\sum_{y=1}^{\ell_{i-1}}\Big(\Phi(\tfrac yn)-\Phi(\tfrac{y+\ell_{i-1}}{n})\Big)f^2(\varphi)d\mu_{\Phi(\cdot)}^n(\varphi)\\
&+Ba_s \sum_{i=1}^M\frac{1}{\ell_i}\frac{1}{\mathcal Z_{\Phi(\cdot)}}\int_{\Omega_n}\sum_{y=1}^{\ell_{i-1}}\Big(\partial_{\varphi(y)}-\partial_{\varphi(y+\ell_{i-1}))}\Big)f^2(\varphi)\exp\left( -\mathcal E (\varphi) +\sum_x \varphi(x)\Phi(\tfrac xn)\right)d\varphi \ .
\end{split}
\end{equation}
Since $f$ satisfies $\int f^2 d\mu_{\Phi(\cdot)}^n=1$, the first term in last display can be bounded from above by
\begin{equation}
\Vert \Phi'\Vert_{\infty} \frac{B|a_s|}{n}\int_{\Omega_n} \sum_{i=1}^M\frac{\ell_{i-1}^2}{\ell_i}\leq  \Vert\Phi'\Vert_{\infty} \frac{2^M\ell_0B |a_s|}{n}=\varepsilon \Vert \Phi'\Vert_{\infty} B|a_s| . 
\end{equation}
Now,   using the identity $\partial_{\varphi(z)}f^2(\varphi)=2{f(\varphi)}\partial_{\varphi(z)}{f(\varphi)}$, we rewrite the remaining  term as 
\begin{equation}
\begin{split}	
Ba_s \sum_{i=1}^M\frac{1}{\ell_i}\int_{\Omega_n} \sum_{y=1}^{\ell_{i-1}}\Big(\partial_{\varphi(y)}{f(\varphi)}-\partial_{\varphi(y+\ell_{i-1})}{f(\varphi)}\Big){f(\varphi)}d\mu^n_{\Phi(\cdot)}(\varphi).
\end{split}
\end{equation}  
From Young's inequality, for any choice of positive constants $A_i$'s,  the last display can be  bounded from above by 
\begin{equation}
\begin{split}	
\label{eq:mpl1}
&B|a_s| \sum_{i=1}^M\frac{1}{2\ell_i}\sum_{y=1}^{\ell_{i-1}}\Bigg(\frac{1}{A_i}\int_{\Omega_n} \Big(\partial_{\varphi(y)}{f(\varphi)}-\partial_{\varphi(y+\ell_{i-1})}{f(\varphi)}\Big)^2d\mu^n_{\Phi(\cdot)}(\varphi)+\int_{\Omega_n} A_i f(\varphi)d\mu^n_{\Phi(\cdot)}(\varphi) \Bigg)\\
=&B|a_s| \sum_{i=1}^M\frac{1}{2A_i\ell_i}\sum_{y=1}^{\ell_{i-1}}I_{y,y+\ell_{i-1}}({f},\mu_{\Phi(\cdot)}^n)+B|a_s| \sum_{i=1}^M\frac{A_i\ell_{i-1}}{\ell_i} .\end{split}
\end{equation}

For the choice $A_i=\frac{A_0}{C_0} \frac{B|a_s|\ell_{i-1}^{\gamma}}{ n^{\gamma-1}\ell_i}$, $1 \le i \le M$,  with $A_0$  being the constant  appearing in Lemma \ref{posMPL} stated below, we can apply Lemma \ref{posMPL} to bound the right-hand side of \eqref{eq:mpl1} by 
$$ \frac{C_0}{2n}{\mathfrak D}^n(\sqrt{f},\mu_{\Phi(\cdot)}^n)+\frac{A_0}{C_0} \frac{B^2 |a_s|^2}{n^{\gamma-1}}\sum_{i=1}^M\frac{\ell_{i-1}^{\gamma+1}}{\ell^2_i} \le  \frac{C_0}{2n}{\mathfrak D}^n(\sqrt{f},\mu_{\Phi(\cdot)}^n)+  \frac{A_0}{C_0} \frac{B^2 |a_s|^2 \varepsilon^{\gamma-1}}{4(2^{\gamma-1} - 1)}.$$
As in the previous lemma, the proof ends  by taking first the limsup as $n\to+\infty$ and then minimizing over $B>0$. Since, for $\varepsilon\leq \delta/\big(2 \Vert \Phi'\Vert_{\infty} \int_0^T |a_s| ds\big)^{-1}$, we have
\begin{equation}
\begin{split}
\inf_{B>0}\Big\{ -\delta B + \frac{A_0}{C_0} \int_0^T &\frac{B^2 |a_s|^2 \varepsilon^{\gamma-1}}{4(2^{\gamma-1} - 1)} \, ds  +  B\Vert \Phi'\Vert_{\infty} \varepsilon \int_0^T  |a_s| ds + \frac{t}{C_0}\Big\}\\
&\le \inf_{B>0}  \left\{ -\cfrac{\delta}{2} B +  \frac{A_0}{C_0} \int_0^T \frac{B^2 |a_s|^2 \varepsilon^{\gamma-1}}{4(2^{\gamma-1} - 1)} \, ds  \right\} + \frac{t}{C_0} \\
&= - \delta^2 \left( \frac{4A_0}{C_0} \int_0^T \frac{ |a_s|^2 \varepsilon^{\gamma-1}}{(2^{\gamma-1} - 1)} \, ds\right)^{-1} + \frac{t}{C_0},
\end{split}
\end{equation}
we complete the proof by taking $\kappa_0 = \min \left(  \frac{4A_0}{C_0} \int_0^T \frac{ |a_s|^2 \varepsilon^{\gamma-1}}{(2^{\gamma-1} - 1)} \,ds, \  C_0 \right).$

\end{proof}

The next result is quite similar to the moving particle Lemma 5.8 established in \cite{BCGS2023spa} but  adapted to the dynamics studied in this article.

\begin{lemma} 
\label{posMPL}
Let $\ell_0$ a positive integer and $M:=M_n$ such that $2^M\ell_0=\varepsilon n$. Define then  $\ell_i:=2^i\ell_0$ for $i=0, \ldots, M$. Let $f:\mathbb R^{\Lambda_n} \to [0,\infty)$ be a continuously differentiable  function such that $\int |f|^2 d\mu_{\Phi(\cdot)}^n=1$. Then, there exists a constant $A_0$ depending only on $\Phi$ and $\gamma$ such that 
	\begin{align}
	\sum_{i=1}^M \sum_{y=1}^{\ell_{i-1}}  \cfrac{I_{y, y+\ell_{i-1}} \Big({f}, \mu_{\Phi(\cdot)}^n \Big)}{\ell_{i-1}^{\gamma}} \le A_0 \,    {\mathfrak D}_b^n  \Big({f}, \mu_{\Phi(\cdot)}^n \Big ) \ .
	\end{align}
\end{lemma}

\begin{proof}
Let us assume without loss of generality that $\ell_0$ is even (for $\ell_0$ odd it is easy to adapt the argument) and then $\ell_{i-1}$ is an even number for any $i \in \{1, \ldots,M\}$.  First fix $i \in \{1, \ldots, M\}$. For every $y \in \{1, \ldots, \ell_{i-1}\}$ and any $j \in \{1,\dots, \tfrac{\ell_{i-1}}{2}\}$ we sum and subtract an intermediate derivative at a point $y_j:=y+\ell_{i-1}/2+j$ so that  by the inequality $(x+y)^2\leq 2x^2+2y^2$ we get 
 \begin{equation}
 I_{y,y+\ell_{i-1}}  \Big({f}, \mu_{\Phi(\cdot)}^n \Big) \leq 2  I_{y,y_j}  \Big({f}, \mu_{\Phi(\cdot)}^n \Big) + 2 I_{y_j,y+\ell_{i-1}}  \Big({f}, \mu_{\Phi(\cdot)}^n \Big) \ .
 \end{equation}
 Since the distance between $y$ and $ y_j$ or  between $y_j$ and $y+ {\ell_{i-1}}$ is at most  $\ell_{i-1}$, we have that $[p (y-y_j )]^{-1}  \lesssim \ell_{i-1}^{1+\gamma}$ and $[p (y_j-(y+{\ell_{i-1}}) )]^{-1}  \lesssim \ell_{i-1}^{1+\gamma}$. From this it follows that 
\begin{equation}\begin{split}
 I_{y,y+\ell_{i-1}}  \Big({f}, \mu_{\Phi(\cdot)}^n \Big) &\leq 2 \ell_{i-1}^{1+\gamma} [p (y-y_j )]^{-1}  I_{y,y_j}  \Big({f}, \mu_{\Phi(\cdot)}^n \Big) \\&+ 2 \ell_{i-1}^{1+\gamma} [p (y_j-(y+{\ell_{i-1}}) )]^{-1} I_{y_j,y+\ell_{i-1}}  \Big({f}, \mu_{\Phi(\cdot)}^n \Big) \ .
 \end{split}
 \end{equation}
 Since last identity holds  for any $j \in \{1,\dots, \tfrac{\ell_{i-1}}{2}\}$ we obtain 
\begin{equation}\begin{split}
\ell_{i-1} I_{y,y+\ell_{i-1}}  \Big({f}, \mu_{\Phi(\cdot)}^n \Big) &\lesssim  \ell_{i-1}^{1+\gamma} \sum_{j=1}^{\ell_{i-1}/2}  \Big\{ [p (y-y_j )]^{-1}  I_{y,y_j}  \Big({f}, \mu_{\Phi(\cdot)}^n \Big) \\&+  [p (y_j-(y+{\ell_{i-1}}) )]^{-1} I_{y_j,y+\ell_{i-1}}  \Big({f}, \mu_{\Phi(\cdot)}^n \Big) \Big\}\end{split}
 \end{equation}
so that
\begin{equation}
\begin{split}
\sum_{i=1}^MI_{y,y+\ell_{i-1}}  \Big({f}, \mu_{\Phi(\cdot)}^n \Big) &\lesssim  \ell_{i-1}^{\gamma}\sum_{i=1}^M \sum_{j=1}^{\ell_{i-1}/2}  \Big\{ [p(y-y_j )]^{-1}  I_{y,y_j}  \Big({f}, \mu_{\Phi(\cdot)}^n \Big) \\&+  [p (y_j-(y+{\ell_{i-1}}) )]^{-1} I_{y_j,y+\ell_{i-1}}  \Big({f}, \mu_{\Phi(\cdot)}^n \Big) \Big\} \ .
\end{split}
 \end{equation}
 Note that since we do not repeat the same bonds when changing  $i,y,j,y_j$, (for details we refer to the proof of  Lemma 5.8 in \cite{BCGS2023spa}) we can bound last display by 
\begin{equation}\label{64}
\underset{z\leq w}{ \sum_{z,w \in \Lambda_n}} p (w-z) I_{z,w} ({f}, \mu_{\Phi(\cdot)}^n ) \lesssim {\mathfrak D}_b^n ({f},\mu_{\Phi(\cdot)}^n) \ .
\end{equation}
And this finishes the proof. 
\end{proof}

\begin{lemma}
\label{RL:boundary}
For any square integrable function $a:[0,T] \to \mathbb R$, for any $t\in[0,T]$ and for any $\delta>0$, it holds
\begin{equation}
\limsup_{n \rightarrow \infty}\frac{1}{n}\log\mathbb{P}_{\mu^n_{\Phi(\cdot)}}\Bigg[\Big|\int_0^T a_s \, \Big(\varphi_s(1)-\Phi_{\ell}\Big)\,ds\Big|>\delta\Bigg] = - \infty.
\end{equation}
The same result holds  replacing $\varphi_s(1)$ by $\varphi_s(n-1)$ and  $\Phi_{\ell}$ by $\Phi_{r}$.
\label{rep1}
\end{lemma}

\begin{proof}
The proof follows exactly the same steps as in  the proof of Lemma \ref{lem:one_block} and for that reason we only sketch it. We present the proof for the left boundary, but for the right it is completely analogous. First we claim that  that for any continuously differentiable  function $f:\mathbb R^{\Lambda_n} \to [0,\infty)$ such that $\int f^2 d\mu_{\Phi(\cdot)}^n=1$ and for any positive constant $A$, it holds 
	\begin{equation}
	\label{eq:est_left_bound}
	\begin{split}
	\left\vert \left\langle \varphi(1)-\Phi_\ell,f^2\right\rangle_{\mu_{\Phi(\cdot)}^n} \right\vert &\le \tfrac{1}{2A}  {\mathfrak D}_\ell({f},\mu_{\Phi(\cdot)}^n)+ \tfrac{A}{2},
	\end{split}
	\end{equation}
The proof of the claim relies on \eqref{eq:usef_id_2} and  a change of variables, which allows writing 
\begin{equation}\begin{split}
	\left\vert \left\langle \varphi(1)-\Phi_\ell ,f^2\right\rangle_{\mu_{\Phi(\cdot)}^n} \right\vert &=	\left\vert \frac{1}{\mathcal Z^n_{\Phi(\cdot)}}\int_{\Omega_n}\big(\varphi(1)-\Phi_\ell\big)\exp\left( -\mathcal E (\varphi) +\Phi_\ell \mathcal V (\varphi) \right) f^2(\varphi)d\varphi \right\vert  \\&=
\left\vert 	\frac{1}{\mathcal Z^n_{\Phi(\cdot)}}\int_{\Omega_n} \partial_{\varphi(1)}\exp\left( -\mathcal E (\varphi) +\Phi_\ell \mathcal V (\varphi) \right) f^2(\varphi)d\varphi \right\vert  
\\&=
\left\vert\int_{\Omega_n} \partial_{\varphi(1)}f^2(\varphi)d\mu_{\Phi(\cdot)}^n(\varphi) \right\vert.   \end{split}
	\end{equation}
Now Young's inequality proves the claim.  
Now, by Chebichev's inequality, for any $B>0$,
\begin{equation}
\begin{split}
\mathbb{P}_{\mu^n_{\Phi(\cdot)}}\Bigg[\Big|\int_0^T a_s \, \Big(\varphi_s(1)-\Phi_{\ell}\Big)\,ds\Big|>\delta\Bigg] & \le e^{-n \delta B } \mathbb E_{\mu^n_{\Phi(\cdot)}}\left[ e^{ nB \vert\int_0^T a_s \, \big(\varphi_s(1)-\Phi_{\ell}\big)\,ds\vert}\right].
\end{split}
\end{equation}
Putting together \eqref{sum_log_super}, Feynman-Kac's formula and Remark \ref{rmk:bound_dirichlet form}, we have 
\begin{equation}
\begin{split}
&\limsup_{n \to \infty} \cfrac{1}{n} \log \mathbb{P}_{\mu^n_{\Phi(\cdot)}}\Bigg[\Big|\int_0^T a_s \, \Big(\varphi_s(1)-\Phi_{\ell}\Big)\,ds\Big|>\delta\Bigg]\\
&\le - \delta B + \limsup_{n \to \infty} \sup_{\theta =\pm 1} \left\{ \int_0^ T \sup_{f} \left\{  B \theta a_s \int_{\Omega_n}  \big(\varphi (1)-\Phi_{\ell}\big) d\mu^n_{\Phi(\cdot)}(\varphi)+\frac{1}{n} \langle ~\mathcal L^n{f},{f}\rangle_{\mu^n_{\Phi(\cdot)}} \right\}  ds \right\}\\
&\le - \delta B + \limsup_{n \to \infty} \int_0^ T \sup_{f} \left\{  B |a_s| \int_{\Omega_n}  \big(\varphi (1)-\Phi_{\ell}\big) d\mu^n_{\Phi(\cdot)} (\varphi) -\frac{C_0}{2n}{\mathfrak D}^{n}({f},\mu_{\Phi(\cdot)}^n) + \frac{1}{C_0} \right\} ds \\
&\le  - \delta B + \frac{T}{C_0} +\limsup_{n \to \infty} \int_0^T \sup_{f} \left\{  B |a_s| \int_{\Omega_n}  \big(\varphi (1)-\Phi_{\ell}\big) d\mu^n_{\Phi(\cdot)}(\varphi)  -\frac{C_0}{2} n^{\gamma -1}  {\mathfrak D}_\ell({f},\mu_{\Phi(\cdot)}^n)  \right\}  ds,
\end{split}
\end{equation}
where the supremum is carried over continuously differentiable  functions $f:\mathbb R^{\Lambda_n} \to [0,\infty)$ such that $\int f^2 d\mu_{\Phi(\cdot)}^n=1$. By using \eqref{eq:est_left_bound} and optimizing over the constant $A$ appearing there, and sending then $B$ to infinity, we get the result.
\end{proof}

\subsection{Upper Bound} 
\label{subsec:UBD}

In this section we prove that for any closed set $\mathcal F$ of $C([0,T],\mathcal M)$, it holds
    \begin{equation}
    \label{eq:UBclosed}
        \begin{split}
            &\limsup_{n\to\infty}\frac{1}{n}\log\mathbb Q_{n}(\pi \in\mathcal F)\leq -\inf_{\pi\in\mathcal F} I_{[0,T]}(\pi|g).
        \end{split}
    \end{equation}

The first step consists in proving exponential tightness for the sequence $\big(\mathbb{Q}_{n}\big)_{n\ge 2}$. This general fact permits then  to reduce the proof of the claim \eqref{eq:UBclosed} to the same claim but for $\mathcal F$ being a compact subset of $C([0,T],\mathcal M)$. 
It is possible to achieve this with \cite[Lemma 2.4]{quastel1995large}, which in our setting is a direct consequence of Propositions \ref{prop: tightness estimate} and  \ref{prop:tightness estimate II}. Hence from now we will assume that $\mathcal F$ is compact and for this reason we denote it  by $\mathcal K$.

\subsubsection{Super-exponential estimates}
We now establish that, as a consequence of some results proved earlier in the paper, the sets $\mathcal K_{\ell},\;\mathcal D_{\ell}$ and $\mathcal B_{\varepsilon,\delta}$, defined below, are superexponentially small. Those properties will be crucial in the proof of the upper bound.

For any $\ell>0$, consider the set 
\begin{equation}
\label{eq:Kell}
\begin{split}
\mathcal K_\ell = \{ \pi \in C([0,T], \mathcal M) \; ; \; \sup_{0 \le t \le T} \Vert \pi_t \Vert_{TV} < \ell\}. 
\end{split}
\end{equation}
By Proposition \ref{prop: tightness estimate}, we have that 
\begin{equation}
\label{eq:setv}
\begin{split}
\limsup_{\ell \to \infty}  \limsup_{n \to \infty} \frac1n \log \mathbb Q_n (\mathcal K^c_\ell) = - \infty.
\end{split}
\end{equation}

Moreover, by Lemma \ref{lem:acL2-superexp} and Lemma \ref{lem:EGPI}, we know that there exists a constant $C>0$ such that for any $k \ge 1$, any sequence $G^{(k)}:=(G^j)_{1 \le j \le k} \in C^{0,\infty}_c (\Omega_T)$, any $\ell \ge 1$,
\begin{equation}
\label{eq:setv1}
\begin{split}
\limsup_{n \to \infty} \cfrac{1}{n} \log \mathbb Q_n \left( {\mathcal D}^c_{\ell, G^{(k)}} \right) \le -C \ell + C^{-1},
\end{split}
\end{equation}
where
\begin{equation}
\label{eq:Dell}
\begin{split}
{\mathcal D}_{\ell, G^{(k)}} = \left\{ \pi \in C([0,T], \mathcal M) \; ; \; \sup_{G \in G^{(k)}} \mathcal E_G (\pi) \leq\ell \quad \text{and} \quad \sup_{G \in G^{(k)}} {\mathcal F}_G (\pi) \leq\ell  \right\} . 
\end{split}
\end{equation}

Recall the definition of averages in \eqref{eq:emp_average}. Let $\iota_\varepsilon^0 = \varepsilon^{-1} {\mathbf 1}_{(0,\varepsilon]}$ and $\iota_\varepsilon^1 = \varepsilon^{-1} {\mathbf 1}_{(1-\varepsilon,1]}$. Observe that { $\overrightarrow{\varphi}_t^{\varepsilon n}(0) = \langle \pi_t^n , \iota_\varepsilon^0 \rangle$} and {$\overleftarrow\varphi^{\varepsilon n}_t(n) = \langle \pi_t^n, \iota_\varepsilon^1 \rangle$.} For any sequence $a^{(k')} = (a^1, \ldots,a^{k'})$ of $k' \ge 1$ real valued functions on $[0,T]$ with $L^\infty$-norm equal to one, any $\delta>0$, we introduce then the subsets of $C([0,T], \mathcal M)$ given by 
\begin{equation}
\label{eq:Bell}
\begin{split}
&\mathcal B_{\varepsilon,\delta}^{a^{(k')} }= \bigcup_{j=1}^{k'}\mathcal B_{\varepsilon,\delta}^{\ell, a^j} \ \cup\  \bigcup_{j=1}^{k'}  \mathcal B_{\varepsilon,\delta}^{r, a^j},
\end{split}
\end{equation}
where for $(p,q)\in\{(0,\ell), (1,r)\}$
\begin{equation}
\begin{split}
&\mathcal B_{\varepsilon,\delta}^{q, a^j}= \left\{ \pi \in C([0,T], \mathcal M) \; ; \; \Big\vert  \int_0^T a^j_t \Big( \Phi_q - \langle \pi_t , \iota_\varepsilon^p \rangle \Big) dt \Big\vert \le \delta \right\},  
\end{split}
\end{equation}
We can then rewrite the result of Proposition \ref{prop:charact dirichlet bd cond super exp-0} for $H=0$  as
\begin{equation}
\label{eq:sebc}
\begin{split}
 \limsup_{n \to \infty}\frac{1}{n}\log\mathbb{Q}_{n} \left(  \Big( \mathcal B_{\varepsilon,\delta}^{a^{(k')}}\Big)^c \right) \le - \kappa_0 \delta^2 \varepsilon^{1-\gamma} +\kappa_0^{-1},
\end{split}
\end{equation}
where $\kappa_0$ is a constant depending only on $\Phi_\ell, \Phi_r, \gamma$ and $T$ and not on ${ a}^{(k')}$, $\varepsilon$ and $\delta$.

\subsubsection{Exponential martingale}

Let $H \in C_c^{1,2}(\Omega_T)$ and consider the exponential martingale $\mathbb M^H$ defined, for any $t\in [0,T]$, by
\begin{equation}
    \mathbb M_t^H=\exp\Big\{F(\varphi_t)-F(\varphi_0)-\int_0^t e^{-F(\varphi_s)}(\partial_s+\mathcal{L}_n)e^{F(\varphi_s)}ds\Big\}
\end{equation}
where $F(\varphi_t)=n\langle \pi_t^n, H_t\rangle$. We recall that by Lemma \ref{lem:super_exchange_measures}, for any $d>0$, we have
\begin{equation}
\begin{split}
\mathbb E \left[ \left( \mathbb  M_T^H \right)^{1+d}\right] \le \exp (Cd (1+d) n) 
\end{split}
\end{equation}
where $C>0$ is a constant independent of $n$ and $d$. Straightforward computations show that  
\begin{equation}
\label{eq:exp_martingale}
\begin{split}
     \mathbb M_t^H=\exp\Bigg\{&n\Big[\langle \pi_t^n, H_t\rangle-\langle \pi_0^n, H_0\rangle-\int_0^t\langle\pi^n_s,(\partial_s+\mathbb{L}_n^\gamma) H_s\rangle ds\\
     &-n^{\gamma-1}\int_0^t\Big[H_s\big(\tfrac{n-1}{n}\big)^2+(\Phi_r-\varphi_s(n-1)) H_s\big(\tfrac{n-1}{n}\big)\Big]ds\\
     &-n^{\gamma-1}\int_0^t\Big[H_s\big(\tfrac{1}{n}\big)^2+(\Phi_\ell-\varphi_s(1))H_s\big(\tfrac{1}{n}\big)\Big]ds\\&-\frac{n^{\gamma-1}}{2}\int_0^t\sum_{x,y\in\Lambda_n}p(y-x)(H_s(\tfrac xn)-H_s(\tfrac yn))^2ds\Big]
     \Bigg\}.
     \end{split}
\end{equation}
Define the functional $J^n_H: C([0,T], \mathcal M) \to \mathbb R$ by 
\begin{equation}
\label{eq:JNG}
    \begin{split}
     J^n_H (\pi)=&\langle \pi_T, H_T\rangle-\langle \pi_0 , H_0\rangle-\int_0^T \langle\pi_s,\mathbb{L}^\gamma_n H_s\rangle ds-\int_0^T\langle\pi_s,\partial_s H_s\rangle\;ds\\
    &-n^{\gamma-1} \frac12 \int_0^T\sum_{x,y\in\Lambda_n}p(y-x)(H_s(\tfrac xn)-H_s(\tfrac yn))^2ds \\
     &=\langle \pi_T, H_T\rangle-\langle \pi_0, H_0\rangle-\int_0^T \left\{ \langle\pi_s,(\partial_s+\mathbb{L}^\gamma_n) H_s\rangle- \langle H_s,H_s\rangle_{n,\gamma/2}\right\} \,  ds
     \end{split}
\end{equation}
and define 
\begin{equation}
    W^\ell_{H_t} (\varphi):=H_t(\tfrac 1n)\big\{\Phi_\ell-\varphi (1)\big\}\quad\text{and}\quad W^r_{H_t} (\varphi):=H_t(\tfrac{n-1}{n})\big\{\Phi_r-\varphi (n-1)\big\}.
\end{equation}

Then we can write the exponential mean one martingale $\mathbb M^H$ introduced in \eqref{eq:exp_martingale} at time $T$ as 
\begin{equation}
\label{eq:expmart-Edouard1}
\begin{split}
\mathbb M_T^H = \exp \left[ n\ \left\{ J_H^n(\pi^n) - n^{\gamma-1} \int_0^T \Big(W^\ell_{H_t} (\varphi_t) + W^r_{H_t} (\varphi_t)\Big) +\Big( H_t^2 (\tfrac 1n) + H_t^2 (\tfrac{n-1}{n}) \Big)dt \right\} \right].
\end{split}
\end{equation}

Recall the definition of $J_H( \cdot |g)$ in \eqref{eq:JHPG} and observe that, by \eqref{eq:Sob1} and \eqref{eq:Sob2},
\begin{equation}\label{eq:continuous J and J^n}
\begin{split}
J^n_H (\pi) = J_H (\pi| \pi_0) + \varepsilon_n (H) T \sup_{0 \le t \le T} \Vert \pi_t \Vert_{TV}
\end{split}
\end{equation}
with $\lim_{n \to \infty} \varepsilon_n (H) =0$.  
Note that the two rightmost terms in the exponential appearing in \eqref{eq:expmart-Edouard1} are negligible in the following sense. First, since $H\in C_c^{1,2} (\Omega_T)$ and $\gamma<2$ 
\begin{equation}\label{eq:H term little o}
\begin{split}
n^{\gamma-1} \int_0^T \Big( H_t^2 (\tfrac 1n) + H_t^2 (\tfrac{n-1}{n}) \Big)\, dt = O(n^{\gamma -2})=o_H(1).
\end{split}
\end{equation}
Second, the following lemma holds.
\begin{lemma}
\label{lem:superexp RL upper bound}
For every $\delta>0$, for $q\in\{\ell,r\}$
    \begin{equation}
        \limsup_{n\to\infty}\frac{1}{n}\log\mathbb P_n\Bigg(\Bigg|n^{\gamma-1}\int_0^T W_{H_t}^{q}(\varphi_t)\;dt\Bigg|>\delta\Bigg)=-\infty.
    \end{equation}
\end{lemma}

\begin{proof}
   As explained in the proof of Lemma \ref{lem:one_block}, to establish this lemma, it is sufficient to prove it without the absolute value. We present the proof for the left boundary, the right one is similar. We observe that the term $n^{\gamma-1} H_t (\tfrac 1n)$ can be rewritten as $n^{\gamma -2} \ n (H_t (\frac 1n) -H_t(0))$ since $H_t(0) =0$, and is therefore of order $n^{\gamma-2}$. Then we can proceed as in the proof of Lemma \ref{RL:boundary}, the dependence in $n$ of the deterministic term $n^{\gamma-1} H_t (\tfrac 1n)$ being irrelevant.
\end{proof}

\subsubsection{Min-max argument}

Let $(G_j)_{j \ge 1}$ be a sequence of functions in $C_c^{0, \infty} (\Omega_T)$ which are dense in $L^2 ([0,T], \mathcal H_0^{\gamma/2})$, and let $(a^j)_{j \ge 1}$ be a sequence of continuous functions dense in the space of real valued continuous functions on $[0,T]$ with $L_\infty$-norm less than one (so that {$\int_0^T |a^j_s|^2 ds \le T$}).  For any positive real numbers $\ell, \varepsilon, \delta,\delta'$, any integer $k,k'\ge 1$, any $H \in C^{0,\infty}_0 (\Omega_T)$, we define the events 
\begin{equation}
\begin{split}
\mathcal W^n_{H,\delta'} = \left\{ \Bigg|n^{\gamma-1}\int_0^T W_{H_t}^{\ell}(\varphi)\;dt\Bigg|\le \delta' \right\} \cup \left\{ \Bigg|n^{\gamma-1}\int_0^T W_{H_t}^{r}(\varphi)\;dt\Bigg|\le \delta' \right\} \subset C([0,T], \Omega_n),
\end{split}
\end{equation}
and
\begin{equation}
\begin{split}
\mathcal U_{\ell, k, k',  \varepsilon, \delta} =\mathcal K_\ell \cap {\mathcal D}_{\ell, G^{(k)}}  \cap \mathcal B_{\varepsilon,\delta}^{a^{(k')}} \subset C([0,T], \mathcal M),
\end{split}
\end{equation}
where $\mathcal K_\ell$, ${\mathcal D}_{\ell, G^{(k)}}$ and $\mathcal B_{\varepsilon,\delta}^{a^{(k')}}$ have been defined in \eqref{eq:Kell}, \eqref{eq:Dell} and \eqref{eq:Bell} respectively. Thanks to \eqref{eq:setv}, \eqref{eq:setv1} and \eqref{eq:sebc}, there exists a constant $C_0>0$ independent of the parameters $\ell, k, k' ,  \varepsilon, \delta$ and a parameter $\varepsilon_0:=\varepsilon_0(\delta)>0$, going to $0$ as $\delta$ goes to zero, such that for $\varepsilon<\varepsilon_0(\delta)$, 
\begin{equation}
\label{eq:cioran}
\begin{split}
\limsup_{n \to \infty} \frac{1}{n} \log \mathbb Q_n \Big( \mathcal U_{\ell, k, k' ,  \varepsilon, \delta}^c \Big) \leq \max \left\{ - C_0 \ell + C_{0}^{-1} \ , -C_0 \delta^2 \varepsilon^{1-\gamma} +C_0^{-1}   \right\}.
\end{split}
\end{equation}

Let $\mathcal O$ be an open subset of $C([0,T], \mathcal M)$. By \eqref{sum_log_super} and Lemma \ref{lem:superexp RL upper bound}, we have 
\begin{equation}
\label{eq:tyman1}
\begin{split}
&\limsup_{n \to \infty} \frac1n \log \mathbb Q_n (\mathcal O)\\
&\leq\max \left\{ \limsup_{n \to \infty} \frac1n \log \mathbb P_n \big( \pi^n \in \mathcal O \cap  \mathcal U_{\ell, k, k' ,  \varepsilon, \delta} \cap \mathcal W^n_{H,\delta'}\big) \ , \ \limsup_{n \to \infty} \frac1n \log \mathbb P_n (\pi^n \in \mathcal U_{\ell, k, k' ,  \varepsilon, \delta}^c) \right\}.
\end{split}
\end{equation}
We now write
\begin{equation}
\begin{split}
\mathbb P_n \big( \pi^n \in \mathcal O \cap   \mathcal U_{\ell, k, k' ,  \varepsilon, \delta}  \cap \mathcal W^n_{H,\delta'} \big)&= \mathbb E^H_n \left[  {\mathbf 1}_{\pi^n \in \mathcal O \cap   \mathcal U_{\ell, k, k' ,  \varepsilon, \delta} \cap \mathcal W^n_{H,\delta'}} \, \Big( \mathbb M_T^H \Big)^{-1}\right]
\end{split}
\end{equation}
and observe that on $\mathcal U_{\ell, k, k' ,  \varepsilon, \delta} \cap \mathcal W^n_{H,\delta'}$, by \eqref{eq:expmart-Edouard1}, \eqref{eq:continuous J and J^n} and \eqref{eq:H term little o}, 
\begin{equation}
\begin{split}
\Big(\mathbb M_T^H\Big)^{-1} = \exp \left\{ - n \left( J_H (\pi^n | \pi_0^n) + \delta'   + (\ell +1) o_H (1) \right) \right\}.
\end{split}
\end{equation}
Recall \eqref{eq:tyman1}. In view of  \eqref{eq:cioran}\footnote{There is a slight abuse of notation here, since the limit in 
$\ell$ and the optimal parameters has not yet been taken. By \eqref{eq:tyman1}, we should keep the maximum in \eqref{eq:cioran} in the proof. Nevertheless, we adopt this notation for the sake of readability.}, we have that 
\begin{equation}
\begin{split}
&\limsup_{n \to \infty} \frac1n \log \mathbb Q_n (\pi \in \mathcal O) \le  \sup_{\pi \in \mathcal O \cap \mathcal U_{\ell, k, k' ,  \varepsilon, \delta}}  \Big[ - J_H (\pi | \pi_0) \Big]  \, - \,   \delta'. 
\end{split}
\end{equation}
By taking the infimum over $\delta'$ on both sides, we get
\begin{equation}
\begin{split}
&\limsup_{n \to \infty} \frac1n \log \mathbb Q_n (\pi \in \mathcal O) \le  \sup_{\pi \in \mathcal O \cap \mathcal U_{\ell, k, k',  \varepsilon, \delta} } \Big[ - J_H (\pi | \pi_0) \Big]  = \sup_{\pi \in \mathcal O} \Big[ - J^{\ell, k , k',  \varepsilon, \delta}_H (\pi | \pi_0) \Big] 
\end{split}
\end{equation}
where the functional $J^{\ell, k, k',  \varepsilon, \delta}_H: C([0,T], \mathcal M) \to [0,+\infty]$ is defined by
\begin{equation}
\begin{split}
J^{\ell, k, k',  \varepsilon, \delta}_H (\pi | \pi_0)=
\begin{cases}
J_H (\pi | \pi_0) , \quad \text{if $\pi \in \mathcal U_{\ell, k, k' ,  \varepsilon, \delta}$},\\
+\infty \quad \text{otherwise},
\end{cases}
\end{split}
\end{equation}
for every $\pi \in C([0,T], \mathcal M)$. Hence we have{ \footnote{ To be precise, the infimum is taken over $H\in C_c^{1,2} (\Omega_T)$, $\ell,k, k' \in \mathbb N$, $\delta>0$ and $\varepsilon< \varepsilon_0(\delta)$.}} 
\begin{equation}
\begin{split}
&\limsup_{n \to \infty} \frac1n \log \mathbb Q_n (\pi \in \mathcal O) \le \inf_{H,\ell,k,k', \varepsilon, \delta} \,  \sup_{\pi \in \mathcal O}\  - J^{\ell, k , k',  \varepsilon, \delta}_H (\pi | \pi_0).
\end{split}
\end{equation}
Observe that $\mathcal U_{\ell,k,k', \varepsilon, \delta}$ is closed  into $C ([0,T], \mathcal M)$ so that, since $J_H (\cdot \vert \pi_0)$ is lower semi-continuous, the functional $J^{\ell, k , k',  \varepsilon, \delta}_H$ is also lower semi-continuous. By applying the min-max theorem  as in \cite[Lemma 3.3, p. 364]{KL}, we conclude that for every compact set $\mathcal K$
\begin{equation}
\begin{split}
&\limsup_{n \to \infty} \frac1n \log \mathbb Q_n (\pi \in \mathcal K) \le \sup_{\pi \in \mathcal K} \inf_{H,\ell, k, k',  \varepsilon, \delta}  - J^{\ell, k , k',  \varepsilon, \delta}_H (\pi | \pi_0).
\end{split}
\end{equation}
By Lemma \ref{lem:ucac} proved below, we have that
\begin{equation}
\begin{split}
\bigcup_{\ell} \bigcap_{\varepsilon, \delta, k , k'}  \mathcal U_{\ell, k , k',  \varepsilon, \delta} = {\mathcal C}^{ac},
\end{split}
\end{equation}
where $\mathcal C^{ac}$ has been introduced in Definition \ref{def:space_C_ab}. This implies that
\begin{equation}
\begin{split}
\inf_{H,\ell, k, k',  \varepsilon, \delta} - J^{\ell, k , k',  \varepsilon, \delta}_H (\pi | \pi_0)=\begin{cases}
- \sup_{H\in C_c^{1,2}(\Omega_T)}J_{H}(\pi|\pi_0),\quad\text{if}\quad \pi\in\mathcal C^{ac},\\
-\infty,\quad\text{otherwise.}
 \end{cases}
\end{split}
\end{equation}
We conclude then that
\begin{equation}
\begin{split}
\limsup_{n \to \infty} \frac1n \log \mathbb Q_n (\pi \in \mathcal K) \le - \inf_{\pi \in \mathcal K} I_{[0,T]} (\pi\vert \pi_0).
\end{split}
\end{equation}

\begin{lemma}
\label{lem:ucac}
We have that
\begin{equation}
\begin{split}
\bigcup_{\ell \ge 1} \ \bigcap_{\substack{\delta >0, \varepsilon<\varepsilon_0 (\delta),\\ k \ge 1 , k'\ge 1} } \  \mathcal U_{\ell, k , k',  \varepsilon, \delta} = {\mathcal C}^{ac},
\end{split}
\end{equation}
where $\mathcal C^{ac}$ has been introduced in Definition \ref{def:space_C_ab}.
\end{lemma}

\begin{proof}
The inclusion 
\begin{equation}
\begin{split}
{\mathcal C}^{ac} \subset \bigcup_{\ell \ge 1} \ \bigcap_{\substack{\delta >0, \varepsilon<\varepsilon_0 (\delta),\\ k \ge 1 , k'\ge 1} } \  \mathcal U_{\ell, k , k',  \varepsilon, \delta}
\end{split}
\end{equation}
is trivial. For the reverse one, let us take $\pi$ belonging to the set on the right-hand side of the last display. There exists $\ell \ge 1$ such that for any $k \ge 1$, 
\begin{equation}
\begin{split}
\sup_{G \in G^{(k)}} \mathcal E_G (\pi) \le \ell \quad \text{and} \quad \sup_{G \in G^{(k)}} {\mathcal F}_G (\pi) \le \ell.  
\end{split}
\end{equation}
By density of the sequence $(G^j)_{j \ge 1}$ in $L^2 ([0,T], {\mathcal H}_0^{\gamma/2})$ we deduce that the previous inequality hold with the supremum over $G^{(k)}$ replaced by the supremum over $L^2 ([0,T], {\mathcal H}_0^{\gamma/2})$. This implies that for almost every $t \in [0,T]$, $\pi_t (du) =\Phi_t (u) du$ with $\Phi \in L^2 ([0,T], {\mathcal H}^{\gamma/2})$ and, consequently, that the function $\Phi_t$ is continuous on $[0,1]$ since $\mathcal H^{\gamma/2}$ is continuously embedded in the space of continuous functions on $[0,1]$.  Moreover, for any $j\ge 1$, any $\delta>0$ and any $\varepsilon \in (0,\varepsilon_0 (\delta))$,
\begin{equation}
\begin{split}
\Big\vert  \int_0^T a^j (t) \Big( \Phi_\ell - \langle \pi_t , \iota_\varepsilon^0 \rangle \Big) dt \Big\vert \le \delta \quad \textrm{and}\quad \Big\vert  \int_0^T a^j (t) \Big( \Phi_r - \langle \pi_t , \iota_\varepsilon^1 \rangle \Big) dt \Big\vert \le \delta.
\end{split}
\end{equation}
We take first the limit $\varepsilon$ to zero. We observe then, by continuity of $\Phi_t$, that $\langle \pi_t, \iota_\varepsilon^0 \rangle$ converges to $\Phi_t (0)$ and $\langle \pi_t, \iota_\varepsilon^1 \rangle$ converges to $\Phi_t (1)$. Hence, we get 
\begin{equation}
\begin{split}
\Big\vert  \int_0^T a^j (t) \Big( \Phi_\ell - \Phi_t (0)  \Big) dt \Big\vert \le \delta \quad \textrm{and}\quad  \Big\vert  \int_0^T a^j (t) \Big( \Phi_r - \Phi_t (1)  \Big) dt \Big\vert \le \delta.
\end{split}
\end{equation}
This is true for any $j\ge 1$ and $(a^j)_{j \ge 1}$ is dense in the unit ball of $L^2([0,T])$. Hence the previous inequalities are in fact true by replacing $a^j$ by an arbitrary function $a$ in the unit ball of $L^2 ([0,T])$. Sending $\delta$ to zero , we deduce that for any such $a$,
\begin{equation}
\begin{split}
\int_0^T a_t \, \Big( \Phi_\ell - \Phi_t (0)  \Big) dt = 0 \quad \textrm{and}\quad \int_0^T a_t \, \Big( \Phi_r - \Phi_t (1)  \Big) dt = 0.
\end{split}
\end{equation}
By multiplying these equalities by an arbitrary scalar, we deduce that they hold for any $a \in L^2 ([0,T])$. This implies that for almost every $t \in [0,T]$, $\Phi_t (0) = \Phi_\ell$ and $\Phi_t (1) = \Phi_r$, and this conclude the proof that $\pi$ belongs to $\mathcal C^{ac}$. 
\end{proof}

\subsection{Lower Bound}
\label{sec:lower bound}

Fix an initial profile $g\in L^2 ([0,1])$ and an open set $\mathcal O$ in $C([0,T],\mathcal M)$. The goal of this section is to prove that 
\begin{equation}
        \liminf_{n\to\infty}\frac{1}{n}\log\mathbb Q_{n}(\pi \in\mathcal O)\geq -\inf_{\pi\in\mathcal O}I_{[0,T]}(\pi|g)
\end{equation}
where $(\mu_n)_{n \ge 2}$ is a sequence of initial probability measures associated to the profile $g$ in the sense of Definition \ref{eq:def-initialprofile}.

\subsubsection{Lower bound for smooth profiles}
From Theorem \ref{theo:existenceuniquenesshydrodynamicequations}, if $H\in C_c^{1,2}(\Omega_T)$, we know that there exists a unique weak solution of the hydrodynamic equation \eqref{eq:Dirichlet Equation2} which is denoted by $\pi^{H,g}$ and the corresponding density by $\Phi^{H,g}$. 
By the standard arguments in \cite[Chapter 10]{KL}, we have that
\begin{equation}
    \liminf_{n\to\infty}\frac 1n\log\mathbb Q_{n}\big(\pi \in\mathcal O\big)\; \geq\; - \; \inf_{H\in C_c^{1,2}(\Omega_T) \ \text{and} \ \pi^{H,g} \in\mathcal O} I_{[0,T]}(\pi^{H,g} \vert g).
\end{equation}
Hence, we are reduced to prove that
\begin{equation}
 \inf_{H\in C_c^{1,2}(\Omega_T) \ \text{and} \ \pi^{H,g} \in\mathcal O} I_{[0,T]}(\pi^{H,g} \vert g) \; \le \; \;  \inf_{\pi \in \mathcal O} I_{[0,T]} (\pi \vert g).
\end{equation}
This last property is a direct consequence of Theorem \ref{theo:lower_bound} stated below.

\subsubsection{$I_{[0,T]}$-density} 
\label{subsubsec:Idensity}
Let ${\mathcal C}^{ac,0}_g$ denote the set of all absolutely continuous paths $\{\pi_t(du) \; ; \; 0 \le t \le T\}=\{\Phi_t(u)du \; ; \; 0\le t \le T\}$ belonging to $\mathcal C^{ac}$, for which there exists $g \in L^{2} ([0,1])$ and $H\in C_c^{1,2}(\Omega_T)$ such that $\pi=\pi^{H,g}$ (and consequently $\Phi=\Phi^{H,g})$. In order to state the result more precisely, we need to introduce some terminology. 
\begin{definition}
    Let $g \in L^2([0,1])$. A subset $\mathcal{A}\subset C([0,T],\mathcal M)$ is said to be $I_{[0,T]}(\cdot|g)$-dense if for any $\pi\in C([0,T],\mathcal M)$ such that $I_{[0,T]}(\pi|g)<+\infty$, there exists a sequence $(\pi^k)_{k\geq 0}$ in $\mathcal{A}$ with $\lim_{k \to \infty} \pi^k =\pi$ in $C([0,T],\mathcal M)$ and $\lim_{k \to \infty} I_{[0,T]}(\pi^k|g) = I_{[0,T]}(\pi|g)$ in $\mathbb R$.
\end{definition}
﻿
We have the following theorem.
﻿
\begin{theorem}
\label{theo:lower_bound}
    For any $g \in L^2 ([0,1])$ the set ${\mathcal C}^{ac,0}_g$ is $I_{[0,T]} (\cdot \vert g)$-dense. Moreover, for any $ \pi \in C([0,T],\mathcal M)$ such that $I_{[0,T]}(\pi|g)<+\infty$, we can write $\pi_t (du) = \Phi_t (u)du$, and we can choose a subsequence $(\pi^k)_{k \ge 0}$ in ${\mathcal C}^{ac,0}_g$ such that $\pi^k (du)=\Phi^k (u)du$ with $(\Phi^k)_{k \ge 0}$ converging weakly in $L^2 ([0,T], {\mathcal H}^{\gamma/2})$ to $\Phi$ and $(\Phi_T^k)_{k \ge 0}$ converging weakly in $L^2 ([0,1])$ to $\Phi_T$.
\end{theorem}

\begin{proof}
   Let $\pi\in C([0,T],\mathcal M)$ such that $I_{[0,T]}(\pi|g)<+\infty$. We know then that $\{\pi_t(du) \; ; \; 0 \le t \le T\}=\{\Phi_t(u)du \; ; \; 0\le t \le T\}$ belongs to $\mathcal C^{ac}$ and by Lemma \ref{lem:LDinitialcondition}, we have $\Phi_0 = g$. By Lemma \ref{lem:Poissoneq}, there exists $H \in L^2 ([0,T], \mathcal H_0^{\gamma/2})$ such that $\pi$ is a weak solution of \eqref{eq:Dirichlet Equation2} and \eqref{eq:rate function representation} holds. Let us consider a sequence of functions $(H^k)_{k \ge 0}$ belonging to $C_c^{1,2} (\Omega_T)$ and converging to $H$ in $L^{2} ([0,1], \mathcal H_0^{\gamma/2})$. For any $k \ge 0$, denote $\pi^k (du) = \pi^{H^k, g} (du) = \Phi^k (u) du$ and let $\pi_{ss} (du)= {\Phi}_{ss} (u)du $ be the stationary solution of \eqref{eq:stationary}. Let us first show that there exists a constant $C>0$ depending on $T$ such that
 \begin{equation}
 \label{eq:boundMarielle}
   \sup_{k \ge 0} \,  \left\{ \sup_{t \in [0,T]} \Vert \Phi_t^k\Vert_2^2 \ + \   \int_{0}^T \Vert \Phi_t^k - \Phi_{ss} \Vert_{{\gamma/2}}^2  \, dt  \right\} \le C.
 \end{equation}

A simple computation of $ \partial_t \Vert \Phi_t^k - \Phi_{ss} \Vert_2^2$  and Young's inequality shows that for any $\varepsilon>0$, any $t \in [0,T]$ and any $k\ge 0$, we have
\begin{equation}
    \Vert \Phi^k_t - \Phi_{ss}\Vert _2^2 - \Vert g -\Phi_{ss} \Vert_2^2 \le -(2-\varepsilon)\int_{0}^t \Vert \Phi^k_r -\Phi_{ss} \Vert_{\gamma/2}^2 \, dr + \varepsilon^{-1} \int_0^t \Vert H^k_r\Vert_{\gamma/2}^2 \, dr .
\end{equation}
Since there exists a constant $C>0$ such that $\sup_{k \ge 0} \int_0^T \Vert H^k_t\Vert_{\gamma/2}^2 \, dt  \le C$, by choosing for example $\varepsilon =1$, we get \eqref{eq:boundMarielle}. The uniform bound \eqref{eq:boundMarielle} implies that we can extract a subsequence $(\Phi^{k^\prime})_{k^\prime \ge 0}$ of $(\Phi^{k})_{k \ge 0}$ converging weakly to a limiting point $\Psi$ in $L^2 ([0,T], {\mathcal H}^{\gamma/2})$ and such that $( \Phi_T^{k^\prime})_{k^\prime \ge 0}$ converges weakly in $L^2 ([0,1])$. In fact, since for any $k\ge 0$,  $\Phi^k - \Phi_{ss}\in L^2 ([0,T], {\mathcal H}_0^{\gamma/2})$, then  $\Psi- \Phi_{ss}\in L^2 ([0,T], {\mathcal H}_0^{\gamma/2})$. It follows that $\Psi$ is a weak solution of \eqref{eq:Dirichlet Equation2}. By uniqueness of weak solutions, see Proposition \ref{prop:uniqueness}, we conclude that the sequence $(\Phi^{k})_{k \ge 0}$ converges weakly to  the unique solution $\pi^{H,g}$ of \eqref{eq:Dirichlet Equation2} in $L^{2} ([0,T], {\mathcal H}^{\gamma/2})$.
We now claim that the sequence $(\pi^k)_{k \ge 0}$ is compact in $C ([0,T], \mathcal M)$. To show it, it is sufficient to prove that for any $G\in C_c^{\infty} ([0,1])$, $(\langle \pi^k, G\rangle)_{k \ge 0}$ is compact in $C([0,T], \mathbb R)$ and that $\Big( \sup_{t \in [0,T]} |\pi_t^k| ([0,1]) \Big)_{k \ge 0}$ is uniformly bounded. By the Cauchy-Schwarz's inequality,  \eqref{eq:boundMarielle} implies that for any $k\ge 0$ and $t \in [0,T]$
\begin{equation}
  |\pi_t^k| ([0,1]) {\leq} \int_0^1 \vert \Phi_t^k (u) \vert \, du \le \sqrt{\int_0^1 \vert \Phi_t^k (u) \vert^2  \, du} \le \sqrt{C}.    
\end{equation}
To prove that for any $G\in C_c^{\infty} ([0,1])$ the sequence $\big( \langle \pi^k, G\rangle  \big)_{k \ge 0}$ is compact in $C([0,T], \mathbb R)$, we rely on Ascoli's theorem. By Cauchy-Schwarz's inequality and \eqref{eq:boundMarielle}, we have for any $k \ge 0$, 
\begin{equation}
    \sup_{t \in [0,T]} \, \vert \langle \pi_t^k, G \rangle \vert = \sup_{t \in [0,T]} \, \vert \langle \Phi_t^k, G \rangle \vert \le \sqrt{C} \Vert G \Vert_2.
\end{equation}
It only remains to prove the equi-continuity of the sequence $( \langle \pi^k, G\rangle )_{k \ge 0}$. Thanks to Cauchy-Schwarz's inequality and \eqref{eq:boundMarielle}, for any $0 \le t \le t' \le T$ and $k \ge 0$, 
\begin{equation}
    \begin{split}
        \left\vert \langle \pi_{t'}^k , G \rangle - \langle \pi_t^k, G \rangle \right\vert &= \left\vert \int_{t}^{t'} \langle {\Phi}_s^k ,  - {\mathbb L}^\gamma G \rangle \, ds \right\vert  \le \int_t^{t'} \left\vert \langle {\Phi}_s^k, G \rangle_{\gamma/2} \right\vert  \, ds\\
        & \le \Vert G \Vert_{\gamma/2} \, \sqrt{t' -t} \, \sqrt{\int_0^t \Vert {\Phi}_s^k\Vert_{\gamma/2}^2 \, ds }\le \sqrt{C} \Vert G \Vert_{\gamma/2}^2 \, \sqrt{t' -t}.  
    \end{split}
\end{equation}

Hence, $\big( \langle \pi^k, G\rangle  \big)_{k \ge 0}$ is equi-continuous and consequently, the sequence $\{\pi^k\}_{k \ge 0}$ is compact in $C ([0,T], \mathcal M)$. Moreover it is easy to see that any limiting point $\pi \in C([0,T], \mathcal M)$ of this sequence is such that $\pi_t (du) = \Phi^{H,g} (u) \, du, 
$
so that the sequence $(\pi^k)_{k \ge 0}$ is in fact converging to $\{\pi_t^{H,g} (du) = \Phi_t ^{H,g} (u) \, du \; ; \; t \in [0,T] \}$. We now prove that $\lim_{k \to \infty} I_{[0,T]}(\pi^k|g) = I_{[0,T]}(\pi|g)$. Recall that 
\begin{equation}
I_{[0,T]} (\pi \vert g) = \frac{1}{4} \int_0^T \Vert H_t \Vert^2_{\gamma/2} \, dt
\end{equation}  
and that for any $k \ge 0$, 
\begin{equation}
I_{[0,T]} (\pi^k \vert g) = \frac{1}{4} \int_0^T \Vert H^k_t \Vert^2_{\gamma/2} \, dt.
\end{equation}  
Since we choose the sequence $(H^k)_{k \ge 0}$ belonging to $C_c^{1,2} (\Omega_T)$ converging to $H$ in $L^{2} ([0,1], \mathcal H_0^{\gamma/2})$, it holds trivially that $\lim_{k \to \infty} I_{[0,T]}(\pi^k|g) = I_{[0,T]}(\pi|g)$. This ends the proof of Theorem \ref{theo:lower_bound}.
\end{proof}

\section{Fractional Macroscopic Fluctuation Theory}
\label{sec:MFT}

The aim of this section is to compute explicitly the non-equilibrium free energy given by \eqref{eq:NEFE} and to show that it coincides with the quasi-potential defined by \eqref{eq:quasi potential}.

\subsection{Static Large Deviations Principle}
\label{sec:static ldp}

Recall that even when $\Phi_\ell\neq\Phi_r$, the NESS is still explicit and given by an inhomogeneous product Gaussian measure.

\begin{proof}[Proof of Theorem \ref{thm:static LDP}]
Recall the functional $\mathfrak F$ defined in \eqref{eq:macro MGF}. We apply Corollary 4.5.27 of \cite{dembo} and therefore we need to check the following facts:
\begin{enumerate}[1.]
    \item For all $G\in C([0,1])$, denoting 
    \begin{equation}
    \label{eq:lambdaNG}
        \mathfrak F_{n}(G)=\log E_{\mu^{n}_{ss}}\left[ e^{n\langle\pi,G\rangle} \right],
    \end{equation}
    the limit  $\lim_{n\to \infty} \frac{\mathfrak F_n (G)}{n}$ exists. Moreover, it is equal to $\mathfrak F (G)$ and it is finite.
    \item   The functional $\mathfrak F$ is  Gateaux differentiable.
\item $\mathfrak F: C([0,1]) \to {\mathbb R}$ is lower semi-continuous.
\item The sequence $ \left( \pi^n \right)_{n\ge 2}$ is exponentially tight. 
\end{enumerate}

Let us prove these four items. The first one is the content of Lemma \ref{lem:MGF}, which will be presented ahead. For the second one, we consider $G,H\in C ([0,1])$ and observe that
\begin{equation}
    \begin{split}
       \lim_{t \to 0} \frac{\mathfrak F(G + tH) -\mathfrak F (G)}{t} &= \lim_{t \to 0} \left\{  \langle H, \Phi_{ss}\rangle+\langle G, H \rangle +\frac{t}{2} \langle H, H \rangle \right\}\\
       &=  \langle H, \Phi_{ss}\rangle+\langle G, H \rangle . 
    \end{split}
\end{equation}
so that the second item is proved. The third item is trivial since $\mathfrak F$ is a continuous function. This follows easily by using the fact that $\Phi_{ss} (\cdot)$ is a bounded function. It remains to prove the forth item. We recall that ${\mathcal M}_A=\left\{ \mu \in {\mathcal M} \; ; \; \Vert \mu \Vert_{\rm{TV}} \le A \right\}$ is a compact subset of $\mathcal M$ for the weak topology. Hence to prove the exponential tightness of the sequence $\left( \pi^n\right)_{n\ge 2}$ it is sufficient to prove that
\begin{equation}
\label{eq:tightnessmuss}
    \limsup_{A \to \infty} \limsup_{n\to \infty} \log \ { P}_{\mu_{ss}^N} \left[ {\dfrac{1}{\vert \Lambda_n \vert}}\sum_{x \in \Lambda_N} \vert \varphi (x) \vert \ge A\right] =-\infty.
\end{equation}

By Markov's inequality we have that
\begin{equation}
 { P}_{\mu_{ss}^n} \left[ {\dfrac{1}{\vert \Lambda_n \vert}}\sum_{x \in \Lambda_N} \vert \varphi (x) \vert  \ge A\right] \le e^{-\vert \Lambda_n \vert  A} \, \prod_{x \in \Lambda_n} E_{\mu_{ss}^n} \Big( e^{\vert \varphi(x) \vert} \Big).    
\end{equation}
By the explicit form of $\mu_{ss}^n$ given in \eqref{NESS}, we have that
\begin{equation}
\begin{split}
    E_{\mu_{ss}^n} \Big( e^{\vert \varphi(x) \vert} \Big) &= \dfrac{1}{\sqrt{2\pi}} \int_{-\infty}^{\infty} e^{|u|} e^{-\tfrac{(u-\Phi_{ss}^n (x))^2}{2}}\, du =\dfrac{1}{\sqrt{2\pi}} \int_{-\infty}^{\infty} e^{|u +\Phi_{ss}^n (x)|} e^{-\tfrac{u^2}{2}}\, du\\
    &\le e^{\vert \Phi_{ss}^n (x) \vert} \, \dfrac{1}{\sqrt{2\pi}} \int_{-\infty}^\infty e^{|u| - \tfrac{u^2}{2}} \, du. 
\end{split}
\end{equation}
Since, by \eqref{eq:boundprofileness}, the microscopic profile is such that $|\Phi_{ss}^n (x)\vert \lesssim 1$ uniformly in $n$ and $x\in \Lambda_n$, we deduce \eqref{eq:tightnessmuss}, and the forth item is proved.
\end{proof}

It remains to prove Lemma \ref{lem:MGF}. Recall that $\mathfrak F_n$ is defined by \eqref{eq:lambdaNG}.

\begin{lemma}
\label{lem:MGF}
For any continuous function $G:[0,1]\to \RR$ we have that
$$\lim_{n\to \infty}\left\vert \frac{{\mathfrak F}_{n}(G)}{n} -\mathfrak F (G)\right\vert = 0.$$
\end{lemma}

\begin{proof}
Since  $\mu_{ss}^n$ is product and for any $x\in\Lambda_n$, $\varphi (x) \sim\mathcal N(\Phi_{ss}^n (x), 1)$ we have that 
\begin{equation}
\begin{split}
\frac{\mathfrak F_{n}(G)}{n} =\dfrac{1}{n}\sum_{x\in \Lambda_{n}} \left[ G\big(\tfrac x n\big ) \Phi_{ss}^n (x) +\dfrac12 G^2\big(\tfrac x n\big )\right].
\end{split}
\end{equation}

Recall now Theorem \ref{thm:hydro static limit}. For any $\delta>0$, consider the event \begin{equation}
    \mathcal A=\Big\{\Big\vert \tfrac{1}{\vert \Lambda_n \vert} \sum_{x\in \Lambda_n} G\big(\tfrac x n\big )\varphi (x) - \langle G, \Phi_{ss} \rangle  \Big\vert >\delta\Big\}.
\end{equation}
By Cauchy-Schwarz inequality, 
\begin{equation}
\begin{split}
&\Big\vert \dfrac{1}{\vert \Lambda_n \vert}\sum_{x\in \Lambda_{n}} G\big(\tfrac x n\big ) \Phi_{ss}^n (x) - \langle G, \Phi_{ss} \rangle \Big\vert=\Big\vert E_{\mu_{ss}^n} \Big[\dfrac{1}{\vert\Lambda_n \vert}\sum_{x\in \Lambda_{n}} G\big(\tfrac x n\big ) \varphi (x) - \langle G, \Phi_{ss} \rangle \Big]\Big\vert\\
&\leq E_{\mu_{ss}^n} \Big[\Big\vert {\vert\Lambda_n \vert}^{-1}\sum_{x\in \Lambda_{n}} G\big(\tfrac x n\big ) \varphi (x) - \langle G, \Phi_{ss} \rangle\Big\vert \ {\mathbf 1}_{\mathcal A} \Big] + \delta\\
& \le  \Big(E_{\mu_{ss}^n} \Big[\Big\vert {\vert\Lambda_n \vert}^{-1}\sum_{x\in \Lambda_{n}} G\big(\tfrac x n\big ) \varphi (x) - \langle G, \Phi_{ss} \rangle\Big\vert^2  \Big] \Big)^{1/2}\mu_{ss}^n (\mathcal A) ^{1/2} + \delta .
\end{split}
\end{equation}

Observe now that by \eqref{eq:boundprofileness},
\begin{equation}
    \begin{split}
       &E_{\mu_{ss}^n} \Big[\Big\vert {\vert\Lambda_n \vert}^{-1}\sum_{x\in \Lambda_{n}} G\big(\tfrac x n\big ) \varphi (x) \Big\vert^2 \Big] = {\vert\Lambda_n \vert}^{-2}\sum_{x,y\in \Lambda_{n}} G\big(\tfrac x n\big ) G\big(\tfrac y n\big ) E_{\mu_{ss}^n} \Big[ \varphi (x) \varphi(y) \Big]\\
       &= {\vert\Lambda_n \vert}^{-2}\sum_{x,y\in \Lambda_{n}} G\big(\tfrac x n\big ) G\big(\tfrac y n\big ) \Big[ \delta_{x=y} + \Phi_{ss}^n (x) \Phi_{ss}^n(y) \Big]\le C
    \end{split}
\end{equation}
where $C$ is a positive constant independent of $n$. Also, by Theorem \ref{thm:hydro static limit}, $\mu^n_{ss}(\mathcal A)\to0$, as $n\to\infty$. We conclude that
\begin{equation}
    \limsup_{n \to \infty} \left\vert \dfrac{1}{\vert \Lambda_n \vert}\sum_{x\in \Lambda_{n}} G\big(\tfrac x n\big ) \Phi_{ss}^n (x) - \langle G, \Phi_{ss} \rangle \right\vert \le \delta,
\end{equation}
and, since this is valid for any $\delta>0$, we obtain that
\begin{equation}
\lim_{n \to \infty} \dfrac{1}{\vert \Lambda_n \vert}\sum_{x\in \Lambda_{n}} G\big(\tfrac x n\big ) \Phi_{ss}^n (x) = \langle G, \Phi_{ss} \rangle
\end{equation}
so that the lemma follows.
\end{proof}

\subsection{The Quasi-Potential}
\label{subsec:QPHJ}

The goal of is section is to prove Theorem \ref{thm:quasipot-freeenergy}. 

\subsubsection{Preliminaries}
\label{subsubsec:preliminaries} Recall \eqref{eq:NEFE} and \eqref{eq:quasi potential}.
The proof follows by first showing the lower bound $V(\varrho)\geq W(\varrho)$ and then the  upper bound $V(\varrho)\leq W(\varrho)$. We start with four auxiliary results, namely Lemmata \ref{lem:rate function L^2}, \ref{lem:CT}, \ref{lem:auxiliary mft 1} and \ref{lem:auxiliary mft 2}, that will be useful for deriving  both inequalities.

Let $\mathcal H^{-\gamma/2}$ be the topological dual of $\mathcal H_0^{\gamma/2}$, i.e., the vector space of continuous linear forms $\mathcal T:F \in \mathcal H_0^{\gamma/2} \to \langle \mathcal T , F\rangle \in \mathbb R$. Let $\mathcal D$ be the Hilbert space composed of functions $\Phi \in L^{2} ([0,T]  ,   \mathcal H^{\gamma/2})$ such that $\partial_t \Phi \in L^2 ([0,T] ,  \mathcal H^{-\gamma/2})$. The space $\mathcal D$ is equipped with the norm $\Vert \cdot \Vert_{\mathcal D}$ defined by 
\begin{equation}
    \forall\,  \Phi \in \mathcal D, \quad \Vert \Phi\Vert_{\mathcal D} = \left[ \int_{0}^T \Vert \Phi_t\Vert_{\gamma/2}^2 \; dt \ + \ \int_{0}^T \Vert \partial_t \Phi_t\Vert_{-\gamma/2}^2 \; dt \right]^{1/2}.
\end{equation}

\begin{lemma}
\label{lem:rate function L^2}
    Let $g \in L^2 ([0,1])$ and $I_{[0,T]}(\cdot\vert g)$ be the functional defined  in \eqref{eq:rate function}. Then, for any $\pi(du) = \Phi(u)du \in C([0,T], \mathcal M)$ such that $\Phi \in \mathcal D$, the non-linear functional $\ H \in C_c^{1,2} (\Omega_T) \to J_H (\pi \vert g)$ (see \eqref{eq:JHPG}) can be continuously extended to $L^2 ([0,T],\mathcal H_0^{\gamma/2})$ and consequently,
    \begin{equation}
    \label{eq:ITdensitytchut}
        I_{[0,T]}(\pi \vert g)=\sup_{H\in L^2 ([0,T] ,\mathcal H_0^{\gamma/2})}J_H(\pi \vert g).
    \end{equation}
\end{lemma}

\begin{proof}
    Observe first that $C_c^{1,2} (\Omega_T)$ is dense in the (hence complete) Hilbert space  $L^{2} ([0,T] , \mathcal H_0^{\gamma/2})$. For any $H \in C_c^{1,2} (\Omega_T)$, an integration by parts (in the distributional sense) gives
\begin{equation}
\begin{split}
    J_{H}(\pi|g) & =\langle\pi_T,H_T\rangle -\langle g,H_0\rangle-\int_0^T\langle\pi_s,(\partial_s+\mathbb L^\gamma)  H_s\rangle\;ds + \int_0^T\|H_s\|_{\gamma/2}^2\;ds\\
    &=\int_0^T\langle \partial_s \pi_s , H_s\rangle\;ds + \int_0^T \langle \pi_s, H_s \rangle_{\gamma/2} \; ds + \int_0^T\|H_s\|_{\gamma/2}^2\;ds. 
\end{split}
\end{equation}
Then, by Cauchy-Schwarz's inequality, for any $G,H \in C_c^{1,2} (\Omega_T)$,    
\begin{equation}
        | J_G (\pi \vert g) - J_H (\pi \vert g) \vert \lesssim \left\{\Vert \Phi \Vert_{\mathcal D}   \, \Bigg[\sqrt{ \int_0^T \Vert H_t\Vert_{\gamma/2}^2 dt }+\sqrt{ \int_0^T \Vert G_t\Vert_{\gamma/2}^2 dt\Big) }\Bigg]\right\}\sqrt{ \int_0^T \Vert H_t-G_t\Vert_{\gamma/2}^2 dt}.
\end{equation}   
This local Lipschitz property is sufficient to extend continuously the functional $H \in C_c^{1,2} (\Omega_T) \to J_H (\pi \vert g)$ to the complete metric space $L^2 ([0,T]  ,  \mathcal H_0^{\gamma/2})$.
\end{proof}

\begin{lemma}
\label{lem:CT}
    Let $H \in L^2 ([0,T] , \mathcal H_0^{\gamma/2})$ and $\Phi \in L^{2} ([0,T]  ,   \mathcal H^{\gamma/2})$ a weak solution of \eqref{eq:Dirichlet Equation2} starting from $g \in L^2 ([0,1])$. Then $\Phi$ belongs to the space $\mathcal D$.
\end{lemma}

\begin{proof}
    Let $\mathcal D (\Omega_T)$ be the space of smooth functions in time and space having a compact support included in $(0,T)\times(0,1)$. Since  $\Phi$ is a weak solution of \eqref{eq:Dirichlet Equation2}, for any $G \in \mathcal D (\Omega_T)$, we have (in the distributional sense), 
    \begin{equation}
    \begin{split}
        -\int_{0}^T \langle \partial_s \Phi_s, G_s \rangle ds & = \int_0^T\left\langle \Phi_{s},\partial_s  G_{s}  \right\rangle \, ds\\
        &= \int_0^T \left\langle \Phi_s, G_s \right\rangle_{\gamma/2} \, ds - \int_0^T \left\langle H_s, G_s \right\rangle_{\gamma/2} \, ds.
    \end{split}
    \end{equation}
{By the Cauchy-Schwarz's inequality, the absolute value of the last term in the previous display is bounded by }
    \begin{equation}
        \left[ \sqrt{\int_{0}^T \Vert \Phi_s\Vert_{\gamma/2}^2 \, ds } + \sqrt{\int_{0}^T \Vert H_s\Vert_{\gamma/2}^2 \, ds }\right] \, \sqrt{\int_{0}^T \Vert G_s\Vert_{\gamma/2}^2 \, ds}, 
    \end{equation}
from where we deduce that the distribution $\partial_s \Phi_s$ belongs to $L^{2} ( [0,T] , {\mathcal H}^{-\gamma/2})$.
\end{proof}

For any measurable function $\Phi:[0,1] \to \mathbb R$, let us define $\Gamma_\Phi =\Phi -\Phi_{ss}$.

\begin{lemma}
\label{lem:auxiliary mft 1}
Let $\Phi\in \mathcal D$ and $\pi \in C ([0,T], \mathcal M)$ such that $\pi_t (du) = \Phi_t (u) du$ for any $t \in [0,T]$. Then
    \begin{equation}
        W(\pi_T)-W(\pi_0)=\int_0^T \big\langle \partial_t\Phi_t,\Gamma_{\Phi_t}\big\rangle \, dt.
    \end{equation}
\end{lemma}

\begin{proof}
    Note that if $\pi$ is absolutely continuous and has a {density $\Phi\in\mathcal D$ which is differentiable}, then 
    \begin{equation}
    \begin{split}
        W(\pi_T) -W(\pi_0)= \int_0^T\frac{d}{dt}W(\pi_t)dt&=\int_0^T\int_0^1(\Phi_t(u)-\Phi_{ss}(u))\partial_t\Phi_t(u)dudt=\int_0^T\big\langle \partial_t\Phi_t,\Gamma_{\Phi_t}\big\rangle \, dt.
    \end{split}
    \end{equation}
    By a suitable approximation {we can remove the differentiability assumption} and  extended it to the case $\Phi \in \mathcal D$.
\end{proof}

\begin{lemma}
\label{lem:auxiliary mft 2} 
Suppose $\rho\in\mathcal H^{\gamma/2}$ with ${\rho} (0)=\Phi_\ell$ and ${\rho} (1)=\Phi_r$. Then, we have
    \begin{equation}
       \Vert \Gamma_{\rho} \Vert^2_{\gamma/2}-\big\langle \Gamma_{\rho} ,{\rho}\big\rangle_{\gamma/2}=0.
    \end{equation}
\end{lemma}

\begin{proof}
    By definition of $\Gamma_{\rho}$, we have that 
     \begin{equation}
    \begin{split}
        \Vert \Gamma_{\rho} \Vert^2_{\gamma/2}-\big\langle \Gamma_{\rho} ,{\rho}\big\rangle_{\gamma/2}=\langle {\rho}-\Phi_{ss}\, , \,  {\rho}-\Phi_{ss}\rangle_{\gamma/2} - \langle {\rho}-\Phi_{ss} \,  , \,  {\rho}\rangle_{\gamma/2}
    \end{split}
    \end{equation}
    Since $\Phi_{ss}$ is the solution  to \eqref{eq:stationary}, for any $F\in C^\infty_c([0,1])$ we have $\langle \Phi_{ss}, F\rangle_{\gamma/2}=0$ and since $\Phi_{ss} \in \mathcal H^{\gamma/2}$, by density, this remains true for $F\in \mathcal H_0^{\gamma/2}$. Observe that ${\rho} - \Phi_{ss}$ belongs to $\mathcal H^{\gamma/2}_0$, {so that $\langle \Phi_{ss}\, , \,  {\rho}-\Phi_{ss}\rangle_{\gamma/2} =0$. This ends the proof}. \end{proof}

\subsubsection{Lower Bound} In order to prove the lower bound, we shall choose the correct test function $H$ in the definition of the rate function $I_{[0,T]}$. Modulo technical regularization details, this will be given by $\delta W/\delta\rho$.

\begin{lemma}
    For each $\varrho\in\mathcal M$, $V(\varrho )\geq W(\varrho)$.
\end{lemma}
\begin{proof}

 We may assume that $\varrho(du) = \rho(u) du$, for some $\rho \in L^2 ([0,1])$, since otherwise $W(\varrho)=V (\varrho)=\infty$. By the variational definition of $V$, we need to prove that $W(\varrho)\leq I_{[0,T]}(\pi|\Phi_{ss})$, for any $T>0$ and for any path $\pi$ in $C([0,T],\mathcal M)$ such that $\pi_0=\pi_{ss}$ and $\pi_T=\varrho$. If the path $\pi\notin \mathcal C^{ac}_{\Phi_{ss}}$ then $I_{[0,T]} (\pi \vert \Phi_{ss})$ is infinite and there is nothing to prove. Therefore we can assume that for any $t \in [0,T]$, $\pi_t (du)= \Phi_t (u) du$ with $\Phi \in L^2 ([0,T],  \mathcal H^{\gamma/2})$, $\Phi_0 =\Phi_{ss}$, $\Phi_T = \rho$, and $\Phi_t (0) =\Phi_\ell$, $\Phi_t (1) = \Phi_r$ for a.e. $t\in [0,T]$.

Recall the definition of $\mathcal C^{ac,0}_{\Phi_{ss}}$ given at the beginning of Section \ref{subsubsec:Idensity} and also the density property stated in Theorem \ref{theo:lower_bound}. We can therefore exhibit a sequence $(\pi^k)_{k\geq 0}$ in $\mathcal C^{ac,0}_{\Phi_{ss}}$ such that $\lim_{k \to \infty} \pi^k =\pi$ in $C([0,T],\mathcal M)$ and $\lim_{k \to \infty} I_{[0,T]}(\pi^k|\Phi_{ss}) = I_{[0,T]}(\pi|{\Phi}_{ss})$ in $\mathbb R$. Moreover, the sequence $(\pi^k)_{k \ge0}$ can be chosen such that $\pi_t^k (du) = \Phi^k_t (u)du$ with $(\Phi_T^k)_{k \ge 0}$ converging to $\Phi_T$ weakly in $L^2 ([0,1])$. Since $ \pi^k \in C^{ac,0}_{\Phi_{ss}}$, it follows from Lemma \ref{lem:CT}  that $\Phi^k \in \mathcal D$. By Lemma \ref{lem:rate function L^2}, we thus get that 
\begin{equation}
I_{[0,T]} (\pi^k \vert \Phi_{ss} ) \ge J_{H} (\pi^k  \vert g )
\end{equation}
for any $H \in L^2 ( [0,T], \mathcal H^{\gamma/2}_0)$, where the right-hand side of the last display can be rewritten, since $\Phi^k \in \mathcal D$, as 
\begin{equation}
\begin{split}
    J_{H}(\pi^k|g) &=\int_0^T\langle \partial_s \Phi^k_s , H_s\rangle\;ds + \int_0^T \langle \Phi^k_s, H_s \rangle_{\gamma/2} \; ds - \int_0^T\|H_s\|_{\gamma/2}^2\;ds. 
\end{split}
\end{equation}
We then choose $H:= H^k=\Gamma_{\Phi^k} = \Phi^k - \Phi_{ss} \in L^2 ( [0,T],  \mathcal H_0^{\gamma/2})$. By Lemma \ref{lem:rate function L^2} and Lemma \ref{lem:auxiliary mft 1}, we have 
    \begin{equation}
        J_{\Gamma_\Phi^k} (\pi^k \vert  \Phi_{ss}) =W(\pi^k_T)-W(\pi_{ss})-\int_0^T\Big\{\Vert\Gamma_{\Phi^k_s} \Vert^2_{\gamma/2}-\langle \Phi^k_s, \Gamma_{\Phi^k_s} \rangle_{\gamma/2}\Big\}\; ds.
    \end{equation}
Hence, by Lemma \ref{lem:auxiliary mft 2} and the definition of $W$,
\begin{equation}
 I_{[0,T]} (\pi^k \vert \Phi_{ss} ) \ge  W(\pi_T^k)=\frac{1}{2} \Vert \Phi_T^k -\Phi_{ss} \Vert^2_{2}.
\end{equation}
Since $(\Phi_T^k)_{k \ge 0}$ is converging to $\Phi_T$ weakly in $L^2 ([0,1])$ and $\lim_{k \to \infty} I_{[0,T]}(\pi^k|\Phi_{ss}) = I_{[0,T]}(\pi|{\Phi}_{ss})$, we get
\begin{equation}
    I_{[0,T]} (\pi \vert \Phi_{ss} ) \ge  W(\pi_T) =W(\varrho),
\end{equation}
which implies $V(\varrho) \ge W(\varrho).$    
\end{proof}

\subsubsection{Upper bound} To prove the upper bound we shall exhibit the optimal path that realizes the variational problem \eqref{eq:quasi potential}. According to MFT, this optimal path is given by the time reversed path of the profile followed by the adjoint hydrodynamic equation. In our case, $\mathcal L_n^*\langle\pi^n_t,G_t\rangle=\mathcal L_n\langle \pi^n_t,G_t\rangle$ (see Appendix \ref{app_adjoint}). Consequently, the adjoint hydrodynamic equation coincides with the original hydrodynamic equation even though the system is microscopically non-reversible. We describe the strategy before presenting the rigorous proof.

{
Fix a measure $\varrho \in \mathcal{M}$. To derive an upper bound for the quasi-potential at $\varrho$, we construct a trajectory connecting the stationary profile $\Phi_{ss}$ to $\varrho$ in two stages. First, Lemma \ref{lem:conv_stationary_profile} establishes that the hydrodynamic solution $\Phi_t$ converges to $\Phi_{ss}$ as $t \to \infty$; thus, for sufficiently large $T_1$, $\Phi_{T_1}$ lies in a small neighborhood of $\Phi_{ss}$. Second, we consider the time-reversed trajectory $\Phi^*_t = \Phi_{T-t}$ starting from $\varrho$. Using Lemmas \ref{lem:auxiliary mft 1} and \ref{lem:auxiliary mft 2}, we estimate the cost of the corresponding measure-valued trajectory $\pi^*$ in terms of the rate function $W(\varrho)$. To bridge the gap between $\Phi_{ss}$ and the profile at $T_1$, Lemma \ref{lem:clever path} provides a path $\widetilde{\pi}$ whose cost is bounded by $\|\Psi - \Phi_{ss}\|_2^2$, for some fixed $\Psi$. By setting $\Psi = \Phi_{T_1}$, we obtain a continuous concatenation $\widehat{\pi}$ connecting $\Phi_{ss}$ to $\varrho$. The cost of the initial segment $\widetilde{\pi}$ is vanishingly small, while the cost of $\pi^*$ is bounded by $W(\varrho)$.} We shall now make this rigorous.

\begin{lemma}\label{lem:conv_stationary_profile}
There exists a universal constant $\alpha>0$ such that for any weak solution $(\Phi_t)_{t \ge 0}$ of \eqref{eq:Dirichlet Equation2} with $H=0$ and initial condition $g \in L^2([0,1])$, 
 \begin{equation}
\begin{split}
\forall t \ge 0, \quad \Vert \Phi_t - \Phi_{ss} \Vert_2 \le \Vert g {-\Phi_{ss}}\Vert_2 \, e^{-\alpha t}.
\end{split}
\end{equation}

\end{lemma}

\begin{proof}
Consider $\bar \Phi = \Phi - \Phi_{ss}$ and observe that $\bar\Phi \in L^2 ([0,T], \mathcal H_0^{\gamma/2})$. Moreover,  for any $G\in C_c^\infty(\Omega_T)$ and $t \in [0,T]$, by the definition of the weak solutions
\begin{equation}
\begin{split}
\langle \bar \Phi_t, G_t \rangle - \langle g - \Phi_{ss}, G_0 \rangle = \int_{0}^t \langle (\partial_s + \mathbb L^\gamma) G_s, {\bar \Phi}_s  \rangle ds.
\end{split}
\end{equation}
By taking a sequence of functions $(G_k)_{k \ge 0}$ in $C_c^{\infty} (\Omega_T)$ converging to $\bar\Phi$ in $L^{2} ([0,T], \mathcal H_0^{\gamma/2})$, we see that, in the previous equation, we can replace $G$ by $\bar\Phi$. Then we get
\begin{equation}
\begin{split}
\Vert \bar \Phi_t \Vert_{2}^2 - \Vert \bar\Phi_0 \Vert_{2}^2 = \frac{1}{2} \int_0^t \partial_s (\Vert \bar\Phi_s\Vert_{2}^2) ds - \int_0^t \Vert \bar \Phi_s \Vert_{\gamma/2}^2\,  ds ,
\end{split}
\end{equation}
which gives 
\begin{equation}
\begin{split}
\Vert \bar \Phi_t \Vert_{2}^2= \Vert \bar\Phi_0 \Vert_{2}^2 - 2 \int_0^t \Vert \bar \Phi_s \Vert_{\gamma/2}^2 \, ds.
\end{split}
\end{equation}
This equality is true for any $T>0$ and $t \in [0,T]$ and thus for any $t\ge 0$. By the fractional Poincar\'e's inequality, there exists a universal constant $C>0$ such that for any $F \in \mathcal H_0^{\gamma/2}$, $\Vert F \Vert_{2}^2 \le C \Vert F \Vert^2_{\gamma/2}$. Hence we deduce that 
\begin{equation}
\begin{split}
\Vert \bar \Phi_t \Vert_{2}^2\le  \Vert \bar\Phi_0 \Vert_{2}^2 - \frac{2}{C} \int_0^t \Vert \bar \Phi_s \Vert_{2}^2 ds,
\end{split}
\end{equation}
and then by Gronwall's inequality, $\Vert \bar \Phi_t \Vert_{2}^2\le  \Vert \bar\Phi_0 \Vert_{2}^2 \; e^{-2t/C}.$
\end{proof}

\begin{lemma}
\label{lem:clever path}
    Let $\Psi\in L^2([0,1])$. There exists a  path $\widetilde{\pi} (du) =\widetilde \Phi(u)du \in C([0,1],\mathcal M)$ with  $\widetilde \Phi \in L^2 ([0,1], \mathcal H^{\gamma/2})$ such that  $\widetilde\pi_0 (du)=\Phi_{ss}(u)du$, $\widetilde\pi_1 (du)=\Psi(u)du$, $\widetilde\Phi_t(0)=\Phi_\ell$, $\widetilde \Phi_t (1)=\Phi_r$ for any $t\in [0,1]$, and a constant $C>0$ satisfying 
    \begin{equation}
        I_{[0,1]}(\widetilde\pi|\Phi_{ss})\leq C \, \Vert\Psi-\Phi_{ss}\Vert_2^2.
    \end{equation}
\end{lemma}

\begin{proof}
By \cite[Prop. 3.6]{TW22} we know that there exists a sequence of eigenvalues 
\begin{equation}
    0 < \lambda_1 \le \lambda_2 \le \ldots \le \lambda_k \le \ldots \to \infty,
\end{equation} 
counted with multiplicities, and a sequence of eigenfunctions $(e_k)_{k \ge 1}$, i.e., $- \mathbb L^\gamma e_k = \lambda_k e_k$ in the weak sense, which forms an orthonormal basis of $L^2([0,1])$.  

We recall that $\mathcal H^{-\gamma/2}$ is the topological dual space associated to $\mathcal H_0^{\gamma/2}$. It is endowed with the dual norm $\Vert \cdot \Vert_{\mathcal H^{-\gamma/2}}$ defined by
\begin{equation}
\label{eq:defH-gamma2}
   \forall \mathcal T \in \mathcal H^{-\gamma/2}, \quad  \Vert \mathcal T \Vert^2_{-\gamma/2} = \sup_{F \in \mathcal H_0^{\gamma/2}} \cfrac{\langle \mathcal T , F\rangle^2}{ \Vert F \Vert_{\gamma/2}^2} = \sum_{k=1}^{\infty} \frac{1}{\lambda_k} \langle \mathcal T, e_k \rangle^2.
\end{equation}

The same computation as in the proof of \cite[Lemma 5.7]{BGL03} shows that $\mathcal T$ defined by
\begin{equation}
  \forall t \in [0,1], \quad \mathcal T_t =  - \sum_{k=1}^\infty \lambda_k \,  \cfrac{2e^{\lambda_k t} -1}{e^{\lambda_k} -1} \, \langle \Psi -\Phi_{ss}, e_k \rangle e_k
\end{equation}
belongs to $ L^2 ([0,1], {\mathcal H}^{-\gamma/2})$. In fact, there exists a finite constant $C>0$ such that 
\begin{equation}
\label{eq:oulalala}
\Vert \mathcal T_t\Vert_{-\gamma/2}^2 \le C \sum_{k=1}^\infty \lambda_k e^{2 \lambda_k (t-1)} \langle \Psi - \Phi_{ss}, e_k \rangle^2.
\end{equation}
Consequently, there exists $H \in L^2 ([0,T], {\mathcal H}_0^{\gamma/2})$ such that for a.e. $t\in [0,1]$,
\begin{equation}
\label{eq:PoissonTt}
    -\mathbb L^\gamma H_t = \mathcal T_t. 
\end{equation}
We want now to use Theorem \ref{theo:existenceuniquenesshydrodynamicequations}, and take $\widetilde \pi$ the weak solution of $\partial_t \widetilde\pi = {\mathbb L}^\gamma \widetilde \pi - \mathbb L^\gamma H ={\mathbb L}^\gamma \widetilde\pi + \mathcal T$ with initial condition $\Phi_{ss}$ and boundary conditions $\widetilde \pi_t (0)= \Phi_\ell$, $\widetilde \pi_t (1)= \Phi_r$. Unfortunately, Theorem \ref{theo:existenceuniquenesshydrodynamicequations} has only been proved for $H \in C_{c}^{1,2} ([0,1]^2)$, which is apriori not true. Therefore, we approximate $H$ as in the proof of Theorem \ref{theo:lower_bound} by a sequence $(H^k)_{k\ge 0}$ in $C_c^{1,2} ([0,1]^2)$ converging to $H$ in $L^{2} ([0,1], \mathcal H_0^{\gamma/2})$. For any $k \ge 0$, denote $\widetilde \pi^k (du) = \pi^{H^k, \Phi_{ss}} (du) = \widetilde\Phi^k (u) du$, with $\widetilde \Phi^k$ provided by Theorem \ref{theo:existenceuniquenesshydrodynamicequations} when $H=H^k$, and let $\pi_{ss} (du)= {\Phi}_{ss} (u)du $ be the solution of \eqref{eq:stationary}. By \eqref{eq:PoissonTt} and  \eqref{eq:defH-gamma2}, 
\begin{equation}\label{eq:bound_clever path}
\int_0^1 \Vert H_t \Vert_{\gamma/2}^2 \, dt = \int_0^1 \Vert \mathcal T_t \Vert_{-\gamma/2}^2 \, dt{\,\leq C\Vert\Psi - \Phi_{ss} \Vert_{2}^2,}
\end{equation}
{where last inequality follows from 
\eqref{eq:oulalala} and the fact that $\int_0^1 \lambda_k e^{2\lambda_k (t-1)} dt = (1- e^{-\lambda_k})/2 \le 1/2$}. Since $(H^k)_{k\ge 0}$ converges to $H$ in $L^{2} ([0,1], \mathcal H^{\gamma/2})$, we deduce there exists $C>0$ such that
\begin{equation}
    \sup_{k \ge 0} \int_0^1 \Vert H^k\Vert_{\gamma/2}^2 \, dt \le C.
\end{equation}
By following the proof of Theorem \ref{theo:lower_bound} we conclude that $(\widetilde \pi^k)_{k \ge 0}$ converges to a path $\widetilde \pi (du) = \widetilde\Phi(u) du$ in $C([0,1], \mathcal M)$ with $\widetilde\Phi \in L^2 ([0,1], \mathcal H^{\gamma/2})$ being a  weak solution to $\partial_t \widetilde\pi = {\mathbb L}^\gamma \widetilde \pi - \mathbb L^\gamma H ={\mathbb L}^\gamma \widetilde\pi + \mathcal T$ with initial condition $\Phi_{ss}$ and boundary conditions $\tilde \Phi_t (0)= \Phi_\ell$, $\widetilde \Phi_t (1)= \Phi_r$. Moreover,  $\lim_{k \to \infty} I_{[0,1]} (\widetilde \pi^k \vert \Phi_{ss}) =I_{[0,1]} (\widetilde \pi^k \vert \Phi_{ss})$. But observe now that by Lemma \ref{lem:Poissoneq} and \eqref{eq:bound_clever path}, 
\begin{equation}
\begin{split}
\lim_{k \to \infty} I_{[0,1]} (\widetilde \pi^k \vert \Phi_{ss}) = \lim_{k \to \infty}  \frac{1}{4} \int_0^1 \Vert H^k\Vert_{\gamma/2}^2\,  dt = \frac{1}{4} \int_0^1 \Vert H_t\Vert_{\gamma/2}^2\,  dt \leq C||\Psi-\Phi_{ss}||^2_2. 
\end{split}  
\end{equation}
\end{proof}

\begin{remark}
\label{rem:extension-existence}
    The proof of the previous lemma  shows in fact that Theorem \ref{theo:existenceuniquenesshydrodynamicequations} holds for $H \in L^2 ([0,T], \mathcal H_0^{\gamma/2})$. 
\end{remark}

\begin{lemma}
    For each $\varrho\in\mathcal M$, $V(\varrho )\leq W(\varrho)$.
\end{lemma}

\begin{proof}
    Fix $\varrho\in\mathcal M$. We may assume that $\varrho (du) = \rho(u)du$ with $\rho \in L^2 ([0,1])$ because otherwise, $W(\varrho)=\infty$ and there is nothing to prove. Let $\Phi$ be the solution to equation \eqref{eq:Dirichlet Equation2} with $H=0$ starting from $\rho$ and $\Phi^*_t:=\Phi_{T-t}$ be its time reversed. In the interval $[0,T]$ it solves (in a weak sense)  the equation:
\begin{equation}
\label{eq:time reversed eq}
\begin{dcases}
    \partial_t\Phi^*=-\mathbb L^\gamma\Phi^*,\\
    \Phi_t^*(0)=\Phi_\ell,\quad \Phi_t^*(1)=\Phi_r,\\
    \Phi^*_0=\Phi_T\quad\text{and}\quad\Phi^*_{T}=\rho.
\end{dcases}
\end{equation}
We shall denote, for any $t\in [0,T]$, by $\pi^*_t(du)=\Phi^*_t(u)du$ the measure whose density is $\Phi_t^*$.
By Lemma \ref{lem:conv_stationary_profile}, for all $\varepsilon>0$, there exists $T_1=T_1(\varepsilon)>0$ such that
    \begin{equation}
    \label{eq:epsilon close stat prof}
        \forall\; t\geq T_1, \quad\|\Phi_t-\Phi_{ss}\|_2^2 \le \varepsilon.
    \end{equation}
    Let $\widetilde\pi$ denotes the path constructed in Lemma \ref{lem:clever path} that connects $\Phi_{ss}$ to $\Psi:=\Phi^*_{1}=\Phi_{T_1}$. Set $T:=T_1+1$ and define 
    \begin{equation}
        \widehat\pi_t=
        \begin{dcases}
            \widetilde\pi_t,\quad 0\leq t\leq 1,\\
            \pi^*_{t},\quad 1\leq t \leq T.
        \end{dcases}
    \end{equation}
    By construction, $\widehat\pi \in C([0,T], \mathcal M)$ connects $\pi_{ss} (du)=\Phi_{ss}(u) du$ to $\varrho$. By the definition of the quasi potential, we have that $V(\varrho)\leq I_{[0,\widetilde T]}(\pi|\Phi_{ss})$ for any $\widetilde T>0$ and for any path $\pi \in C([0, \widetilde T], \mathcal M)$ satisfying $\pi_0=\pi_{ss}$ and $\pi_{\widetilde{T}}=\varrho$. Since our constructed path $\widehat\pi$ satisfies these conditions, we shall estimate its cost in order to obtain an upper bound for the quasi potential. By the definition of the rate function and of the path $\widehat\pi$ we have that 
    \begin{equation}
         I_{[0,T]}(\widehat\pi|\Phi_{ss})=I_{[0,1]}(\widetilde\pi|\Phi_{ss})+I_{[1,T]}(\pi^*|\Phi^*_{1}).
    \end{equation}
    For any $r>0$, define the time shift of $\pi \in C([0,T], \mathcal M)$ by $\tau_r\pi$, where $\tau_r\pi_t =\pi_{t+r}$, $t \in [0,T-r]$. Thanks to the invariance of the rate function with respect to time shifts, we can write
    \begin{equation}\label{eq:rate function time decomp}
        I_{[0,T]}(\widehat\pi|\Phi_{ss})=I_{[0,1]}(\widetilde\pi|\Phi_{ss})+I_{[0,T_1]}(\tau_1\pi^*|\Phi^*_1).
    \end{equation}
    By Lemma \ref{lem:clever path} and \eqref{eq:epsilon close stat prof}, we may estimate the first term on the right-hand side of the previous display by $I_{[0,1]}(\widetilde\pi|\Phi_{ss})\leq C\|\Phi_{T_1}-\Phi_{ss}\|_2^2 \le C \varepsilon$. We estimate now the second term. Let $\lambda^* \in C([0,T_1], \mathcal M)$ be such that $\lambda^*_t:=\tau_1\pi^*_t=\pi^*_{t+1}$ for $t\in [0,T_1]$. Using \eqref{eq:time reversed eq}, we have that $\lambda^*$ satisfies
    \begin{equation}
    \begin{dcases}
        \partial_t\lambda^*_t=-\mathbb L^\gamma\lambda^*_t\\
        \lambda^*_0=\Phi^*_1 \quad\text{and}\quad\lambda^*_{T_1}=\rho.
    \end{dcases}
    \end{equation}
    Moreover, by Lemma \ref{lem:CT}, Lemma \ref{lem:auxiliary mft 1}, and  \eqref{eq:time reversed eq}, we have
    \begin{equation}
        W(\lambda^*_{T_1})-W(\lambda^*_0)=\int_0^{T_1}\langle \partial_t\lambda^*_t,\Gamma_{\lambda^*_t}\rangle\,  dt=-\int_0^{T_1}\langle\mathbb L^\gamma\lambda^*_t,\Gamma_{\lambda^*_t}\rangle\,  dt=\int_0^{T_1}\langle \lambda^*_t,\Gamma_{\lambda^*_t}\rangle_{\gamma/2}\, dt,
    \end{equation}
    where in the last equality we used simply the definition of the fractional inner product. Now, by Lemma \ref{lem:auxiliary mft 2}, we get then
     \begin{equation}
         W(\lambda^*_{T_1})-W(\lambda^*_0)=\int_0^{T_1}\langle \lambda^*_t,\Gamma_{\lambda^*_t}\rangle_{\gamma/2}\, dt=\int_0^{T_1}\|\Gamma_{\lambda^*_t}\|^2_{\gamma/2} \, dt.
    \end{equation}
    Note, again by \eqref{eq:time reversed eq}, that we have $H^*:=2\Gamma_{\lambda^*}$ solves equation \eqref{eq:Dirichlet Equation2} for $\lambda^*$. Hence by the representation of the rate function in Lemma \ref{lem:Poissoneq}, we conclude that 
    \begin{equation}
        \int_0^{T_1}\|\Gamma_{\lambda^*_t}\|^2_{\gamma/2} \, dt=\frac{1}{4}\int_0^{T_1}\|H^*_t\|^2_{\gamma/2} \, dt =I_{[0,T_1]}(\lambda^*|\Phi^*_1).
    \end{equation}
   We conclude that $I_{[0,T_1]}(\lambda^*|\Phi^*_1)=W(\lambda^*_{T_1})-W(\lambda^*_0)$. Therefore, by \eqref{eq:rate function time decomp}, Lemma \ref{eq:epsilon close stat prof} and recalling that $\lambda^*_{T_1}=\varrho$,
   \begin{equation}
       \begin{split}
            I_{[0,T]}(\widehat\pi|\Phi_{ss})&=I_{[0,1]}(\widetilde\pi|\Phi_{ss})+I_{[0,T_1]}(\tau_1\pi^*|\Phi^*_1)\\
            &\leq C\varepsilon+W(\lambda^*_{T_1})-W(\lambda^*_0)\\
            &\leq C\varepsilon+W(\varrho)
       \end{split}
   \end{equation}
    where we also used that $W$ is a positive functional. Due to the arbitrariness of $\varepsilon$, we conclude that $V(\varrho)\leq W(\varrho)$, which finishes the proof.
\end{proof}



\appendix

\section{The Adjoint Operator}
\label{app_adjoint}


We compute now the adjoint operator $\mathcal L_n^*$ with respect to the invariant measure $\mu_{ss}^n$.  Recall the definition of $\mathcal L_n$ {given in \eqref{eq:generator}}. We will compute the adjoint of each term separately. To do so, we shall introduce some notation. Let ${\Phi_x}, {\bar \Phi}_{x}$ be the multiplicative, hence self-adjoint, operators acting on $f:\Omega_n \to \mathbb R$ as
\begin{equation}
({\Phi}_{x} f) (\varphi) = \varphi(x) f(\varphi)\quad\text{and}\quad({\overline \Phi}_{x} f) (\varphi) = \Big[ \varphi(x) - \Phi^n_{ss}(x) \Big] f(\varphi).
\end{equation}

We start with the adjoint {of the boundary dynamics}. Note that 
$$\mathcal L^\ell = \partial^2_{\varphi(1)} +(\Phi_\ell-\varphi(1))\partial_{\varphi(1)}= \partial^2_{\varphi(1)}-{\bar \Phi}_1 \partial_{\varphi(1)}+(\Phi_\ell-\Phi_{ss}^n(1))\partial_{\varphi(1)}.$$ 
Since $\mu^n_{ss}$ is a Gaussian product measure with mean $\Phi^n_{ss}(x)$ and variance 1 at each site $x$, the adjoint of the partial derivative $\partial_{\varphi (x)}$ is given by 
\begin{equation}
 \label{eq:adjointphix}
      (\partial_{\varphi(x)})^* = -\partial_{\varphi(x)} + {\bar \Phi}_x.
 \end{equation}
Therefore we have that
\begin{equation}
\begin{split}
({\mathcal L^\ell})^* &= \partial^2_{\varphi(1)}-{\bar \Phi}_1 \partial_{\varphi(1)}-(\partial_{\varphi(1)}-\bar\Phi_1)[\Phi_\ell-\Phi_{ss}^n(1)]\\
&= {\mathcal L}^\ell + [\Phi_{ss}^n (1) - \Phi_\ell] \left( 2 \partial_{\varphi(1)} -{\bar \Phi}_1\right) .
\end{split}
\end{equation}
This implies that 
\begin{equation}\label{eq:symmetric_lef}
    \mathcal S^\ell=\cfrac{\mathcal L^\ell + (\mathcal L^\ell)^*}{2}= \partial^2_{\varphi(1)}-{\bar \Phi}_1 \partial_{\varphi(1)}+ \frac{1}{2}\bar\Phi_1[\Phi_\ell-\Phi_{ss}^n(1)]
\end{equation}
and 
\begin{equation}\label{eq:antisymmetric_lef}
    \mathcal A^\ell= \cfrac{\mathcal L^\ell - (\mathcal L^\ell)^*}{2} = [\Phi_\ell-\Phi_{ss}^n(1)] \left(\partial_{\varphi(1)}-\frac 12\bar\Phi_1 \right).
\end{equation}
Analogously, we have that 
\begin{equation}
\begin{split}
({\mathcal L^r})^* & = \partial^2_{\varphi(n-1)}-{\bar \Phi}_{n-1} \partial_{\varphi(n-1)}-(\partial_{\varphi(n-1)}-\bar\Phi_{n-1})[\Phi_r-\Phi_{ss}^n(n-1)]\\
& = {\mathcal L}^r + [\Phi_{ss}^n (n-1) - \Phi_r] \left( 2 \partial_{\varphi(n-1)} -{\bar \Phi}_{n-1} \right) .
\end{split}
\end{equation}
and 
\begin{equation}
    \mathcal S^r= \partial^2_{\varphi(n-1)}-{\bar \Phi}_{n-1} \partial_{\varphi(n-1)}+ \frac{1}{2}\bar\Phi_{n-1}[\Phi_r-\Phi_{ss}^n(n-1)].
\end{equation}
\begin{equation}\label{eq:antisymmetric_right}
    \mathcal A^r=[\Phi_r-\Phi_{ss}^n(n-1)] \left(\partial_{\varphi(n-1)}-\frac 12\bar\Phi_{n-1}\right).
\end{equation}

We now compute the adjoint of the bulk generator $\mathcal{L}^{\text{bulk}}$. Consider the Hamiltonian $\mathscr H:\mathbb R^{\Lambda_n}\to\mathbb R$ defined by
\begin{equation}
\label{eq:Hamiltonian}
\mathscr H(\varphi) = \frac{1}{2} \sum_{x \in \Lambda_n} (\varphi(x) - \Phi_{ss}^n (x) )^2
\end{equation}
and the operator $\mathcal L_n^{\mathscr H}$ defined by $\mathcal L_n^{\mathscr H} = \sum_{x,y \in \Lambda_n} p(y-x) e^{\mathscr H} D_{x,y} \big( e^{-\mathscr H} D_{x,y} \big)$ where $e^{\pm \mathscr H}$ are multiplicative operators and $D_{x,y}$ is the first order differential operator $D_{x,y}=\partial_{\varphi(y)} - \partial_{\varphi(x)}.$ We can rewrite $\mathcal L_n^\mathscr H$ as 
\begin{equation}\label{eq:genham}
\begin{split}
\mathcal L_n^\mathscr H &= \sum_{x,y \in \Lambda_n} p(y-x) \Big\{ D_{x,y}^2 -(D_{x,y} \mathscr H) D_{x,y} \Big\}= \sum_{x,y \in \Lambda_n} p(y-x) \Big\{ D_{x,y}^2 -({\bar \Phi}_y -{\bar \Phi}_x) D_{x,y} \Big\}\\
&=2 {\mathcal L}^{\rm{bulk}} + \sum_{x,y \in \Lambda_n} p(y-x) (\Phi^n_{ss}(y)-\Phi^n_{ss}(x)) D_{x,y} .
\end{split}
\end{equation}
Observe that $\mathcal L_n^\mathscr H$ is self-adjoint in $L^2 (\mu_{ss}^n)$ since $\mu_{ss}^n (d\varphi) \propto e^{- \mathscr H (\varphi)} d\varphi$. On the other hand 
\begin{equation}
\label{eq:adjointDxy}
\begin{split}
    D_{x,y}^* &= (\partial_{\varphi(y)} - \partial_{\varphi(x)})^*= -D_{x,y} + {\bar \Phi_y} -{\bar \Phi}_x.
\end{split}
\end{equation}
It follows that 
\begin{equation}
\begin{split}
    (\mathcal L^{\rm{bulk}})^*& = \frac12\mathcal L_n^\mathscr H + \frac12 \sum_{x,y \in \Lambda_n} p(y-x) (\Phi^n_{ss}(y)-\Phi^n_{ss}(x)) D_{x,y}\\
    &-\frac12 \sum_{x,y \in \Lambda_n} p(y-x) (\Phi^n_{ss}(y)-\Phi^n_{ss}(x)) ({\bar \Phi_y} -{\bar \Phi}_x) \\
    &=\mathcal L^{\rm{bulk}} + \sum_{x,y \in \Lambda_n} p(y-x) (\Phi^n_{ss}(y)-\Phi^n_{ss}(x)) D_{x,y}+\sum_{x,y \in \Lambda_n} p(y-x) (\Phi^n_{ss}(y)-\Phi^n_{ss}(x)) {\bar \Phi}_x.
\end{split}  
\end{equation}
Recall  \eqref{eq:discreteprofile}. Multiply this equation by $\bar\Phi_x$ and sum other $x\in \Lambda_n$. Then
\begin{equation}
    \begin{split}
       (\mathcal L^{\rm{bulk}})^* & = \mathcal L^{\rm{bulk}} + \sum_{x,y \in \Lambda_n} p(y-x) (\Phi^n_{ss}(y)-\Phi^n_{ss}(x)) D_{x,y}\\
       & +\bar\Phi_1 (\Phi_{ss}^n (1) -\Phi_\ell) + \bar\Phi_{n-1} (\Phi_{ss}^n (n-1) - \Phi_r). 
    \end{split}
\end{equation}
It follows that
\begin{equation}
\begin{split}
\mathcal{S}^{\text{bulk}} 
&= \mathcal{L}^{\text{bulk}}+ \frac 12\sum_{x,y\in\Lambda_n} p(y-x) ( \Phi^n_{ss}(y) - \Phi^n_{ss}(x)) D_{x,y} \\
&+\frac{1}{2} \bar\Phi_1 (\Phi_{ss}^n (1) -\Phi_\ell) + \frac12 \bar\Phi_{n-1} (\Phi_{ss}^n (n-1) - \Phi_r). 
\end{split}
\end{equation}
and 
\begin{equation}
\begin{split}
\mathcal{A}^{\text{bulk}} 
&=- \frac 12\sum_{x,y\in\Lambda_n} p(y-x) ( \Phi^n_{ss}(y) - \Phi^n_{ss}(x)) D_{x,y} \\
 &-\frac{1}{2} \bar\Phi_1 (\Phi_{ss}^n (1) -\Phi_\ell) + \frac12 \bar\Phi_{n-1} (\Phi_{ss}^n (n-1) - \Phi_r). 
\end{split}
\end{equation}
In conclusion, we have that 
\begin{equation}
\label{eq:adjoint-formula2}
\begin{split}
{\mathcal L}_n^* &= {\mathcal L}_n +n^\gamma  \sum_{x,y\in\Lambda_n} p(y-x) ( \Phi^n_{ss}(y) - \Phi^n_{ss}(x)) D_{x,y} \\
 &+{n^\gamma} \bar\Phi_1 (\Phi_{ss}^n (1) -\Phi_\ell) + {n^\gamma} \bar\Phi_{n-1} (\Phi_{ss}^n (n-1) - \Phi_r)\\
 &+n^\gamma(\Phi^n_{ss}(1) - \Phi_\ell)(2\partial_{\varphi(1)}-\bar\Phi_1)+ n^\gamma(\Phi_{ss}^n (n-1) - \Phi_r(2\partial_{\varphi(n-1)}-\bar\Phi_{n-1})\\
 &={\mathcal L}_n +n^\gamma  \sum_{x,y\in\Lambda_n} p(y-x) ( \Phi^n_{ss}(y) - \Phi^n_{ss}(x)) D_{x,y} \\
 &+ 2 n^\gamma (\Phi^n_{ss}(1) - \Phi_\ell) \partial_{\varphi(1)} + 2 n^\gamma (\Phi^n_{ss}(n-1) - \Phi_r) \partial_{\varphi (n-1)}
\end{split}
\end{equation}
and hence the symmetric part of $\mathcal L_n$ is given by 
\begin{equation}
\label{eq:S_n-expression}
\begin{split}
{\mathcal S}_n &= {\mathcal L}_n +\frac{n^\gamma}{2}  \sum_{x,y\in\Lambda_n} p(y-x) ( \Phi^n_{ss}(y) - \Phi^n_{ss}(x)) D_{x,y} \\
 &+ n^\gamma (\Phi^n_{ss}(1) - \Phi_\ell) \partial_{\varphi(1)} +  n^\gamma (\Phi^n_{ss}(n-1) - \Phi_r) \partial_{\varphi (n-1)} ,
\end{split}
\end{equation}
and the antisymmetric part is
\begin{equation}
\label{eq:antisymmetric2}
\begin{split}
\mathcal{A}_n&= - \frac {n^\gamma}{2}\sum_{x,y} p(y-x) ( \Phi^n_{ss}(y) - \Phi^n_{ss}(x)) D_{x,y} \\
 &- n^\gamma (\Phi^n_{ss}(1) - \Phi_\ell) \partial_{\varphi(1)} -  n^\gamma (\Phi^n_{ss}(n-1) - \Phi_r) \partial_{\varphi (n-1)}.
\end{split}
\end{equation}

Now we claim that $\mu^n_{ss}$ is indeed invariant for $\mathcal L_n$. We thus need to check that ${\mathcal L}_n^* {\mathbf 1} =0$. But
\begin{equation}
\label{eq:adjoint-formula_1}
\begin{split}
{\mathcal L}_n^* \mathbf1 =&-\frac{n^\gamma}{2} \sum_{x,y} p(y-x) ( \Phi^n_{ss}(y) - \Phi^n_{ss}(x))(\Phi_y -\Phi_x) + \frac{n^{\gamma}}{2} \sum_{x,y} p(y-x) ( \Phi^n_{ss}(y) - \Phi^n_{ss}(x))^2\\
 &{+n^\gamma(\Phi_\ell-\Phi^n_{ss}(1))\bar\Phi_1+n^\gamma(\Phi_r-\Phi^n_{ss}(n-1))\bar\Phi_{n-1}}.
\end{split}
\end{equation}
And our goal is to show that the last display is null. Recall \eqref{eq:phissneharmonic} and multiply it by $\varphi(x)-\Phi_{ss}^n(x)$ and then sum over $x\in\Lambda_n$. 
We get 
\begin{equation} \begin{split}
    &\sum_{x,y \in \Lambda_n} p(x-y) (\Phi_{ss}^n (y) -\Phi_{ss}^n (x))(\varphi(x)-\Phi_{ss}^n(x)) \\+&(\Phi_\ell - \Phi_{ss}^n (1)) (\varphi(1)-\Phi_{ss}^n(1)) +(\Phi_r -\Phi_{ss}^n (n-1))(\varphi(n-1)-\Phi_{ss}^n(n-1))=0.
\end{split}
\end{equation}
We can write the last display as 
\begin{equation} \begin{split}
    &\sum_{x,y \in \Lambda_n} p(x-y) (\Phi_{ss}^n (y) -\Phi_{ss}^n (x))\Phi_x
    -\sum_{x,y \in \Lambda_n} p(x-y) (\Phi_{ss}^n (y) -\Phi_{ss}^n (x))
    \Phi_{ss}^n(x) \\&+(\Phi_\ell - \Phi_{ss}^n (1)) \bar\Phi_1+(\Phi_r -\Phi_{ss}^n (n-1))\bar\Phi_{n-1}=0,
\end{split}
\end{equation} 
By writing the terms in the first line of last display as one half plus one half of each of them and making a changes of variables plus the fact that the transition probability $p(\cdot)$ is symmetric, last display writes as
\begin{equation} \begin{split}\label{eq:check_adjoint}
    &\frac{1}{2}\sum_{x,y \in \Lambda_n} p(x-y) (\Phi_{ss}^n (y) -\Phi_{ss}^n (x))(\Phi_x-\Phi_y)
    \\-&\frac 12\sum_{x,y \in \Lambda_n} p(x-y) (\Phi_{ss}^n (y) -\Phi_{ss}^n (x))
   ( \Phi_{ss}^n(x)- \Phi_{ss}^n(y)) \\+&(\Phi_\ell - \Phi_{ss}^n (1)) \bar\Phi_1+(\Phi_r -\Phi_{ss}^n (n-1))\bar\Phi_{n-1}=0
\end{split}
\end{equation}
and this proves the claim.

\section{Moments and Entropy Bounds}

\begin{lemma}
\label{lem:Patricia0}
Assume \eqref{eq:Kr} and recall the definition of the NESS $\mu_{ss}^n$ given in Lemma \ref{lem:NESS} and of the reference probability measure $\mu^n_{\Phi(\cdot)}$ from \eqref{eq:initial_measure}. Then for any $q \in (1,r)$ there exists a constant $\kappa_q>0$ such that
\begin{equation}
\label{eq:kappaq}
\begin{split}
\forall n \ge 2, \quad \int_{\Omega_n} \left(\frac{d\mu_n}{d\mu_{ss}^n} \right)^q d\mu_{ss}^n \le e^{\kappa_q n},
\end{split}
\end{equation}
\begin{equation}
\label{eq:kappaqbis}
\begin{split}
\forall n \ge 2, \quad \int_{\Omega_n} \left(\frac{d\mu_n}{d\mu_{\Phi(\cdot)}^n} \right)^q d\mu_{\Phi(\cdot)}^n \le e^{\kappa_q n}.
\end{split}
\end{equation}
\begin{equation}
\label{eq:kappaqbis_new}
\begin{split}
\forall n \ge 2, \quad \int_{\Omega_n} \left(\frac{d\mu_{ss}^n}{d\nu_{n}} \right)^q d\nu_{n} \le e^{\kappa_q n}.
\end{split}
\end{equation}
\end{lemma}

\begin{proof}
Since the proof of the three inequalities are the same we only prove the first one. Let $\alpha=r/q>1$ and $\beta>1$ such that $1/\alpha+1/\beta=1$. By H\"older's inequality, observing that $q\alpha=r$ and using \eqref{eq:Kr}, we have
\begin{equation}
\begin{split}
 \int_{\Omega_n} \left(\frac{d\mu_n}{d\mu_{ss}^n} \right)^q d\mu_{ss}^n &=  \int_{\Omega_n}  \left( \cfrac{d\mu_n}{d\nu_n}\right)^q \left(\frac{d\nu_n}{d\mu_{ss}^n} \right)^{q-1} d\nu_n\\
 &\le \left( \int_{\Omega_n} d\nu_n \left(  \cfrac{d\mu_n}{d\nu_n}\right)^{q\alpha} \right)^{1/\alpha} \, \left( \int_{\Omega_n}   d \nu_n \left(\frac{d\nu_n}{d\mu_{ss}^n} \right)^{(q-1)\beta} \right)^{1/\beta}\\
 &\le e^{K_r n /\alpha}   \, \left( \int_{\Omega_n}   d \nu_n \left(\frac{d\nu_n}{d\mu_{ss}^n} \right)^{(q-1)\beta} \right)^{1/\beta}.
\end{split}
\end{equation}
By the form of $\mu_{ss}^n$ provided in Lemma \ref{lem:NESS}, we have that
\begin{equation}
\begin{split}
\left( \frac{d\nu_n}{d\mu_{ss}^n}\right) (\varphi) = \exp\left\{ -\sum_{x \in \Lambda_n} \Phi_{ss}^n (x) \varphi(x) +\frac12 \sum_{x\in \Lambda_n} \Big[\Phi_{ss}^n (x) \Big]^2  \right\}.
\end{split}
\end{equation}
By \eqref{eq:boundprofileness}, it is not difficult to see that there exists a constant $C:=C(\beta, q, \Phi_\ell, \Phi_r)$ such that
\begin{equation}
\begin{split}
\left( \int_{\Omega_n}   d \nu_n \left(\frac{d\nu_n}{d\mu_{ss}^n} \right)^{(q-1)\beta} \right)^{1/\beta} \le e^{C n}.
\end{split}
\end{equation}
This concludes the proof with $\kappa_q = K_r/\alpha +C$.
\end{proof}

\begin{lemma}
\label{lem:Patricia1}
Let $(X,\mu)$ be a probability space and $f$ a density probability with respect to $\mu$. Then for any $q>1$ we have
\begin{equation}
\begin{split}
\int_X f \log f d \mu \le \frac{1}{q-1} \log \left( \int_X f^q d\mu \right).
\end{split}
\end{equation}
\end{lemma}

\begin{proof}
Consider the probability measure $d\nu = f d\mu$. Since the function $\log$ is concave, by Jensen's inequality
\begin{equation}
\begin{split}
\int_X f \log f d\mu &= \frac{1}{q-1} \int_X f \log \Big( f^{q-1} \Big) d\mu = \frac{1}{q-1} \int_X \log \Big( f^{q-1} \Big) d\nu \le  \frac{1}{q-1} \log \left( \int_X f^{q-1} d\nu \right). 
\end{split}
\end{equation}
Since $\int_X f^{q-1} d\nu =\int_X f^q d\mu$, this concludes the proof.
\end{proof}

\begin{lemma}
\label{lem:mussnmun}
Let $q>1$ and $\kappa_q$ the constants appearing in \eqref{eq:kappaq} or \eqref{eq:kappaqbis}. Let $p>1$ such that $p^{-1} +q^{-1} =1$. Let $A$ be an event belonging to $\sigma( \varphi (t) \; , \; 0 \le t \le T)$. Then  
\begin{equation}
\begin{split}
 \frac{1}{n} \log \mathbb P_{\mu_n} (A) \le  \cfrac{1}{pn} \log \mathbb P_{\mu_{ss}^n} (A) + \cfrac{\kappa_q}{q}\quad\text{and}\quad  \frac{1}{n} \log \mathbb P_{\mu_n} (A) \le  \cfrac{1}{pn} \log \mathbb P_{\mu_{\Phi(\cdot)}^n} (A) + \cfrac{\kappa_q}{q}.
\end{split}
\end{equation}
\end{lemma}

\begin{proof}
We have
\begin{equation}
        \mathbb P_{\mu_n}(A)=\int \mathbf{1}_A \frac{d\mathbb P_{\mu_n} }{d\mathbb P_{\mu^n_{ss}} } d\mathbb P_{\mu^n_{ss}} =\int \mathbf{1}_A \frac{d\mu_n} {d\mu^n_{ss}}  d\mathbb P_{\mu^n_{ss}}. 
    \end{equation}
Using H\"older's inequality, we get 
\begin{equation}
        \mathbb P_{\mu_n}(A)\leq \Big(\int \mathbf{1}_A  d\mathbb P_{\mu^n_{ss}} 
\Big)^{1/p}\left(\int \Big(\frac{d\mu_n} {d\mu^n_{ss}}\Big)^q  d {\mu^n_{ss}} \right)^{1/q}
    \end{equation}
and from \eqref{eq:kappaq} we conclude that 
    \begin{equation}\label{eq:est_tight}
        \mathbb P_{\mu_n}(A)\leq \Big(   \mathbb P_{\mu^n_{ss}} (A)\Big)^{1/p} e^{\frac{\kappa_q n}{q}}
    \end{equation} 
from where the result follows. 
\end{proof}

\begin{lemma}
\label{lem:super_exchange_measures}
There exists a positive constant $C:=C(H)$ such that for any $d>0$ and $n \ge 2$,
\begin{equation}
\label{eq:bound power d+1 exp mart}
    \mathbb E_{n}\Big[\Big(\frac{d\mathbb P^H_n}{d\mathbb P_n}\Big)^{1+d}\Big]\leq\exp{\Big(Cd(1+d)n\Big)}.
\end{equation}
Hence for any finite constant $d>0$ and any event $A$ belonging to $\sigma( \varphi_t \; , \; 0 \le t \le T)$
\begin{equation}
\label{eq:Glucksmann}
    \limsup_{n\to\infty}\frac{1}{n}\log\mathbb P_n^H(A)\leq Cd+\frac{d}{1+d}\limsup_{n\to\infty}\frac{1}{n}\log\mathbb P_n(A).
\end{equation}
\end{lemma}

\begin{proof}
We note that this lemma would be trivial if the values taken by $\varphi(x)$ were bounded, but in the present context, it requires a proof. The second inequality is a trivial consequence of the first one. Hence we prove only the first.  By Girsanov's theorem, we have that  $\frac{d\mathbb P^H_n}{d\mathbb P_n} =\mathbb M_T^H $ where $\mathbb M^H$ is the exponential martingale defined, for any $t \in [0,T]$ by
\begin{equation}
    \mathbb M_t^H=\exp\Big\{F(\varphi_t)-F(\varphi_0)-\int_0^t e^{-F(\varphi_s)}(\partial_s+\mathcal{L}_n)e^{F(\varphi_s)}ds\Big\}
\end{equation}
with $F(\varphi_t)=n\langle \pi_t^n, H_t\rangle$. The $d+1$ power of $\mathbb M^H_t$ can be rewritten as
\begin{equation}
\label{eq:power d+1 exp mart}
    \Big(\frac{d\mathbb P^H_n}{d\mathbb P_n}\Big)^{1+d}=\mathbb E_{n}\Big[\exp{\Big((1+d)n M^{H,n}_t (H) -(1+d)\tfrac{n^2}{2}\langle M^{H,n}(H)\rangle_t\Big)}\Big]
\end{equation}
where $M^{H,n}_t(H)$ and $\langle M^{H,n} (H) \rangle_t$ were introduced in \eqref{eq:dynkin martingale} and \eqref{eq:QV}, respectively. Since the quadratic variation is deterministic, it can go out of the expectation. Moreover, since the first term in the above exponential is again a martingale, we can compensate the expression with its quadratic variation so that \eqref{eq:power d+1 exp mart} is equal to
\begin{equation}
\begin{split}
& \exp{\Big((1+d)^2\tfrac{n^2}{2}\langle M^{H,n}(H)\rangle_t-(1+d)\tfrac{n^2}{2}\langle M^{H,n}(H) \rangle_t\Big)} \\
&\quad \quad \times \mathbb E_{n}\Big[\Big((1+d)n M^{H,n}_t (H)-(1+d)^2\tfrac{n^2}{2}\langle M^{H,n} (H) \rangle_t\Big)\Big].
    \end{split}
\end{equation}
Note that the term inside the expectation is of the form $\theta M^{H,n}_t (H)-(\theta^2/2)\langle M^{H,n}(H)  \rangle_t$, for $\theta=(1+d)n>0$, which since $M^{H,n} (H)$ is a martingale, is again a martingale and, in addition, with mean one. Therefore,
\begin{equation}
      \mathbb E_{n}\Big[\Big(\frac{d\mathbb P^H_n}{d\mathbb P_n}\Big)^{1+d}\Big]=\exp{\Big(d(1+d)\tfrac{n^2}{2}\langle M^{H,n} (H) \rangle_t\Big)}.
\end{equation}
We claim now that $\langle M^{H,n}(H) \rangle_t$ given by \eqref{eq:QV} is of order $1/n$. Indeed, for the boundary part in \eqref{eq:QV}, since the test functions $H$ vanishes at the boundary, we have that for $x=1/n$ and $x=(n-1)/n$, $H_s(x)=O(1/n^2)$. Hence, the boundary term is of order $n^{\gamma-4}$. In the above exponential, it contributes with $n^{\gamma-2}$ and since $\gamma\in(1,2)$, this term vanishes. For the bulk term, we note that it is equal to 
\begin{equation}
    \frac{1}{n^{\gamma+1}}\frac{n^\gamma}{n^2}c_\gamma\int_0^t
\sum_{x,y\in\Lambda_n}\frac{1}{\big|\tfrac yn-\tfrac xn\big|^{\gamma+1}}\big[H_s(\tfrac yn)-H_s(\tfrac xn)\big]^2ds
\end{equation}
which is bounded from above by $(1/n)\|H_s\|^2_{\gamma/2}$. As a consequence, we get \eqref{eq:bound power d+1 exp mart}.
\end{proof}

\section{A Non-Reversible Maximal Inequality}
\label{app_controlling the mass}

In this section, we  establish a non-reversible maximal inequality for continuous-time Markov processes with  invariant probability measures that are not necessarily reversible. Specifically, we show that the classical estimate of Kipnis and Varadhan \cite{kipnis1986central} remains valid for a non-reversible continuous-time Markov process, assuming its generator satisfies a sector condition. See \cite{komorowski2012fluctuations} for a comprehensive review of the subject.

Let $\{X_t:t\geq0\}$ denote a continuous-time Markov chain with state-space $\mathcal X$, generator $L$ and assume that it has an invariant (not necessarily reversible) probability measure $\mu$. Hence $L$ is not apriori a self-adjoint operator in $L^2(\mu)$. Let $L^*$ denote the adjoint of $L$ in $L^2 (\mu)$. We assume that $L$ and $L^*$ have a common core $\mathcal C$. One can decompose the generator into the sum of its symmetric and anti-symmetric parts in $L^2(\mu)$ as $L=S+A$ where
\begin{equation}
    S=\frac{L+L^*}{2}\quad\text{and}\quad A=\frac{L-L^*}{2}.
\end{equation}
Let us define the Dirichlet form by
\begin{equation}
    \mathfrak D(f)= \langle f, -Lf \rangle_\mu = \langle f,-Sf\rangle_\mu \ge 0.
\end{equation}
The generator $L$ satisfies a \textit{Sector Condition} if there exists a constant $K>0$ such that 
\begin{equation}
\label{eq:sector_cond}
    \langle f,-Lg\rangle_\mu^2\leq K\langle f,-Sf\rangle_\mu \langle g,-Sg\rangle_\mu
\end{equation}
for all $f,g \in \mathcal C$. In words, this means that we can control the behavior of $L$ by its symmetric part. 
Given a function $g:\mathcal X\to\mathbb R$ such that $ E_\mu(g)=0$, we define the $H_1$-norm of $g$ by the variational formula
\begin{equation}
\|g\|^2_{-1}=\sup_{f\in \mathcal C} \Big\{2\langle g,f\rangle-\langle f,-S f\rangle\Big\}.
\end{equation}
Formally we have that $\|g\|^2_{-1}=\langle g, (-S)^{-1}g\rangle_\mu$. We state now a celebrated result  due to Kipnis and Varadhan.

\begin{lemma}[Kipnis-Varadhan Inequality]
\label{lem:kipnis-varadhan} 
Let $V:\mathcal X \to\mathbb R$ be a mean zero function in $L^2(\mu)$. Then, there exists a universal constant $C_0>0$ such that
\begin{equation}
     E_\mu\Bigg[\sup_{0\leq t\leq T}\Bigg(\int_0^T V(X_s)ds\Bigg)^2\Bigg]\leq C_0 T \Vert V\Vert^2_{-1}.
\end{equation}
\end{lemma}

We can state now a \emph{non-reversible maximal inequality} which generalizes an inequality first derived in \cite{KV86} for the (reversible)  symmetric simple exclusion process and which is a consequence of the Kipnis-Varadan inequality. This (reversible) maximal inequality has been also used in other reversible particle systems such as the (nearest-neighbor) Ginzburg-Landau dynamics \cite[proof of Lemma 6.1]{GPV88}. The reader can consult \cite[Appendix 1, Theorem 11.1]{KL} for a general statement for jump Markov processes. But as far as we know, this `maximal inequality' has never been extended nor used for non-reversible systems.

\begin{theorem}[Non-Reversible Maximal Inequality]
\label{thm:maximal_ineq} 
Let $L$ denote the generator of a continuous time Markov process $(X_t)_{t\geq0}$ with an invariant (not necessarily reversible) probability measure $\mu$ and core $\mathcal C$. Suppose that $L$ satisfies a sector condition of the form \eqref{eq:sector_cond} with constant $K$. Let us fix a time  horizon $T>0$. Then, there exists a universal constant $C_0=C_0(T)>0$, depending only on $T$, such that, for any function $g:\mathcal X\to\mathbb R$ in the domain of $L$, we have that 
\begin{equation}
       P_\mu\Big(\sup_{0\leq t\leq T}g(X_t)\geq\ell\Big)\leq \frac{C_0}{\ell}\sqrt{\langle g,g\rangle_\mu+T (K^2+1) \;\mathfrak{D}(g)}.
\end{equation}
\end{theorem}

\begin{proof}
   We consider the Dynkin's martingale for the function $g$. The process $(M_t)_{t\ge 0}$ defined by
    \begin{equation}
       M_t:=g(X_t)-g(X_0)-\int_0^tLg(X_s)\;ds
    \end{equation}
    is a mean zero martingale with respect to the natural filtration of the process.  By the triangle inequality we have the upper bound 
    \begin{equation}
        \sup_{0\leq t\leq T}g(X_t)\leq g(X_0)+\sup_{0\leq t\leq T}|M_t|+\Bigg|\sup_{0\leq t\leq T}\int_0^tLg(X_s)\;ds\Bigg|,
    \end{equation}
    which implies that the probability $ P_\mu\Big(\sup_{0\leq t\leq T}g(X_t)\geq \ell\Big)$ is bounded from above by
    \begin{equation}
    \begin{split}
           P_\mu\Big(g(X_0)\geq \ell\Big)+    P_\mu\Big(\sup_{0\leq t\leq T}|M_t|\geq \ell\Big)+  P_\mu\Bigg(\Bigg|\sup_{0\leq t\leq T}\int_0^tLg(X_s)\;ds\Bigg|\geq \ell\Bigg).
         \end{split}
    \end{equation}
    Applying Markov's inequality we get that the last display is bounded by
    \begin{equation}
    \begin{split}
       \frac{1}{\ell}   E_\mu\big[|g(X_0)|\big]+\frac{1}{\ell}   E_\mu\Big[\sup_{0\leq t\leq T}|M_t|\Big] +\frac{1}{\ell}   E_\mu\Bigg[\Bigg|\sup_{0\leq t\leq T}\int_0^tLg(X_s)\;ds\Bigg|\Bigg].
        \end{split}
    \end{equation}
    We now bound  each term in last display. From Jensen's inequality we have $ E_\mu(|g(X_0)|)\leq\sqrt{\langle g,g\rangle_\mu}.$ For the martingale term, we apply  Jensen's and  Doob's inequalities, to get 
    \begin{equation}
           E_\mu\Big(\sup_{0\leq t\leq T}|M_t|\Big)\leq \sqrt{   E_\mu\Big(\sup_{0\leq t\leq T}|M_t|^2\Big)}\leq \sqrt{4   E_\mu\big(M_T^2\big)}=2\sqrt{  E_\mu\big(\langle M\rangle_T\big)}.
    \end{equation}
    Note that the quadratic variation of the martingale $(M_t)_{t\ge 0}$ is given by $\langle M\rangle_t=\int_0^tLg^2(X_s)-2g(X_s)Lg(X_s)\;ds$. Thus, since we are under the invariant measure ($ E_\mu(Lf)=0$, for any $f$) 
    \begin{equation}
    \begin{split}
           E_\mu\big(Lg^2(X_s)\big)-2   E_\mu(g(X_s)Lg(X_s))
         = E_\mu\big(Lg^2\big)-2 E_\mu(g Lg) =2\langle g,-Lg\rangle_\mu= 2 {\mathfrak D} (g). 
    \end{split}
    \end{equation}
    Therefore, we can control the martingale term by \begin{equation}\label{eq:control_martingale_max ineq}
           P_\mu\Big(\sup_{0\leq t\leq T}|M_t|\geq \ell\Big)\le\frac{1}{\ell}\sqrt{8 T\;\mathfrak{D}(g)}.
    \end{equation}
    We turn now to the integral term. Again, by Jensen's inequality,
    \begin{equation}
         E_\mu\Big[\Big|\sup_{0\leq t\leq T}\int_0^t Lg(X_s)\;ds\Big|\Big]\leq\sqrt{    E_\mu\Big[\sup_{0\leq t\leq T}\Big(\int_0^t Lg(X_s)\;ds\Big)^2\Big]}.
    \end{equation}
    By Lemma \ref{lem:kipnis-varadhan}, since $E_\mu\big[Lg\big]=0$, we have 
    \begin{equation}
          E_\mu\Big[\sup_{0\leq t\leq T}\Big(\int_0^t Lg(X_s)\;ds\Big)^2\Big]\leq C_0 T \Vert Lg \Vert ^2_{-1}\le C_0 T\sup_{f \in \mathcal C}\Big\{2\langle f,Lg\rangle_\mu-\langle f,-S f\rangle_\mu\Big\}.
    \end{equation}  
    By the sector condition \eqref{eq:sector_cond}, the supremum on the right-hand side of the previous display is bounded above by 
    \begin{equation}
    \begin{split}
        &\sup_f\Big\{2K\sqrt{\langle f,-Sf\rangle_\mu}\sqrt{\langle g,-Sg\rangle_\mu}-\langle f,-Sf\rangle_\mu\Big\}\le \sup_{u\ge 0} \Big\{2K\sqrt{\mathfrak D (g)} \,  u -u^2\Big\}= K^2 {\mathfrak D} (g). 
    \end{split}
    \end{equation}
We conclude that
\begin{equation}
      P_\mu\Bigg(\sup_{0\leq t\leq T}\int_0^tLg(X_s)\;ds\geq \ell\Bigg)\leq \frac{\sqrt{C_0T K^2\mathfrak{D}(g)}}{\ell}.
\end{equation}
Therefore, applying the trivial inequality $\sqrt a+\sqrt{b}+\sqrt{c}\leq\sqrt{3}\sqrt{a+b+c}$, we get
    \begin{equation}
          P_\mu\Big(\sup_{0\leq t\leq T}g(X_t)\geq \ell\Big)\le\frac{\sqrt{3}}{\ell}\sqrt{\langle g,g\rangle_\mu+(8+C_0K^2)T \;\mathfrak{D}(g)}.
    \end{equation}  
    Hence, it is possible to find some $C'>0$ such that
    \begin{equation}
         P_\mu\Big(\sup_{0\leq t\leq T}g(X_t)\geq \ell\Big)\le\frac{C'}{\ell}\sqrt{\langle g,g\rangle_\mu+(1+K^2)T \;\mathfrak{D}(g)}.
    \end{equation}
\end{proof}    
    
\section{Sector Condition for $\mathcal L_n$}\label{app. sector cond}

We now prove that $\mathcal L_n$ satisfies a sector condition. We begin by establishing a Poincar\'e inequality for $\mathcal L_n$ w.r.t. its stationary state $\mu_{ss}^n$, which plays a key role in estimating the antisymmetric part of the generator and in proving the sector condition. The inequality is obtained via the classical link between Poincar\'e and logarithmic Sobolev inequalities, relying in particular on the well-known result of Bakry and Émery.

Let $\Lambda$ be a finite set and consider a probability measure $\mu$ on $\mathbb R^{\Lambda}$. For a function $f \in L^2 (\mu)$, we define its variance by $\text{Var}_\mu(f):=E_\mu[f^2]-E_\mu[f]^2$. If $f:\mathbb R^\Lambda\to [0, \infty)$ is such that $f\log f$ is integrable, its entropy is denoted by
\begin{equation}
     H_\mu(f)=E_\mu[f\log f]-E_\mu[f]\log E_\mu[f].
\end{equation}
Let $\mathcal H^1(\mu) = \{ f \in L^2 (\mu) \; ; \; \nabla f \in L^2 (\mu)  \}$ be the usual Sobolev space w.r.t. $\mu$. We recall the following well known definitions: a measure $\mu$ satisfies 
\begin{enumerate}
    \item a Poincaré inequality if there exists a constant $\alpha>0$ such that for every function $f\in \mathcal H^1 (\mu)$,
    \begin{equation}
        \text{Var}_\mu(f)\leq \frac{1}{\alpha} E_\mu[|\nabla f|^2].
    \end{equation}
    \item a log-Sobolev inequality if there exists a constant $\alpha>0$ such that, for every function $f \in \mathcal H^1 (\mu)$, 
    \begin{equation}
         H_\mu(f^2)\leq \frac{2}{\alpha} E_\mu[|\nabla f|^2].
    \end{equation}
\end{enumerate}
We state now a very useful criteria to verify if a measure satisfies a Log-Sobolev inequality due to Bakry and Émery \cite{bakry2006diffusions}.

\begin{theorem}[Bakry-Émery Criteria]
\label{thm:Bakry_Emery}
Consider a probability measure $\mu$ on $\mathbb R^{\Lambda}$ of the form \begin{equation}
    \mu(d\varphi)=\frac{1}{Z_{\varphi}}e^{-\mathscr  H(\varphi)}\;d\varphi ,
\end{equation}
where the two times continuously differentiable Hamiltonian $\mathscr H: \mathbb R^\Lambda \to [0, \infty)$ satisfies, as quadratic forms,
$    \rm{Hess}\,  \mathscr H\geq\lambda\; \text{Id},$
for some $\lambda>0$. Then, the Log-Sobolev constant $\alpha$ of $\mu$ satisfies $\alpha\geq \lambda$.
\end{theorem}
The next result gives the connection between the log-Sobolev inequality and other functional inequalities.

\begin{proposition}[See Lemma 2.2.2 in \cite{saloff2006lectures}]
The Log-Sobolev inequality with constant $\alpha$ implies the Poincaré inequality with the same constant.
\end{proposition}
We introduce and compute now the Dirichlet form ${\mathfrak D}^n$ associated to $\mathcal L_n$ which is defined, for any smooth function $f$ with compact support, by ${\mathfrak D}^n (f) = \langle f,  -\mathcal L_n  f \rangle_{\mu_{ss}^n} \ge 0.$
Recall the definitions of the operators $D_{x,y}$, $\Phi_x$, $\bar\Phi_x$ introduced in Appendix \ref{app_adjoint} and recall \eqref{eq:genham}.
By performing an integration by parts, we have that $$\langle f,  -\mathcal L^{\mathscr H}_n  f \rangle_{\mu^n_{ss}} = \sum_{x,y\in\Lambda_n} p(y-x) E_{\mu^n_{ss}} \Big[(D_{x,y} f)^2\Big].$$ On the other hand,
\begin{equation}
\begin{split}
&\sum_{x,y \in \Lambda_n} p(y-x) (\Phi^n_{ss}(y)-\Phi^n_{ss}(x)) \langle f, D_{x,y}f \rangle_{\mu^n_{ss}}\\
&= \cfrac{1}{2} \sum_{x,y \in \Lambda_n} p(y-x) (\Phi^n_{ss}(y)-\Phi^n_{ss}(x)) \langle f, (D_{x,y} + D_{x,y}^*) f \rangle_{\mu}\\
&=  \cfrac{1}{2} \sum_{x,y \in \Lambda_n} p(y-x) (\Phi^n_{ss}(y)-\Phi^n_{ss}(x)) \langle f, \Delta_{x,y}  f \rangle_{\mu^n_{ss}}
\end{split}
\end{equation}

Now recall \eqref{eq:S_bulk_refurnished}, from where we get that 
\begin{equation}
\label{eq:S_bulk_refurnished}
\begin{split}
\mathcal{S}^{\text{bulk}} = \mathcal L^\mathscr H_n -\frac 12 (\Phi_\ell-\Phi_{ss}^n(1))\bar\Phi_1-\frac12(\Phi_r-\Phi_{ss}^n(n-1))\bar\Phi_{n-1}.
\end{split}
\end{equation}

This allows us to conclude that 
\begin{equation}
\begin{split}
 \langle f,  -\mathcal L^{\rm{bulk}}  f \rangle_{\mu^n_{ss}}& = \sum_{x,y\in \Lambda_n} p(y-x)  E_{\mu^n_{ss}} \Big[(D_{x,y} f)^2\Big]\\
&+\frac 12 (\Phi_\ell-\Phi_{ss}^n(1))E_{\mu^n_{ss}}[\bar\Phi_1 f^2]+\frac12(\Phi_r-\Phi_{ss}^n(n-1))[\bar\Phi_{n-1} f^2].\end{split}
\end{equation}
Finally recall \eqref{eq:symmetric_lef}. A simple computation shows that \begin{equation}
\begin{split}
&\langle f,  -\mathcal L^\ell  f \rangle_{\mu^n_{ss}} =  E_{\mu^n_{ss}} \left[ (\partial_{\varphi (1)} f)^2  \right]-\frac 12 (\Phi_\ell-\Phi_{ss}^n(1))E_{\mu^n_{ss}}[\bar\Phi_1 f^2]\\&    \langle f,  -\mathcal L^r  f \rangle_{\mu^n_{ss}} =  E_{\mu^n_{ss}} \left[ (\partial_{\varphi (n-1)} f)^2  \right]-\frac12(\Phi_r-\Phi_{ss}^n(n-1))[\bar\Phi_{n-1} f^2]. 
\end{split}
\end{equation}

To conclude, we proved that
\begin{equation}
\label{eq:Dirichletform}
\begin{split}
{\mathfrak D}^n (f) = n^{\gamma}  E_{\mu^n_{ss}} \Big[ (\partial_{\varphi (1)} f)^2 + (\partial_{\varphi (n-1)} f)^2 + \sum_{x,y\in\Lambda_n} p(y-x) (D_{x,y} f)^2 \Big]. 
\end{split}
\end{equation}

The main result of this section reads as follows.
\begin{theorem}[Sector Condition]
\label{thm:sector cond}
The generator $\mathcal L_n$ satisfies the following sector condition
\begin{equation}
    \langle f, -\mathcal L_n g\rangle_{\mu^n_{ss}}^2\leq K_n \;\langle f,-\mathcal L_n f\rangle_{\mu^n_{ss}}\;\langle g,-\mathcal L_n g\rangle_{\mu^n_{ss}}
\end{equation}
where the constant $K_n$ is polynomial in $n$.
\end{theorem}


\begin{proof} 
Recall the expression \eqref{eq:adjointDxy} (resp. \eqref{eq:adjointphix}) of the adjoint of $D_{x,y}$ (resp. $\partial_{\varphi(x)}$). We get
\begin{equation}
\begin{split}
\langle  f, -\mathcal A_ng \rangle_{\mu^n_{ss}} & =
- \frac{n^\gamma}{2}  \sum_{x,y\in\Lambda_n} p(y-x) (\Phi^n_{ss}(y)-\Phi^n_{ss}(x)) \langle  f , D_{x,y} g \rangle_{\mu^n_{ss}} \\
&-{n^\gamma} \sum_{x,y\in\Lambda_n} p(y-x) (\Phi^n_{ss}(y)-\Phi^n_{ss}(x)) \langle  f , {\bar\Phi_x} g \rangle_{\mu^n_{ss}} \\
&+n^\gamma(\Phi_\ell-\Phi^n_{ss}(1)) \langle  f , \partial_{\varphi(1)} g\rangle_{\mu_{ss}^n} +n^\gamma(\Phi_r-\Phi^n_{ss}(n-1)) \langle  f , \partial_{\varphi({n-1})} g \rangle_{\mu_{ss}^n}\\
&- n^\gamma(\Phi_\ell-\Phi^n_{ss}(1)) \langle  f , \bar\Phi_1  g\rangle_{\mu_{ss}^n} -n^\gamma(\Phi_r-\Phi^n_{ss}(n-1)) \langle  f , \bar\Phi_{n-1 } g \rangle_{\mu_{ss}^n}.
\end{split}
\end{equation}
Recall \eqref{eq:discreteprofile} and multiply it by $\bar\Phi_x$ and sum over $x \in \Lambda_n$ to get
\begin{equation}
 \sum_{x,y\in\Lambda_n} p(y-x) (\Phi^n_{ss}(y)-\Phi^n_{ss}(x)){\bar\Phi_x}
\ + \  (\Phi_\ell-\Phi^n_{ss}(1))\bar\Phi_1 + (\Phi_r -\Phi_{ss}^n (n-1)) \bar\Phi_{n-1 } =0.  
\end{equation}
Hence we obtained
\begin{equation}\label{eq:anti-symmetric part SC}
    \begin{split}
\langle  f, -\mathcal A_ng \rangle_{\mu^n_{ss}} &= - \frac{n^\gamma}{2}  \sum_{x,y\in\Lambda_n} p(y-x) (\Phi^n_{ss}(y)-\Phi^n_{ss}(x)) \langle  f , D_{x,y} g \rangle_{\mu^n_{ss}} \\
&+n^\gamma(\Phi_\ell-\Phi^n_{ss}(1)) \langle  f , \partial_{\varphi(1)} g\rangle_{\mu_{ss}^n} +n^\gamma(\Phi_r-\Phi^n_{ss}(n-1)) \langle  f , \partial_{\varphi({n-1})} g \rangle_{\mu_{ss}^n}.
    \end{split}
\end{equation} 
In order to establish a sector condition, we need to find a suitable bound for this quantity. The set of smooth functions with compact support is a common core to $\mathcal L_n$ and $\mathcal L_n^*$. Therefore, in the following estimates, we assume that all functions involved are smooth with compact support. Observe first that we can assume that $f$ and $g$ are mean zero with respect to $\mu^n_{ss}$ since $\mathcal A_n {\bf 1} =0$.

We start with the first term in \eqref{eq:anti-symmetric part SC}. From the Cauchy-Schwarz's inequality we get
\begin{equation}
\label{eq:maduro2}
\begin{split}
\vert \langle  f,-\mathcal A_n g \rangle_{\mu^n_{ss}} \vert 
&\le \frac{n^\gamma}{2}  \sqrt{ {\rm{Var}}_{\mu^n_{ss}} (f) } \  \sum_{x,y\in\Lambda_n} p(y-x)\vert \Phi^n_{ss}(y)-\Phi^n_{ss}(x) \vert  \sqrt{\langle  (D_{x,y} g)^2 \rangle_{\mu^n_{ss}}} \\
& \le \frac{n^\gamma}{2}  \ \sqrt{ {\rm{Var}}_{\mu^n_{ss}} (f)} \  \sqrt{\sum_{x,y\in\Lambda_n} p(y-x) \vert  \Phi^n_{ss}(y)-\Phi^n_{ss}(x) \vert } \ \sqrt{{\mathfrak D}^n (g)}.   
\end{split}
\end{equation}
Lets now show  that there exists  a  constant $C$ such that for any $f \in L^2 (\mu^n_{ss})$ 
\begin{equation}
\label{eq:Poincare}
\begin{split}
{\rm{Var}}_{\mu^n_{ss}} (f) \le C n^2 {\mathfrak D}^n (f).
\end{split}
\end{equation}
We recall that $\mu^n_{ss}$ is given by  $ \mu^n_{ss} (d\varphi) \propto \, e^{-\mathscr H (\varphi)} \, d\varphi$ with the Hamiltonian $\mathscr H$ given by \eqref{eq:Hamiltonian}. The Hamiltonian $\mathscr H$ is strictly elliptic since $\text{Hess}(\mathcal H) \ge {\rm {Id}}$. Indeed, for $y,z\in\Lambda_n$,  
\begin{equation}
    \frac{\partial \mathscr H}{\partial\varphi(y)}(\varphi)=\varphi(y)-\Phi^n_{ss}(y)\quad\text{and}\quad \frac{\partial^2 \mathscr H}{\partial\varphi(y)\partial\varphi(z)}(\varphi)=\delta_{y,z}.
\end{equation}
Hence, $\text{Hess}(\mathscr H)=\text{Id}$. From Theorem \ref{thm:Bakry_Emery},  the stationary state $\mu^n_{ss}$ satisfies the usual Sobolev inequality with constant $1$, and consequently the usual Poincar\'e inequality with constant $1$ (independently of $n$), i.e.,
\begin{equation}
{\rm{Var}}_{\mu^n_{ss}} (f) \le \sum_{x=1}^{n-1}  E_{\mu^n_{ss}} \Big[ (\partial_{\varphi (x)} f)^2 \Big].
\end{equation}

We now estimate the right-hand side of the previous display. Recall   \eqref{eq:Dirichletform}. 
We have that $ E_{\mu^n_{ss}} \Big[ (\partial_{\varphi (1)} f)^2 \Big]  \le n^{-\gamma} {\mathfrak D}^n (f),$ and note that for any $x\in \{2, \ldots, n-1\}$, it holds
\begin{equation}
\begin{split}
 E_{\mu^n_{ss}} \Big[ (\partial_{\varphi (x)} f)^2 \Big]  &\le 2  E_{\mu^n_{ss}} \Big[ (\partial_{\varphi (1)} f)^2 \Big] + 2  E_{\mu^n_{ss}} \Big[ (D_{1,x}  f)^2 \Big]\\
&= 2   E_{\mu^n_{ss}} \Big[ (\partial_{\varphi (1)} f)^2 \Big] + \cfrac{2}{p(x-1)}  p(x-1) E_{\mu^n_{ss}} \Big[ (D_{1,x}  f)^2 \Big] \\
& \lesssim E_{\mu^n_{ss}} \Big[ (\partial_{\varphi (1)} f)^2 \Big] + n^{\gamma+1} n^{-\gamma} {\mathfrak D}^n (f) \lesssim n {{\mathfrak D}^n} (f) \ .
\end{split}
\end{equation}
Then \eqref{eq:Poincare} holds. Therefore, by \eqref{eq:maduro2} 
\begin{equation}
\langle  f, -\mathcal A_n g \rangle_{\mu^n_{ss}}^2  \lesssim n^{2(\gamma+1)} C_n^2\  {\mathfrak D}^n (f)\  {\mathfrak D}^n (g)   
\end{equation}
with $C_n=\sqrt{\sum_{x,y\in\Lambda_n} p(y-x) \vert  \Phi^n_{ss}(y)-\Phi^n_{ss}(x) \vert }$. By \eqref{eq:boundprofileness},  $C_n$ is bounded by a constant times $\sqrt n$. For the boundary terms, by the previous paragraph and the Log-Sobolev's inequality and the Cauchy-Schwarz's inequality, we conclude that
\begin{equation}
    |\langle f,\partial_{\varphi_1}g\rangle_{\mu^n_{ss}}|^2\leq \text{Var}_{\mu^n_{ss}}(f) E_{\mu^n_{ss}}\big[(\partial_{\varphi_1}g)^2\big]\leq Cn^{2-\gamma}{\mathfrak D}^n(f){\mathfrak D}^n(g).
\end{equation}
Hence, the first term on the second line of \eqref{eq:anti-symmetric part SC} squared is bounded by
\begin{equation}
    n^{2\gamma}(\Phi_\ell-\Phi^n_{ss}(1))^2n^{2-\gamma}{\mathfrak D}^n(f){\mathfrak D}^n(g).
\end{equation}
The squared term in the previous display is of order mostly one, so that we get the bound $n^{\gamma+2}{\mathfrak D}^n(f){\mathfrak D}^n(g)$. The argument for the right boundary is completely analogous. Therefore we conclude that 
\begin{equation}
    \langle  f, -\mathcal A_n g \rangle_{\mu^n_{ss}}^2\leq C n^{2(\gamma+1)}C_n^2 {\mathfrak D}^n(f){\mathfrak D}^n(g).
\end{equation}
\end{proof}

\section*{Acknowledgements}
The article was prepared within the framework of the Basic Research Program at HSE University. P.G. thanks Fundação para a Ciência e Tecnologia
FCT/Portugal for financial support through the projects UIDB/04459/2020 and UIDP/04459/2020 and ERC-FCT.
C.B. and J.M. thank the hospitality of Instituto Superior Técnico, Lisbon, for their research visit during the period of July 2025  and March 2026 when parts of this work were developed. 
J.M. also thanks the financial support from  CNPq through the Ph.D. schoolarship.

\subsection*{Data availability}
No datasets were generated or analysed during the current study.

\subsection*{Conflicts of interest}
The authors declare that they have no conflicts of interest that could have influenced the work reported in this paper.

\bibliographystyle{alpha}
\bibliography{bibliography}
\end{document}